\documentclass[12pt]{amsart}
\usepackage[a4paper,hmargin=2cm,vmargin=3cm]{geometry}
\usepackage[T1]{fontenc}
\usepackage[utf8]{inputenc}
\usepackage{color}
\usepackage{amsmath,bm,bbm}
\usepackage{amsthm}
\usepackage{amssymb}
\usepackage{graphicx}
\usepackage[style=alphabetic,maxnames=10,giveninits=true,backend=biber,doi=false,url=false]{biblatex}
\usepackage[shortlabels]{enumitem}
\usepackage{hyperref}
\usepackage{amsaddr}
\usepackage{mathrsfs}
\usepackage[normalem]{ulem}

\usepackage{marginnote}

\definecolor{notefontcolor}{rgb}{0.800781, 0.800781, 0.800781}
\definecolor{grey30}{rgb}{0.7,0.7,0.7}

\hypersetup{unicode=true,
  colorlinks=true, citecolor=blue, linkcolor=blue,
  pdftitle={Fluctuations of the ground state in the spiked SSK model},
  pdfauthor={D. Belius, L. Fröber}}

\DeclareSourcemap{
  \maps[datatype=bibtex]{
    \map{
      \step[fieldset=issn, null]
      }
    }
  }
\numberwithin{equation}{section}

\theoremstyle{plain}
\newtheorem{theorem}{Theorem}[section]
\newtheorem{lemma}[theorem]{Lemma}
\newtheorem{proposition}[theorem]{Proposition}
\newtheorem{corollary}[theorem]{Corollary}

\theoremstyle{remark}
\newtheorem{remark}[theorem]{Remark}

\newcommand{\IN}[1]{s_{}(#1)}

\newcommand{\GN}[2]{s_{{#1,u}}(#2)}
\newcommand{\INk}[2]{s^{(#1)}_{}(#2)}
\newcommand{\HNk}[3]{s^{(#1)}_{{#2}}(#3)}
\newcommand{\GNk}[3]{s^{(#1)}_{{#2,u}}(#3)}
\newcommand{\INd}[1]{s'_{}(#1)}
\newcommand{\INdd}[1]{s''_{}(#1)}
\newcommand{\INddd}[1]{s'''_{}(#1)}

\newcommand\numberthis{\addtocounter{equation}{1}\tag{\theequation}\numberwithin{equation}{section}}
\DeclareMathOperator*{\argmax}{arg\,max}

\title[Fluctuations of the ground state of the spiked SSK model]{Fluctuations of the ground state of the spiked spherical Sherrington-Kirkpatrick model}
\author{David Belius$^{*,**}$, Leon Fr\"ober$^{*}$}
	\email{david.belius@cantab.net}
	\email{leon.froeber@unibas.ch}
	\thanks{* Supported by SNSF grant 176918.}
        \thanks{** Supported by SNSF grant 206148.} 
    \address{Universit\"at Basel, Departement Mathematik und Informatik}

\begin{document}

\maketitle

\begin{abstract}    
    The Sherrington-Kirkpatrick Hamiltonian is a random quadratic function on the high-dimensional sphere. This article studies the ground state (i.e. maximum) of this Hamiltonian with external field, or more generally with a non-linear
    ``spike'' term. We
    compute the level of the maximum to leading order, and under appropriate
    condition its first- and second-order fluctuations. The equivalent results are also derived for the maximum of the model's TAP free energy on the ball.
\end{abstract}

\section{Introduction}\label{section: introduction}

This article studies the maximum of a natural random quadratic optimization problem in $N$ variables over the sphere or ball in $\mathbb{R}^N$, in the presence of a possibly non-linear ``spike'' term. We prove a leading order law of large numbers as $N\to \infty$, and study the fluctuations around the limit. In the context of spin glasses \cite{firstSKmodel1975,mezard1987spin, talagrand2010mean, panchenko2013sherrington} the maximum on the sphere that we study is precisely the ground state of the spherical Sherrington-Kirkpatrick  Hamiltonian \cite{kosterlitz1976spherical}  with external field, or more generally with a non-linear ``spike''. Our result on maximum on the ball applies to the TAP  free energy \cite{thouless1977solution, crisanti1995thouless, belius-kistler} of this Hamiltonian.

The random quadratic optimization problem $\sup_{\sigma \in \mathbb{R}^N:|\sigma|=1}\{ \sigma^{T}J\sigma+\sigma \cdot v \}$ for an $N \times N$ random matrix $J$ and vector $v\in \mathbb{R}^N$
constitutes arguably the most basic yet interesting high-dimensional
random optimization problem and merits special attention. The case where $J$ is a GOE random matrix is representative. The large deviations of this maximum has been studied in \cite{fyodorov2014topology, dembo2015matrix}. A natural generalization is
to replace the linear ``external field'' term $\sigma\cdot v$ with $f\left(\sigma \cdot v\right)$
for some non-linear ``spike'' function $f$ \cite{richard2014statistical, lesieur2017constrained, lelarge2019fundamental, arous2019landscape}. 
The present paper determines the leading order of the maximum for general $f$, and gives a precise description of its fluctuations (i.e. its ``typical deviations''). In particular Theorem \ref{thm: sphere} provides both a law of large numbers that computes the order $N$ asymptotic of the maximum, and under appropriate
assumptions on $f$ also determines first- and second-order subleading
fluctuation terms of order $N^{1/2}$ and $1$ respectively.

Our main motivation comes from mean-field spin glasses, and concerns
the maximum of the TAP free energy, which is a function of the form
$m\mapsto m^{T}Jm+f(m\cdot v)+g(\left|m\right|)$ defined on the unit ball,
for a certain function $g$ that we recall below. Theorem \ref{thm: ball} computes the leading order and fluctuations of the maximum of such a function on the ball, for a general $g$. Below we discuss the spin-glass motivation in more detail.

To formally state our results, define the Sherrington-Kirkpatrick 
\textit{Hamiltonian}
\begin{equation}\label{def:HN}
    H_N(\sigma) = \sqrt{N} \sigma^T J \sigma \text{ for }\sigma \in \mathbb{R}^N
\end{equation}    
where $J$ is an $N\times N$ GOE random matrix, i.e. a symmetric matrix with centered Gaussian entries $J_{i,j}$ mutually independent for $i \le j$, and ${\rm{Var}}( J_{i,j} ) = \frac{1}{2}(1 + \delta_{i=j})$. Let $f : [-1,1] \to \mathbb{R}$ be a real function, $\beta >0$ a constant which we call the \textit{inverse temperature} and 
$v \in \mathbb{R}^N,|v|=1,$ a unit vector giving the direction of the spike.
The \textit{ground state} is the maximum
\begin{equation}\label{def: L_N}
    L_N = \sup_{ |\sigma| =1 } \{ \beta H_N(\sigma) + N f(v \cdot \sigma) \}
\end{equation}
over the unit sphere. Let $\overset{\mathbb{P}}{\to}$ denote convergence in probability,  $\overset{d}{\to}$ convergence in distribution, and $\mathcal{N}(\mu,\sigma^2)$ the Gaussian distribution with mean $\mu$ and variance $\sigma^2$. Our result about the maximum on the sphere is the following.

\begin{theorem}[Maximum on sphere]\label{thm: sphere}
    Let $f \in C^{0}([-1,1])$ and
    \begin{equation}\label{def: B(alpha)}
        \mathcal{B}(\alpha) = f(\alpha) + \beta\sqrt{2(1-\alpha^2)}.
    \end{equation}
    \textbf{(a)} (Leading order) It holds that
    \begin{equation}\label{eq:leading_order_sphere}
        \frac{1}{N} L_N \overset{\mathbb{P}}{\longrightarrow}
        \sup_{\alpha\in [-1,1]}\mathcal{B}(\alpha).
    \end{equation}
    \\
    \textbf{(b)} (Fluctuations)
    If additionally $f\in C^3([-1,1])$ and $\mathcal{B}(\alpha)$ has a unique global maximizer $\hat{\alpha}\neq 0$ with $\mathcal{B}''(\hat{\alpha}) < 0$,
    then there exist a constant $\kappa$ and a matrix $G$ such that
    \begin{equation}\label{eq: theorem1-equation}
        L_N 
        -
            N \mathcal{B}(\hat{\alpha}) 
            - \sqrt{N}\kappa {U}_{N}
            - \left( \kappa\Lambda_N 
                -\frac{1}{2}\begin{pmatrix}{U}_{N}\\
                {U}_{N}'
                \end{pmatrix}^{T}G\begin{pmatrix}{U}_{N}\\
                {U}_{N}'
                \end{pmatrix}
                \right)
                \overset{\mathbb{P}}{\longrightarrow}0,
    \end{equation}
    where $U_N,U_N',\Lambda_N$ are stochastically bounded random variables defined by
    \[    
           {U}_{N}=\sqrt{N}\left(v^{T} G_N v-\frac{{\rm Tr}G_{N}}{N}\right),\ {U}_{N}^{'}=-\sqrt{N}\left(v^{T}G_{N}^{2}v-\frac{{\rm Tr}G_{N}^{2}}{N}\right),\ 
 \Lambda_{N}=\frac{1}{N}\sum_{i=1}^{N}\frac{1}{\hat{l}-\lambda_{i}}-\hat{z},
    \]
    for $\hat{z}=\sqrt{2(1-\hat{\alpha})^{2}}$, $\hat{l}=\frac{2-\hat{\alpha}^{2}}{\hat{z}}$, $G_{N}=\left(\hat{l} \cdot I-\frac{J}{\sqrt{N}}\right)^{-1}$. \\
    The random variables satisfy
    $$
        ({{U}}_{N}, {{U}}_{N}', \Lambda_{N})
        \stackrel{d}{\longrightarrow}
        ({{U}},{{U'}},\Lambda),
    $$
    where 
    \begin{equation}\label{eq: theorem1-variables}
        U \sim \mathcal{N}\left(0, \frac{\hat{z}^4}{\hat{\alpha}^2}\right)
        , \quad
        U' \sim \mathcal{N}\left(0, \frac{\hat{z}^6(2+\hat{\alpha}^2+\hat{\alpha}^4)}{\hat{\alpha}^{10}}\right)
        , \quad 
        \Lambda \sim \mathcal{N}\left(\frac{\hat{z}^3}{2\hat{\alpha}^4}
            , \frac{\hat{z}^4}{\hat{\alpha}^8}\right)
    \end{equation}
    with $({{U}},{{U}}')$ and $\Lambda$ independent and
    \begin{equation}
        \text{Cov}(U,U')  = 
            -\frac{\hat{z}^5 (1+\hat{\alpha}^2)}{\hat{\alpha}^6}.
    \end{equation}
    The constant and matrix are given by
    \[
        \kappa=\frac{\beta\hat{\alpha}^{2}}{\hat{z}^{2}},\quad\quad 
        G=
        \beta\left(
        \frac{8 \beta \hat{\alpha}^2}{\hat{z}^8 \mathcal{B}''(\hat{\alpha})}
        \begin{pmatrix}
            2
            &
            \frac{\hat{\alpha}^4}{\hat{z}}
            \\
            \frac{\hat{\alpha}^4}{\hat{z}}
            &
            \frac{\hat{\alpha}^{8}}{2\hat{z}^2}
        \end{pmatrix}
        +\left(\begin{matrix}\frac{2\hat{\alpha}^{2}}{\hat{z}^{3}} & 0\\
        0 & 0
        \end{matrix}\right)\right).
    \]

    The same holds if $\mathcal{B}(\alpha)$ has a pair of global unique maximizers $\pm\hat{\alpha} \ne 0$ with $\mathcal{B}(\hat{\alpha})=\mathcal{B}(-\hat{\alpha}),\mathcal{B}''(\hat{\alpha})=\mathcal{B}''(-\hat{\alpha})<0$.
\end{theorem}
Part (a) for a linear spike functions $f(x) = hx, h\in\mathbb{R},$ appears in \cite[Lemma 20]{belius-kistler} and is implicit in \cite[Theorem 1.3]{dembo2015matrix}. In that case the maximizer $\hat{\alpha}$ is unique and $\mathcal{B}(\hat{\alpha}) = \sqrt{2\beta^2+h^2}$. The first-order fluctuation result of part (b) in the same linear-spike case,  namely the convergence in law of $N^{-1/2}(L_{N}-\sqrt{2\beta^{2}+h^{2}})$
to a centered Gaussian, is implied also by \cite[Theorem 5]{chen2017parisi} as explained in Remark \ref{rem:CS17_remark}. This corresponds to the first-order fluctuation term $\sqrt{N} \kappa U_N$ in \eqref{eq: theorem1-equation}.

Part (b) of the theorem covers the regime where the fluctuations are determined by the central limit-type behavior of sums over eigenvalues and entries of the spike vector $v$, and for this reason requires $\hat{\alpha} \ne 0$. When $\hat{\alpha}=0$ the fluctuations should instead be determined by the fluctuations of the extreme eigenvalues of $J$ (indeed for $f=0$ the maximum is exactly the largest eigenvalue, which has non-Gaussian fluctuations \cite{tracey1996orthogonal}).

In Section \ref{section: example leading order} and \ref{section: example fluctuations} we give more explicit formulas for leading order and fluctuations for monomial spike functions $f$, and for these determine critical inverse temperatures $\beta$ where the behavior of the ground state changes.

\bigskip
Our second main results concerns the fluctuations of the maximum on the ball of combinations of $H_N$ with a spike and a deterministic radial function. For functions $f: [-1,1] \to \mathbb{R}$ and $g: [0,1] \to \mathbb{R}$ define
\begin{equation}\label{def: L tilde}
    \tilde{L}_N = \sup_{m \in B_N(\mathcal{R})} \{ \beta H_N(m) + N f(v\cdot m) + N g( |m|) \},
\end{equation} 
where $B_N(\mathcal{R}) = \{m\in\mathbb{R}^N: |m|\in\mathcal{R}\}$ and $\mathcal{R}\subset [0,1]$.
The prototypical example is the maximum of the TAP free energy, where the function $g$ takes a particular form and the maximum is taken only over $m$ with $|m|^2$ in a certain range, which is why we include the set $\mathcal{R}$ in the formulation (see the discussion after the theorem).

\begin{theorem}[Maximum on ball]\label{thm: ball}
    For $f \in C^0([-1,1])$, $\mathcal{R} \subset [0,1]$ closed and $g \in C^0(\mathcal{R})$ let
    \begin{equation}\label{def: tilde B_r}
        \tilde{\mathcal{B}}(\alpha,r) = f(r \alpha) + g(r) + \beta r^2 \sqrt{2(1-\alpha^2)}.
    \end{equation}  
    \textbf{(a)}  (Leading order)
    It holds that
    \begin{equation}\label{eq: L tilde leading order}
        \frac{1}{N}\tilde{L}_N \overset{\mathbb{P}}{\longrightarrow}
        \sup_{r\in\mathcal{R},\alpha\in[-1,1]} \tilde{\mathcal{B}}(\alpha,r).       
    \end{equation}
    \textbf{(b)}  (Fluctuations)
    If additionally $f \in C^3([-1,1]),g \in C^3(\mathcal{R})$ and 
    $\tilde{\mathcal{B}}(\alpha,r)$ has a unique global maximizer $(\hat{\alpha},\hat{r})$ in the interior of $\mathcal{R}\times[-1,1]$ with $\hat{\alpha} \neq 0,\hat{r} \neq 0$, and the Hessian matrix $\nabla^2\tilde{\mathcal{B}}(\hat{\alpha},\hat{r})$ is negative definite, then there is a matrix $\tilde{G}$ such that
    \begin{equation}\label{eq: ball result equation}
        \tilde{L}_N -
        \tilde{\mathcal{B}}(\hat{\alpha},\hat{r}) 
        -
        \sqrt{N} \kappa {{U}}_{N}
        - \left( \kappa \Lambda_N 
        - \frac{1}{2} 
            \begin{pmatrix}U_{N}\\
                U_{N}^{'}
                \end{pmatrix}^{T}\tilde{G}\begin{pmatrix}U_{N}\\
                U_{N}^{'}
                \end{pmatrix}
        \right)
       \overset{\mathbb{P}}{\longrightarrow} 0,
    \end{equation} 
    where $U_N, U'_N, \Lambda_N, \kappa$ are as in Theorem \ref{thm: sphere}.
    The matrix is given in terms of $\hat{z}=\sqrt{2(1-\hat{\alpha}^{2})}$ by
    \begin{align*}
        \tilde{G}
        &=
        K^{T}\left(\nabla^{2}\mathcal{B}(\hat{\alpha},\hat{r})\right)^{-1}K
        +
        \begin{pmatrix}
            2\beta\frac{\hat{r}^{2}\hat{\alpha}^{2}}{\hat{z}^{3}}
            & 0 \\ 0 & 0
        \end{pmatrix}
        \ \text{ where }\ 
        K = 
        \frac{2\beta\hat{r}\hat{\alpha}}{\hat{z}^2}
        \begin{pmatrix}
            \frac{2\hat{r}}{\hat{z}^2}
            &  \frac{\hat{r}\hat{\alpha}^4}{\hat{z}^3}
            \\ \hat{\alpha}
            & 0
        \end{pmatrix}
        .
    \end{align*}
    
    The same holds if $\tilde{\mathcal{B}}(\alpha,r)$ has a pair of global unique maximizers $(\pm\hat{\alpha},\hat{r})$ in the interior of $\mathcal{R}\times[-1,1]$ with $\hat{\alpha} \ne 0$, $\tilde{\mathcal{B}}(\hat{\alpha},\hat{r})=\tilde{\mathcal{B}}(-\hat{\alpha},\hat{r})$ and $\nabla^2\tilde{\mathcal{B}}(\hat{\alpha},\hat{r})=\nabla^2\tilde{\mathcal{B}}(-\hat{\alpha},\hat{r})$ negative-definite.
\end{theorem}

In the Thouless-Andersson-Palmer (TAP) \cite{thouless1977solution} approach to spin glasses one aims to extract important information about spin glass models from their {\emph{TAP free energy}}, which is a random
function arising from the Hamiltonian $H_{N}$ of the model. For the
spiked spherical Sherrington-Kirkpatrick model of this article it
is given by \cite{crisanti1995thouless,belius-kistler}
\begin{equation}\label{def:FTAP}
    F_{{\rm TAP}}(m)=\beta H_{N}(m)+Nf(v\cdot m)+\frac{N}{2}\log(1-|m|^{2}) + \frac{N}{2} \beta^{2}(1-|m|^{2})^{2}, \ \left|m\right|<1.
\end{equation}
Only $m$ satisfying certain conditions are believed to be ``relevant''
\cite{thouless1977solution, plefka1, plefka1982convergence, subag2018landscapes, belius-kistler}. For the spherical Sherrington-Kirkpatrick model the only needed condition is {\emph{Plefka's condition}}, requiring that $\sqrt{2}\beta(1-|m|^{2})\le1$. \cite{belius-kistler}. The maximal TAP free energy over $m$
that satisfy Plefka's condition is of the form \eqref{def: L tilde} with $g(r)=\frac{1}{2}\left(\log(1-r^2) + \beta^{2}(1-r^2)^{2}\right)$
and $\mathcal{R}=\{r:r^2\ge1-\frac{1}{\sqrt{2}\beta}\}$. In Sections \ref{section: example leading order} and \ref{section: example fluctuations} we determine more concretely for this $g$ and monomial $f$  when the conditions of Theorem \ref{thm: sphere} and \ref{thm: ball} are satisfied and what the resulting formulas for leading order and fluctuations are.

\subsection{Fluctuations and the TAP approach}

The SK model and its variants consist of a high-dimensional spin space such the
sphere $\{ \sigma \in \mathbb{R}^N: |\sigma|=1 \}$ and a random energy such as $\beta H_{N}(\sigma)+Nf(\sigma\cdot v)$
associated to each spin configuration vector $\sigma$, where $H_{N}(\sigma)$
is a high-dimensional Gaussian field of which $H_{N}(\sigma)$ from \eqref{def:HN} is
a special case. From this energy one constructs the Gibbs measure,
which in the case of a spherical spin space is the probability measure with density proportional to the Gibbs factor $\exp(\beta H_{N}(\sigma)+Nf(\sigma\cdot v)$) with respect to the uniform measure on the sphere. The normalizing factor of the measure is known as the partition function and usually denoted by $Z_{N}$. The vector $\sigma$ sampled according to the Gibbs measure models
the spins of exotic magnet materials, or other complex phenomena in related models \cite{mezard1987spin, mezard2009information}. The ultimate goal of the area is to describe the behavior of $\sigma$ sampled according to the Gibbs measure.

For the general class of mixed $p$-spin Hamiltonians $H_{N}$ \cite{derridaRandomEnergyModelLimit1980,grossSimplestSpinGlass1984,talagrand2000multiple, auffinger2013complexity} this
is a formidable task that is far from being accomplished. In the general
case the ``geometry'' of the random landscape $H_{N}$ is extraordinarily
complex \cite{fyodorovHighDimensionalRandomFields2013,auffinger2013complexity, auffinger2013random, subag2017complexity}, and this is expected to be reflected in the behavior of the
Gibbs measure. The Sherrington-Kirkpatrick Hamiltonian \eqref{def:HN} is the special case of a
{\emph{$2$-spin}} Hamiltonian, which when combined with a spherical spin space
has significantly simpler behavior, and is much easier to study due
to the spherical symmetry and quadratic nature of the Hamiltonian
allowing many explicit calculations that are impossible in general. As such the $2$-spin setting provides a valuable testing ground
for new ideas and techniques. The motivation for this paper is to
use the $2$-spin spherical Hamiltonian as a starting point to explore fluctuations in spin
glasses via a TAP approach.

A first step in understanding the Gibbs measure is computing the \emph{free
energy} which is the limit of $\frac{1}{N}\log Z_{N}$ as $N\to\infty$,
i.e. the rate of exponential growth of the partition function. Knowledge
of the free energy morally speaking corresponds to knowledge of which
regions of the spin space have probability at least $e^{-o(N)}$ under
the Gibbs measure, rather than exponentially small probability. Finer
estimates for the free energy, such as lower order corrections and fluctuations, morally
correspond to finer knowledge of the Gibbs measure. There are several approaches to computing the free energy \cite{parisi1980sequence,guerra, ASS, tsphere, talagrand, contucci2013dmitry, panchenko2014parisi, CASS}. In the TAP approach one expects that the free energy is roughly speaking given by the maximum of the TAP free energy $F_{\rm TAP}(m)$ of the model. The final term of $F_{\rm TAP}(m)$ is called the ``Onsager term'', and the $F_{\rm TAP}(m)$ of general  mixed $p$-spin spherical models coincides with \eqref{def:FTAP} but with a more general Onsager term. The TAP approach for general models is under active investigation \cite{bolthausen2014iterative,bolthausenMoritaTypeProof2018,brenneckeNoteReplicaSymmetric2021,subag2017geometry,chenGeneralizedTAPFree2022,subag2021free,BeliusUpperBound} and the correspondence between free energy and maximal TAP free energy is proven mathematically rigorously without appealing to powerful machinery like the Parisi formula only in a few cases \cite{subag2021free,belius-kistler,tapvectorspin}. One of these is the spherical $2$-spin case of this paper, where the free energy was computed completely within a TAP approach in \cite{belius-kistler}.

From the point of view of the TAP approach the fluctuations of the free energy should arise on the one hand from the fluctuations of the maximum of $F_{{\rm TAP}}$, and on the other hand from the fluctuations of certain ``local'' integrals (over ``slices'' in the terminology of \cite{belius-kistler,tapvectorspin,BeliusUpperBound} and over ``bands'' in the terminology of \cite{subag2017geometry,subag2018landscapes,chenGeneralizedTAPFree2022}; the Onsager term of $F_{\rm{TAP}}$ describes the leading order behavior of these integrals). In this article we completely determine the former kind of fluctuations for the spherical $2$-spin model, to the highest degree of precision that is plausibly relevant for the study of the fluctuations of the free energy and Gibbs measure. The analysis of the latter type of fluctuations, and consequences for the fluctuations of the free energy, are left to future work.

See for instance \cite{aizenman1987some, bovier2002fluctuations, chatterjee2009disorder, baik-lee, chen2017parisi, subag2017extremal,baikSphericalSpinGlass2020,landonFreeEnergyFluctuations2020,landonFluctuations2spinSSK2020,banerjee2021fluctuations,bovier2022fluctuations} for work on fluctuations in spin glasses from a non-TAP point of view.

\subsection{Sketch of proof}\label{section: Sketch of proof}
In this subsection we give a brief sketch of our arguments. To prove Theorem \ref{thm: sphere} we diagonalize the matrix $J$ and obtain that
\begin{equation}\label{eq:sketch_first_step}
    \frac{1}{N}L_{N}=\sup_{|\sigma|=1}\left\{ \beta\frac{1}{N}H_{N}(\sigma)+f(v\cdot\sigma)\right\} \stackrel{d}{=}\sup_{|\sigma|=1}\left\{ \beta\sum_{i=1}^{N}\lambda_{i}\sigma_{i}^{2}+f\left(\sum_{i=1}^{N}u_{i}\sigma_{i}\right)\right\} ,
\end{equation}
where $\lambda_1<\ldots<\lambda_N$ are the eigenvalues of $\frac{1}{\sqrt{N}}J$
and $u$ is the spike vector $v$ written in the diagonal
basis. By the orthogonal invariance of $J$ the vector $u$ is uniform
on the sphere and independent of the $\lambda_i$. Next we decompose the maximization in \eqref{eq:sketch_first_step} according to the value of $\sum_{i=1}^{N}u_{i}\sigma_{i}$
to obtain 
\begin{equation}
\frac{1}{N}L_{N}\stackrel{d}{=}\sup_{\alpha\in[-1,1]}\left\{ f\left(\alpha\right)+\beta\sup_{\substack{|\sigma|=1\\
\sigma\cdot u=\alpha
}
}\sum_{i=1}^{N}\lambda_{i}\sigma_{i}^{2}\right\} .\label{eq:LN_fix_overlap_intro}
\end{equation}
In the proof of Theorem \ref{thm: ball} we use the similar identity \eqref{eq: fix overlap ball} for $\frac{1}{N}\tilde{L}_{N}$
where the outer supremum is also over $r$.

We then solve the constrained optimization problem in (\ref{eq:LN_fix_overlap_intro})
using Lagrange multipliers, and obtain the
identity
\begin{equation}\label{eq:inf_l_s_lambda_u_intro}
    \sup_{\substack{|\sigma|=1\\
    \sigma\cdot u=\alpha
    }
    }\sum_{i=1}^{N}\lambda_{i}\sigma_{i}^{2}=\inf_{l>\lambda_{N}}\left\{l-\frac{\alpha^{2}}{s_{\lambda,u}(l)}\right\}\quad\text{ for }\quad s_{\lambda,u}(l)=\sum_{i=1}^{N}\frac{u_{i}^{2}}{\lambda_{i}-l},
\end{equation}
provided $\left|\alpha\right|\ge |u_{N}|$ over $l  > \lambda_N$. This reduces the high-dimensional optimization over $\sigma\in\mathbb{R}^N$ to a low-dimensional one. We recognize the random
function $s_{\lambda,u}(l)$ as the Stieltjes transform of the empirical
spectral distribution of $J$ weighted by $u_{i}^{2}$. It is easy
to see that it converges to the Stieltjes transform of the semi-circle
law $s(l)$. In our normalization it is given by $s(l)=l-\sqrt{l^{2}-2}$,
and also $\lambda_{N}\to\sqrt{2}$ in probability. We thus obtain
from (\ref{eq:inf_l_s_lambda_u_intro}) a limiting optimization problem
which is explicitly solvable:
\begin{equation}\label{eq:inf_l_s_intro}
    \inf_{l>\sqrt{2}}\left\{l-\frac{\alpha^{2}}{s(l)}\right\}=\sqrt{2(1-\alpha^{2})},
\end{equation}
cf. \eqref{def: B(alpha)}. To prove the leading order
results Theorem \ref{thm: sphere} (a) and Theorem \ref{thm: ball} (a) it suffices to
approximate the infimum in (\ref{eq:inf_l_s_lambda_u_intro}) by that
in (\ref{eq:inf_l_s_intro}). For this purpose we obtain in Section \ref{section: leading order} sufficiently
uniform estimates for the convergence of $s_{\lambda,u}(l)$ to $s(l)$,
and combine these with a simple ad-hoc argument for $\left|\alpha\right|\le |u_{N}|$
to prove Theorem \ref{thm: sphere} (a) and Theorem \ref{thm: ball} (a).

For the fluctuation result Theorem \ref{thm: sphere} (b) the assumption that $\hat{\alpha}\ne0$
makes the identity (\ref{eq:inf_l_s_lambda_u_intro}) hold in a neighborhood $[ \hat{\alpha}-\varepsilon,\hat{\alpha}+\varepsilon]$
of the unique maximizer $\hat{\alpha}$ with high probability, and using this the maximum can be written exactly as the minimax
\[
\frac{1}{N}L_{N} =\sup_{\alpha\in[ \hat{\alpha}-\varepsilon,\hat{\alpha}+\varepsilon]}\inf_{l>\lambda_{N}}h(\alpha,l,s_{\lambda,u}(l))\quad\text{ for }\quad h(\alpha,l,g)=f(\alpha)+\beta\left(l-\frac{\alpha^{2}}{g}\right).
\]
A similar function $h((\alpha,r),l,g)$ gives a similar ``high
probability'' identity for $\frac{1}{N}\tilde{L}_{N}$ (see \eqref{eq: high probability equality tilde L_N}). Therefore
both Theorem \ref{thm: sphere} (b) and Theorem \ref{thm: ball} (b) can be proved by
studying fluctuations of
\begin{equation}
\sup_{y\in\mathcal{Y}}\inf_{l\in\mathcal{L}}h(y,l,s_{\lambda,u}(l)),\label{eq:gen_minimax_intro}
\end{equation}
for a general function $h(y,l,g)$ where $y\in\mathcal{Y}\subset\mathbb{R}^{n},n\ge1$
and $\mathcal{L}\subset(\sqrt{2},\infty$), under the assumption that
the limiting minimax $\sup_{y\in\mathcal{Y}}\inf_{l\in\mathcal{L}}h(y,l,s(l))$
has a unique optimizer $(\hat{y},\hat{l})$. In Section \ref{section: leading order} we study the fluctuations of $s_{\lambda,u}(l)$ around
$s(l)$ using a combination of central limit theorems for sums over
eigenvalues and the entries of the spike vector $v$. We then expand $h(y,l,s_{\lambda,u}(l))$
quadratically in these fluctuations and in $y,l$, around the point
$\hat{y},\hat{l},s(\hat{l})$. The first- and second-order fluctuations
of (\ref{eq:gen_minimax_intro}) are obtained by solving the minimax
optimization for this approximating quadratic, leading to the proof of \eqref{eq: theorem1-equation} and \eqref{eq: ball result equation}.

\subsection{Organization}
In the preliminary Section \ref{section: preliminaries} we recall some useful results about the GOE random matrix and its eigenvalues. In Section \ref{section: problem reduction} we use Lagrange multipliers to reduce the optimizations over $\sigma$ in $L_N$ and $\tilde{L}_N$ to low-dimensional optimization as described in the sketch above. In Section \ref{section: leading order} we prove uniform leading order estimates for the convergence of $s_{\lambda,u}(l)$ to $s(l)$, and deduce from these the leading order estimates Theorem \ref{thm: sphere} (a) and Theorem \ref{thm: ball} (a). Then in Section~\ref{section: example leading order} we provide some concrete examples of $f$ and $g$ to which the leading order results apply. In Section \ref{section: subleading order fluctuations} we study the fluctuations of $s_{\lambda,u}$, and use this and the quadratic expansion described in the sketch to prove the fluctuation results Theorem \ref{thm: sphere} (b) and Theorem \ref{thm: ball} (b). Finally in Section~\ref{section: example fluctuations} we apply these to study the fluctuation for the examples of Section~\ref{section: example leading order}.

\subsection{Notation}
We use the following notations, in addition to those already introduced before Theorem \ref{thm: sphere}. The unit sphere is denoted $\mathcal{S}_{N-1}=\{ \sigma \in \mathbb{R}^N:|\sigma|=1\}$. Furthermore we write $O_{\mathbb{P}}$ and $o_{\mathbb{P}}$ for probabilistic versions of the standard  notation for the order of quantities as $N\to\infty$. More precisely we write $X_N={{O}}_{\mathbb{P}}(T(N))$ if $X_N/T(N)$ is stochastically bounded, i.e. if
\begin{equation}\label{eq: big O}
    \lim_{x\to\infty} \limsup_{N\to\infty} \mathbb{P}\left( \frac{|X_N|}{T(N)} \ge x \right) = 0,
\end{equation}
and $X_N = o_{\mathbb{P}}(T(N))$ if
\begin{equation}\label{eq: small o}
    \frac{|X_N|}{T(N)} \overset{\mathbb{P}}{\longrightarrow} 0.
\end{equation}

\section{Random matrix preliminaries}\label{section: preliminaries}
In this section we recall some standard results about the eigenvalues of the GOE. We denote the semi-circle law on $[-\sqrt{2},\sqrt{2}]$ by
\begin{equation}\label{eq:semicircle_law_def}
    \mu_{\text{sc}}(dx) = \frac{\sqrt{2 - x^2}}{\pi} dx.
\end{equation}

Let $\theta_{1/N},...,\theta_{N/N} \in [-\sqrt{2},\sqrt{2}]$ be given by
\begin{equation}\label{eq: classical locations}
    \int_{-\sqrt{2}}^{\theta_{k/N}} \mu_{\text{sc}}(dx) = \frac{k}{N},
\end{equation}
which are sometimes called the \textit{classical locations} of the eigenvalues of $J$.
From e.g. \cite[Theorem 2.2]{ev_rigidity} we know that the eigenvalues concentrate around these, i.e.:

\begin{lemma}\label{lem: rigidity of eigenvalues}
    For any $\varepsilon > 0$ and all $k \in\{1,...,N\}$
    $$
        |\lambda_k - \theta_{k/N}| \le N^{-\frac{2}{3}+\varepsilon} 
            \min\left\{ k^{-\frac{1}{3}}, (N-k)^{-\frac{1}{3}} \right\}
    $$
    with probability tending to one as $N\rightarrow\infty$.
\end{lemma}
In particular
\begin{equation}\label{eq: extremal evals}
    \lambda_N \overset{\mathbb{P}}{\to} \sqrt{2} \text{ and } \lambda_1 \overset{\mathbb{P}}{\to} -\sqrt{2}.
\end{equation}
It is elementary to estimate sums of the classical locations with integrals over the semi-circle law. The next lemma records this.
\begin{lemma}\label{lem: sums}
    For all $w\in {C^{1}}\left([-\sqrt{2}-\varepsilon,\sqrt{2}+\varepsilon]\right)$
    it holds that 
    \begin{equation}\label{eq: sum class loc int}
        \left|
        \frac{1}{N}\sum_{i=1}^{N}w\left(\theta_{i/N}\right)
            -
        \int_{-\sqrt{2}}^{\sqrt{2}}w\left(x\right)\mu_{sc}\left(dx\right)
        \right| 
        \le \frac{2\sqrt{2}|w'|_\infty}{N},
    \end{equation}

\end{lemma}
\begin{proof}
    It follows from \eqref{eq: classical locations} that
    \begin{align*}
        \int_{-\sqrt{2}}^{\sqrt{2}}w(x)\mu_{sc}(dx)
        & =\sum_{i=1}^{N}\int_{\theta_{(i-1)/N}}^{\theta_{i/N}}w(x)\mu_{sc}(dx)
        \\
        & = \sum_{i=1}^{N}w(\theta_{i/N})\frac{1}{N}
        + \xi(w)\frac{1}{N}\sum_{i=1}^{N}(\theta_{i/N}-\theta_{(i-1)/N}),
    \end{align*}
    where $\xi(w) \in [-|w'|_{\infty},|w'|_{\infty}]$.
\end{proof}

The next lemma is concerned with fluctuations of sums over the eigenvalues.
\begin{lemma}\label{lem: ev fluctuation}
    If $\varepsilon > 0$ and $w \in C^1([-\sqrt{2}-\varepsilon, \sqrt{2}+\varepsilon])$
        $$
            \sum_{i=1}^N w(\lambda_i) - N \int_{-\sqrt{2}}^{\sqrt{2}} w(x) \mu_{\text{sc}}(dx)
            \stackrel{d}{\longrightarrow} 
            \mathcal{N}\left(m(w), v(w) \right),
        $$
    where
        \begin{align*}
            m(w) = & 
                    \frac{w(\sqrt{2}) + w(-\sqrt{2})}{4} -
                    \frac{1}{2\pi} \int_{-\sqrt{2}}^{\sqrt{2}} 
                    w(x) \frac{1}{\sqrt{2-x^2}} dx
            \\
            v(w) = & 
                    \frac{1}{2\pi^2}
                    \int_{-\sqrt{2}}^{\sqrt{2}}\int_{-\sqrt{2}}^{\sqrt{2}}
                    \left(\frac{w(x)-w(y)}{x-y}\right)^2
                    \frac{2 - x y}{\sqrt{2-x^2}\sqrt{2-y^2}}
                    dx dy.
        \end{align*}
\end{lemma}
\begin{proof}
    Let 
    $\tilde{\lambda}_1,...,\tilde{\lambda}_N \stackrel{d}{=} \sqrt{2}\lambda_1,...,\sqrt{2}\lambda_N$ and 
    $\tilde{\mu}_{\text{sc}}$ the measure of the semi-circle law on the interval $[-2,2]$.
    By \cite[Theorem 1.1]{bai-yao} with $\kappa = \sigma^2 = 2$ and $\beta = 0$ it holds that for differentiable $w$
    $$
        \sum_{i=1}^N w(\tilde{\lambda}_i) - N \int_{-2}^{2} w(x) \tilde{\mu}_{\text{sc}}(dx)
        \stackrel{d}{\longrightarrow} 
            \mathcal{N}\left(\tilde{m}(w), \tilde{v}(w) \right)
    $$
    with expectation
    $$
        \tilde{m}(w) =
        \frac{w(2) + w(-2)}{4}
        - \frac{1}{2\pi} \int_{-1}^1 w(2t) \frac{1}{\sqrt{1-t^2}} dt
    $$
    and variance
    $$
        \tilde{v}(w) =
                    \frac{1}{2\pi^2}
                    \int_{-2}^{2}\int_{-2}^{2}
                    w'(s)w'(t)
                    \log\left(\frac{4-ts + \sqrt{4-s^2}\sqrt{4-t^2}}{4-ts - \sqrt{4-s^2}\sqrt{4-t^2}}\right)
                    ds dt
    $$
    By a change of variables we immediately get $m(w)$ from $\tilde{m}(w(\tfrac{1}{\sqrt{2}} \cdot ))$. For the variance we can show that the expressions match by using integration by parts twice. Note that
    \begin{align*}
        \frac{\partial}{\partial y}
        \log\left(\frac{4-xy + \sqrt{4-x^2}\sqrt{4-y^2}}{4-xy - \sqrt{4-x^2}\sqrt{4-y^2}}\right)
            =
            2 \frac{\sqrt{4-x^2}}{\sqrt{4-y^2}(x-y)}
    \end{align*}
    and
    \begin{align*}
        \frac{\partial}{\partial x}
        \frac{\sqrt{4-x^2}}{\sqrt{4-y^2}(x-y)}
            =
            -\frac{4 - xy}{\sqrt{4-x^2}\sqrt{4-y^2}(x-y)^2},
    \end{align*}
    which gives
    \begin{align*}
        & \int_{-2}^{2}\int_{-2}^{2}
                    \left(w(x)-w(y)\right)^2
                    \tfrac{4 - x y}{(x-y)^2\sqrt{4-x^2}\sqrt{4-y^2}}
                    dx dy
        \\ = &
        \int_{-2}^{2}
        \left(
            \left[
            \left(w(x)-w(y)\right)^2
            \tfrac{-\sqrt{4-x^2}}{\sqrt{4-y^2} (y-x)}
            \right]_{x=-2}^{2}
                -
            \int_{-2}^{2}
                    2 \left(w(x)-w(y)\right) w'(x) 
                    \tfrac{-\sqrt{4-x^2}}{\sqrt{4-y^2} (y-x)}
                    dx 
        \right) dy
        \\ = &
        \int_{-2}^{2}
        \bigg(
            \left[
            \left(w(x)-w(y)\right) w'(x)
            \log\left(\tfrac{4-xy + \sqrt{4-x^2}\sqrt{4-y^2}}{4-xy - \sqrt{4-x^2}\sqrt{4-y^2}}\right)
            \right]_{y=-2}^{2}
        \\ & {\color{white}++++++++}
                +
            \int_{-2}^{2}
                    w'(y) w'(x) 
                    \log\left(\tfrac{4-xy + \sqrt{4-x^2}\sqrt{4-y^2}}{4-xy - \sqrt{4-x^2}\sqrt{4-y^2}}\bigg)
                    dy 
        \right) dx
        \\ = &
        \int_{-2}^{2}
        \int_{-2}^{2}
                    w'(y) w'(x) 
                    \log\left(\tfrac{4-xy + \sqrt{4-x^2}\sqrt{4-y^2}}{4-xy - \sqrt{4-x^2}\sqrt{4-y^2}}\right)
                    dy dx.
    \end{align*}
    By a change of variables we thus get the expression $v(w)$ from $v(w(\tfrac{1}{\sqrt{2}}\cdot))$.
\end{proof}

\section{Reduction to a low-dimensional optimization}\label{section: problem reduction}
In this section we start the proof of Theorem \ref{thm: sphere} and Theorem \ref{thm: ball}
by applying the method of Lagrange multipliers to the original high-dimensional optimization problem and as a result reduce it to a low-dimensional optimization problem.

Recall from \eqref{def: L_N} that
\begin{equation}
    L_N = \sup_{ |\sigma| =1 } \{ \beta H_N(\sigma) + N f(v \cdot \sigma) \}
\end{equation}
where $v$ is a fixed unit vector. We have
\begin{equation}\label{eq: LN in diagonalizing basis}
    \frac{1}{N} L_N =
        \sup_{|\sigma| = 1}\left\{
            \beta \sum_{i=1}^N \lambda_i \sigma_i^2
            + f\left(\sum_{i=1}^N \sigma_i u_i \right)
        \right\},
\end{equation}
where $\lambda_1 \le \lambda_2 \le ... \le \lambda_N$ are the eigenvectors of $\tfrac{1}{\sqrt{N}} J$ and $u = (u_1,...,u_N)$ is $v$ in the diagonalizing basis of $J$. Note that $u$ is a random unit vector uniform on the sphere, independent of $\lambda_1,...,\lambda_N$. We can rewrite \eqref{eq: LN in diagonalizing basis} as
\begin{equation}\label{eq: fix overlap}
    \frac{1}{N} L_N =
    \sup_{\alpha \in [-1,1]} \left\{
        f(\alpha) + \beta
        \sup_{\substack{|\sigma| = 1 \\ \sigma\cdot u = \alpha}} \sum_{i=1}^N \lambda_i \sigma_i^2
    \right\}.
\end{equation}
Similarly, using the substitution $m = r \sigma$ with $r=|m|$ for $|\sigma|=1$ in \eqref{def: L tilde},
\begin{equation}\label{eq: fix overlap ball}
    \frac{1}{N} \tilde{L}_N =
    \sup_{r \in \mathcal{R}, \alpha \in [-1,1]} \left\{
        f(\alpha r ) + g(r) + \beta r^2
        \sup_{\substack{|\sigma| = 1 \\ \sigma\cdot u = \alpha}} \sum_{i=1}^N \lambda_i \sigma_i^2
    \right\}.
\end{equation}

The next lemma will in turn rewrite the inner supremum of \eqref{eq: fix overlap}, \eqref{eq: fix overlap ball} in terms of the Stieltjes transform of the weighted empirical spectral measure
\begin{equation}\label{eq:mu_lambda_u_def}
    \mu_{\lambda,u} = \sum_{i=1}^N u_i \delta_{\lambda_{i}}.
\end{equation}
Recall that the Stieltjes transform of a measure $\mu$ on $\mathbb{R}$ is given by 
\begin{equation}\label{eq:stieltjes_def}
    s_\mu(l)=\int_{\mathbb{R}} \frac{1}{l-\lambda} \mu(d\lambda),
\end{equation}
for $l$ outside the support of $\mu$, so that
\begin{equation}\label{def: GNk}
    s_{\mu_{\lambda,u}}(l)  = \sum_{i=1}^N \frac{u_i^2}{l-\lambda_i}.
\end{equation}
In the interest of compact notation we drop the $\mu$ and write
\begin{equation}\label{eq:s_lambda_u_def}
    s_{\lambda,u} = s_{\mu_{\lambda,u}}.
\end{equation}
We can now formulate our result on the inner optimization in \eqref{eq: fix overlap}, which is an exact identity if $|\alpha| \ge |u_N|$ and a bound that is sufficient for our purposes if $|\alpha| < |u_N|$. This and all further results in this section hold deterministically for any $u \in \mathcal{S}_{N-1}$ with $u_1^2,...,u_N^2 \in (0,1)$ and $-\infty <\lambda_1 < ... < \lambda_N <\infty $.
\begin{lemma}\label{lem: simplify}
    For any $\lambda_1 < ... < \lambda_N$ and $u_1,...,u_N \in (-1,1)\setminus\{0\}$ with $\sum_{i=1}^N u_i^2 = 1$ 
    it holds that if $1>\left|\alpha\right|\ge\left|u_{N}\right|$ then
    \begin{equation}\label{eq: simplify main case}
        \sup_{{\substack{|\sigma|=1 \\\sigma\cdot u=\alpha}}} \sum_{i=1}^{N}\lambda_{i}\sigma_{i}^{2} =\inf_{l>\lambda_{N}}\left\{ l-\frac{\alpha^{2}}{  \GN{\lambda}{l} } \right\}.
    \end{equation}
    If $\alpha\in [-|u_{N}|,|u_{N}|]$ then
    \begin{equation}\label{eq: simplify alpha small}
        \lambda_{N}-\frac{2 u_{N}^{2} }{\sqrt{1-u_{N}^{2}}}\left(\lambda_{N}-\lambda_{1}\right)       
        \le\sup_{{\substack{|\sigma|=1 \\ \sigma\cdot u=\alpha}}} \sum_{i=1}^{N}\lambda_{i}\sigma_{i}^{2} \le\lambda_{N}.
    \end{equation}
\end{lemma}

In the proof and later we will consider the function
$\varphi_N: [\lambda_N,\infty) \mapsto \mathbb{R}^+$ given by
\begin{equation}\label{eq: phi def}
    \varphi_N(l)
    =
    -\frac{1}{ \GN{\lambda}{l}} 
    \overset{\eqref{def: GNk}}{=}
    - \left(\sum_{i=1}^N \frac{u_i^2}{l-\lambda_i}\right)^{-1}
    \text{ for }l>\lambda_N,
\end{equation}
which satisfies
\begin{equation}\label{eq: varphi limit}
    \varphi_N(l) \to 0 \text{ as }l\downarrow \lambda_N
\end{equation}
when $u_N\ne0$, so that defining $\varphi_N(\lambda_N)=0$ makes $\varphi_N$ a continuous function. Note that for $l > \lambda_N$
\begin{equation}\label{eq: phi derivative in l}
    \varphi_N'(l)
        =
        \frac{\GNk{1}{\lambda}{l}}{\GN{\lambda}{l}^2}
        \overset{\eqref{def: GNk}}{=}
        -\frac{\sum_{i=1}^N \frac{u_i^2}{(l-\lambda_i)^2}}{\left(\sum_{i=1}^N \frac{u_i^2}{l-\lambda_i}\right)^2}
        ,
\end{equation}
so if $u_N \ne 0$
\begin{equation}\label{eq: phi derivative in l limit}
    \varphi_N'(l) \to -\frac{1}{u_N^2} \text{ for }l\downarrow \lambda_N,
\end{equation}
so that $\varphi_N$ is also differentiable on $[\lambda_N,\infty)$.

\begin{proof}[Proof of Lemma \ref{lem: simplify}]
    Starting with the main case \eqref{eq: simplify main case}, note that introducing Lagrange multipliers we have
    \begin{equation}
    \sup_{\substack{    
                    |\sigma| = 1
                    \\
                    \sigma\cdot u=\alpha
                }} \sum_{i=1}^N\lambda_{i}\sigma_{i}^{2}\le\inf_{l,r\in\mathbb{R}}\sup_{\sigma\in\mathbb{R}^{N}}\mathscr{L}\left(\sigma,l,r\right),\label{eq: lagrange UB}
    \end{equation}
    where
    \[
    \begin{array}{rcl}
    \mathscr{L}(\sigma,l,r) & = & \sum_{i=1}^{N}\lambda_{i}\sigma_{i}^{2}-l\left(\sum_{i=1}^{N}\sigma_{i}^{2}-1\right)-r\left(\sum_{i=1}^{N}\sigma_{i}u_{i}-\alpha\right)\\
     & = & \sum_{i=1}^{N}\left((\lambda_{i}-l)\sigma_{i}^{2}-ru_{i}\sigma_{i}\right)+l+\alpha r.
    \end{array}
    \]
    Furthermore if there are some $l,r\in\mathbb{R},\sigma\in\mathbb{R}^{N}$
    achieving the minimax on the r.h.s. of (\ref{eq: lagrange UB}), then these $(\sigma,l,r)$ are a  critical point of $\mathscr{L}$ and in fact \eqref{eq: lagrange UB} with equality.
    
    When $l<\lambda_{N}$ then $\sup_{\sigma\in\mathbb{R}^{N}}\mathscr{L}\left(\sigma,l,r\right)=\infty$.
    The same is true for $l=\lambda_{N}$ and $r\ne0$. For $l\ge\lambda_{N},r=0$
    we have $\sup_{\sigma\in\mathbb{R}^{N}}\mathscr{L}\left(\sigma,l,0\right)=\lambda_{N}$.
    Thus
    \begin{equation}
    \inf_{l,r: l\le\lambda_{N}\text{ or }r=0}\sup_{\sigma\in\mathbb{R}^{N}}\mathscr{L}\left(\sigma,l,r\right)=\lambda_{N}.\label{eq: other cases}
    \end{equation}
    
    Now consider the remaining case $l>\lambda_{N},r\ne0$. In
    this case $\sup_{\sigma\in\mathbb{R}^{N}}\mathscr{L}\left(\sigma,r,l\right)$
    is maximized by 
    \[
    \sigma_{i}^{*}\left(l,r\right)=\frac{1}{2}\frac{ru_{i}}{\lambda_{i}-l},
    \]
    for which 
    \[
    \mathscr{L}\left(\sigma^{*}\left(r,l\right),r,l\right)=l+\alpha r+\frac{1}{4}\sum_{i=1}^{N}\frac{r^{2}u_{i}^{2}}{l-\lambda_{i}}.
    \]
    Since $\alpha\ne0$ by assumption we have
    \[
    \inf_{r\ne0}\sup_{\sigma\in\mathbb{R}^{N}}\mathscr{L}\left(\sigma,l,r\right)=\inf_{r\ne0}\left\{ l+\alpha r+\frac{r^{2}}{4}\sum_{i=1}^{N}\frac{u_{i}^{2}}{l-\lambda_{i}}\right\} =l-\frac{\alpha^{2}}{\sum_{i=1}^{N}\frac{u_{i}^{2}}{l-\lambda_{i}}},
    \]
    where the infimum is attained at 
    $r^{*}\left(l\right)=-2\frac{\alpha}{\sum_{i=1}^{N}\frac{u_{i}^{2}}{l-\lambda_{i}}}\ne0$,
    for which
    \begin{equation}\label{eq: lambda at r star} 
        \mathscr{L}\left(\sigma^{*}\left(l,r^{*}\right),l,r^{*}\left(l\right)\right)=l+\alpha^{2} \varphi_N(l).
    \end{equation}
    We have
    \begin{equation}\label{eq: y l to inf}
        l+\alpha^2 \varphi_N(l) \ge l - \alpha^2 (l-\lambda_1) \overset{|\alpha|<1}{\to} \infty \text{ as } l \uparrow \infty,
    \end{equation}
    and recalling \eqref{eq: varphi limit} we have $l+\alpha^{2}\varphi_N(l)\to\lambda_{N}$
    as $l\downarrow\lambda_{N}$. Furthermore by \eqref{eq: phi derivative in l limit}
    \begin{equation}\label{eq: y deriv when l is lambdaN}
        \frac{d}{dl}\left\{ l+\alpha^{2}\varphi_N(l) \right\} \to1-\frac{\alpha^{2}}{u_{N}^{2}}\overset{|\alpha|>|u_N|}{<}0\text{ as }l\downarrow\lambda_{N},
    \end{equation}
    so the infimum of \eqref{eq: lambda at r star} over $l>\lambda_{N}$ is attained at some $l^{*}>\lambda_{N}$,
    and we obtain
    \[
    \inf_{l,r: l>\lambda_{N},r\ne0}\sup_{\sigma\in\mathbb{R}^{N}}\mathscr{L}\left(\sigma,l,r\right) = 
    \mathscr{L}\left(\sigma^{*}\left(l^{*},r^{*}\right),l^{*},r^{*}\left(l^{*}\right)\right)=\inf_{l>\lambda_{N}}\left\{ l-\frac{\alpha^{2}}{\sum_{i=1}^{N}\frac{u_{i}^{2}}{\lambda_{i}-l}}\right\} <\lambda_{N}.
    \]
    Together with (\ref{eq: other cases}) this proves that the minimax
    in (\ref{eq: lagrange UB}) is indeed attained at some $l^{*},r^{*}\in\mathbb{R},\sigma^{*}\in\mathbb{R}^{N}$,
    so (\ref{eq: lagrange UB}) holds in equality and \eqref{eq: simplify main case}
    follows.
    
    Next considering \eqref{eq: simplify alpha small} note that the upper bound is trivial, and the lower bound follows by plugging in 
    \begin{equation}
    \begin{array}{rcl}
        \sigma &=& \sqrt{\frac{1-\alpha^2}{1-u_N^2}} e_N + \left(\alpha-\sqrt{\frac{1-\alpha^2}{1-u_N^2}} u_N\right) u 
    \end{array}
    \end{equation}
    where $e_N=(0,\ldots,0,1)$, which satisfies $|\sigma|=1$, $\sigma\cdot u = \alpha$
    and $\sigma_{N}^{2}\ge1-u_{N}^{2}\left(1+\frac{1}{\sqrt{1-u_{N}^{2}}}\right)$ so that
    \begin{equation}
        \sum_{i=1}^N\lambda_{i}\sigma_{i}^{2}
            \ge
            \lambda_N \sigma_N^2
            +
            \lambda_1 (1-\sigma_N^2)
            =
            \lambda_{N}-u_{N}^{2}\left(1+\frac{1}{\sqrt{1-u_{N}^{2}}}\right)\left(\lambda_{N}-\lambda_{1}\right).
    \end{equation}
\end{proof}

We now prove a few results about this minimization problem. The next lemma shows that the map $l \to -s_\mu(l)^{-1}$ is convex for any measure $\mu$, so in particular $\varphi_N(l)$ is convex. Note that for any $\mu$ and $k\in\mathbb{N}$
\begin{equation}\label{eq: k-th derivative of s_mu(l)}
    s_\mu^{(k)}(l) \overset{\eqref{eq:stieltjes_def}}{=} (-1)^{k} k! \int_{\mathbb{R}} \frac{1}{(l-x)^{k+1}} \mu(dx).
\end{equation}
\begin{lemma}[Convexity]\label{lem: phi convexity}
    For any $\lambda\in\mathbb{R}$ and measure $\mu$ on $\mathbb{R}$
    with support contained in $(-\infty,\lambda]$ the map $l\to-\frac{1}{s_{\mu}(l)}$
    is convex in $(\lambda,\infty)$. If the support of $\mu$ is not
    a singleton it is strictly convex.
\end{lemma}

\begin{proof}
    For $l>\lambda$ the second derivative equals
    \[
        \left(-\frac{1}{s_{\mu}(l)}\right)''=\frac{s_{\mu}^{''}(l)-2s_{\mu}^{'}(l)^{2}s_{\mu}(l)}{s_{\mu}(l)^{3}}.
    \]
    Letting $w(x)=\frac{1}{l-x}$ and using \eqref{eq: k-th derivative of s_mu(l)} the numerator equals
    \[
        2\int w(x)^{3}\mu(dx)-2\left(\int w(x)^{2}\mu(dx)\right)^{2}\int w(x)\mu(dx).
    \]
    Since $w(x)>0$ on the support of $\mu$ it holds that $\int w(x)\mu(dx)>0$,
    and dividing through by this quantity we obtain
    \[
        2\left(\frac{\int w(x)^{3}\mu(dx)}{\int w(x)\mu(dx)}-\left(\frac{\int w(x)^{2}\mu(dx)}{\int w(x)\mu(dx)}\right)^{2}\right).
    \]
    This is non-positive by the Cauchy-Schwartz inequality, and equals zero only if $w(x)^{2}$
    is constant on the support of $\mu$, which is only the case if the
    support of $\mu$ is a singleton.
\end{proof}

The previous lemma implies the following about a general version of the minimization in \eqref{eq: simplify main case}.
\begin{lemma}[Uniqueness]\label{lem: l uniqueness general}
    Let $\mu$ be a real measure with support which is not a singleton and is contained in $[\lambda_{-},\lambda_{+}]$ for $-\infty <\lambda_{-}<\lambda_{+}<\infty$. For any $\alpha^2 < 1$ there is a unique $l^* \ge \lambda_{+}$ that achieves the infimum of
    $$
        \inf_{l >    \lambda_{+} }
          \left\{l -  \frac{\alpha^2}{s_\mu(l)}\right\}
        ,
    $$
    and $l^* > \lambda_{+}$ iff $\alpha^{2}>\lim_{l\downarrow\lambda_{+}}\frac{s_{\mu}(l)^{2}}{-s_{\mu}'(l)}$. If $\alpha = \pm 1$ then the infimum equals $\int \lambda \mu(d\lambda)$ and is achieved for $l\to\infty$.
\end{lemma}
\begin{proof}
    If $\alpha = 0$ then $l^* = \lambda_{+}$ is the unique minimizer. If $\alpha^2 \in (0,1)$ then $l + \alpha^2 / s_\mu(l)$ is strictly convex  for $l \in (\lambda_+,\infty)$ by Lemma \ref{lem: phi convexity}, and similarly to \eqref{eq: y l to inf} it holds that $ l - \alpha^2/s_{\mu(l)} \ge l - \alpha^2 (l-\lambda_{-})\to\infty$ for $l \to \infty$. This implies that there is a unique minimizer in $[\lambda_{+},\infty)$. The minimizer is $\lambda_{+}$ iff
    $\lim_{l\downarrow\lambda_N} \frac{d}{dl} \{l - \alpha^2/s_{\mu}(l) \} \ge 0$ and
    \begin{equation}\label{eq:gen_opt_derivative}
        \frac{d}{dl}\left\{ l-\frac{\alpha^{2}}{s_{\mu}(l)}\right\} =1+\alpha^{2}\frac{s_{\mu}^{'}(l)}{s_{\mu}(l)^{2}},
    \end{equation}
    giving the condition in the statement.
    
    For  $\alpha=\pm 1$ it follows from \eqref{eq:stieltjes_def} and \eqref{eq: k-th derivative of s_mu(l)} with $k=1$ that the r.h.s of \eqref{eq:gen_opt_derivative} converges to $0$ for $l\to\infty$, which together with the convexity shows that the infimum is achieved for $l \to \infty$. Taylor expanding $\frac{1}{l-\lambda}$ yields
    \begin{equation}\label{eq:stieltjes_large_l_expansion_second}
        s_{\mu}(l)=\frac{1}{l}+\frac{1}{l^{2}}\int \lambda \mu(d\lambda)+O\left(\frac{\lambda_{+}^{2}}{l^{2}}\right)\text{ for }l \ge \lambda_{+} +1, 
    \end{equation}
    from which one can verify that the limit for $l\to\infty$ is $\int \lambda \mu(d\lambda)$.
\end{proof}

In particular for the minimization in \eqref{eq: simplify main case} we obtain the following from the previous lemma and \eqref{eq: phi derivative in l}-\eqref{eq: phi derivative in l limit}.
\begin{corollary}[Uniqueness for $\varphi_N$]\label{cor: l uniqueness}
    For any $\alpha^2 < 1, \lambda_1 < ... <\lambda_N$, $u_1^2,...,u_N^2 \in (0,1)$ with $\sum_{i=1}^N u_i^2 = 1$ there is a unique $l^* \ge \lambda_{{N}}$ that achieves the infimum of
    $$
        \inf_{l \ge \lambda_{{N}}}
          \left\{l + \alpha^2\varphi_N(l)\right\}
        ,
    $$
    and $l^* > \lambda_N$ iff $\alpha^2 > u_N^2$. If $\alpha = \pm 1$ then the infimum equals $\sum_{i=1}^N u_i^2 \lambda_i$ and is achieved for $l\to\infty$.
\end{corollary}

From Lemma \ref{lem: simplify} and Corollary \ref{cor: l uniqueness} we obtain the following.

\begin{corollary}\label{lem:estimate}
    For any $\lambda_1 < ... <\lambda_N$, $u_1^2,...,u_N^2 \in (0,1)$ with $\sum_{i=1}^N u_i^2 = 1$ it holds that
    \[
     \sup_{\alpha \in [-1,1]}\left|\sup_{\substack{|\sigma|=1\\ \sigma\cdot u=\alpha}}\sum_{i=1}^{N}\lambda_{i}\sigma_{i}^{2}-\inf_{l>\lambda_{N}}\left\{ l-\frac{\alpha^{2}}{s_{\lambda,u}(l)}\right\} \right|\le 2 (\lambda_N - \lambda_1)    \frac{u_{N}^{2}}{\sqrt{1-u_{N}^{2}}}.
    \]
\end{corollary}
\begin{proof}
    Note that the difference is exactly zero for $|\alpha| \ge |u_N|$ by \eqref{eq: simplify main case}.
    For $\left|\alpha\right|<\left|u_{N}\right|$ it follows from \eqref{eq: simplify alpha small} that
    \[
        \left|\sup_{\sigma\cdot u=\alpha}\sum_{i=1}^{N}\lambda_{i}\sigma_{i}^{2} - \lambda_{N}\right|
        \le
        2(\lambda_N - \lambda_1)\frac{u_{N}^{2}}{\sqrt{1-u_{N}^{2}}}.
    \]
    Also Corollary \ref{cor: l uniqueness} implies that if $|\alpha| < |u_N|$ then $l^* = \lambda_N$ and therefore
    \begin{equation*}
        \inf_{l>\lambda_{N}}\left\{ l-\frac{\alpha^{2}}{s_{\lambda,u}(l)}\right\} = \lambda_{N}.
    \end{equation*}
\end{proof}

\section{Leading order behavior}\label{section: leading order}

In this section we will study the behavior of $L_N$ and $\tilde{L}_N$ to leading order, proving Theorem \ref{thm: sphere} (a) and Theorem \ref{thm: ball} (a). These are in fact immediate consequences of \eqref{eq: fix overlap}, \eqref{eq: fix overlap ball} and the following proposition.
\begin{proposition}\label{prop:max_slice_term}
    It holds that
        \begin{equation}\label{eq:max_slice_term}
            \sup_{\alpha\in[-1,1]}\left|\sup_{\sigma\cdot u=\alpha}\sum_{i=1}^{N}\lambda_{i}\sigma_{i}^{2}-\sqrt{2(1-\alpha^{2})}\right| \overset{\mathbb{P}}{\to} 0.
        \end{equation}
\end{proposition}
Thanks to Corollary \ref{lem:estimate} and the facts that $u_N \overset{\mathbb{P}}{\to}0$ for $u$ uniform on the unit sphere, and that $\lambda_1,\lambda_N$ are stochastically bounded, this in turn is a direct consequence of 
\begin{equation}\label{eq:max_slice_term_inf}
    \sup_{\alpha\in[-1,1]}\left|\inf_{l>\lambda_{N}}\left\{ l-\frac{\alpha^{2}}{s_{\lambda,u}(l)}\right\}-\sqrt{2(1-\alpha^{2})}\right| \overset{\mathbb{P}}{\to} 0.
\end{equation}
The goal of the section is thus to prove \eqref{eq:max_slice_term_inf} and therefore Proposition \ref{prop:max_slice_term}.

To do so we will show laws of large numbers for $s_{\lambda,u}(l)$ and its derivatives in the first subsection, and in the second subsection use them to compute the infimum in \eqref{eq:max_slice_term_inf}.

\subsection{Law of large numbers for weighted Stieltjes transform}

In this subsection we give a leading order estimate for $s_{\lambda,u}(l)$, showing roughly speaking that $s_{\lambda,u}(l) \to s_{\mu_{\rm{sc}}}(l)$.  The following notations and results will also be useful later to handle the fluctuations of $s_{\lambda,u}(l)$ and $L_N,\tilde{L}_N$ in Section \ref{section: subleading order fluctuations}. To approximate $s_{\lambda,u}(l)$ by $s_{\mu_{\rm{sc}}}(l)$ we use the Stieltjes transforms of the measures
\begin{equation}\label{eq:mu_lambda_theta_theta_u}
    \mu_{\lambda}=\frac{1}{N}\sum_{i=1}^{N}\delta_{\lambda_{i}},\quad\quad\mu_{\theta}=\frac{1}{N}\sum_{i=1}^{N}\delta_{\theta_{i/N}},\quad\quad\mu_{\theta,u}=\frac{1}{N}\sum_{i=1}^{N}u_{i}^{2}\delta_{\theta_{i/N}},
\end{equation}
where the first two are empirical measures of random eigenvalues and
deterministic classical locations (recall \eqref{eq: classical locations}) respectively, and $\mu_{\theta,u}$ is a randomly weighted version of $\mu_{\theta}$, cf. \eqref{def: GNk}. As we already have for the Stieltjes transform of $\mu_{\lambda,u}$ we use the abbreviations (see \eqref{eq:semicircle_law_def}, \eqref{eq:stieltjes_def})
\begin{equation}\label{eq:s_plain_lambda_theta_theta_u}
    \begin{array}{lcl}
        \displaystyle{s(l)=s_{\mu_{{\rm sc}}}(l) = \int_{-\sqrt{2}}^{\sqrt{2}}\frac{\frac{1}{\pi}\sqrt{2-x^{2}}}{l-x}dx}, &  & \displaystyle{s_{\lambda}(l)=s_{\mu_{\lambda}}(l)=\frac{1}{N}\sum_{i=1}^{N}\frac{1}{l-\lambda_{i}}},\\
        \displaystyle{s_{\theta}(l)=s_{\mu_{\theta}}(l)=\frac{1}{N}\sum_{i=1}^{N}\frac{1}{l-\theta_{i/N}}}, &  & \displaystyle{s_{\theta,u}(l)=s_{\mu_{\theta,u}}(l)=\frac{1}{N}\sum_{i=1}^{N}\frac{u_{i}^{2}}{l-\theta_{i/N}}}.
    \end{array}
\end{equation}
The integral for $s(l)$ in \eqref{eq:s_plain_lambda_theta_theta_u}  can be computed explicitly yielding the following useful identities
\begin{equation*}
\begin{array}{rclcrcl}
    \IN{l} &=& l - \sqrt{l^2-2}, 
    &{\color{white}\bigg |}&
    \INk{1}{l}&=& -\frac{l-\sqrt{l^2-2}}{\sqrt{l^2-2}},
    \\
    \INk{2}{l}&=& \frac{2}{(l^2-2)^\frac{3}{2}}, 
    &{\color{white}\bigg |}&
    \INk{3}{l} &=& -\frac{6l}{(l^2-2)^{\frac{5}{2}}},
    \end{array}
   \numberthis\label{eq: useful integrals}
\end{equation*}
for all $k \in \mathbb{N}$ and $l > \sqrt{2}$. The two identities on the top row play a role in the study of the leading order here, and the higher derivatives on the bottom row will play a role in the study of the fluctuations in Section~\ref{section: subleading order fluctuations}.
Note that    
\begin{equation}\label{eq:s_props}
    s(\sqrt{2})=\sqrt{2},\quad\quad s(l)\text{ is decreasing on }[\sqrt{2},\infty),\text{ and }\quad\quad \lim_{l\to\infty}s(l)=0.
\end{equation}

We record the following direct consequence of Lemma \ref{lem: rigidity of eigenvalues}, comparing weighted sums over eigenvalues with the corresponding sum over classical locations.
\begin{lemma}\label{lem: sum eval class loc}
    For any $\delta>0$ we have
    \begin{equation}\label{eq: sum eval class loc}
        \mathbb{P}\left(
            \forall w\in C^{1}([-\sqrt{2}-\varepsilon,\sqrt{2}+\varepsilon]), u \in \mathcal{S}_{N-1}: \
            \left|
            \sum_{i=1}^{N} u_i^2 w\left(\lambda_{i}\right)-
            \sum_{i=1}^{N}u_i^2 w\left(\theta_{i/N}\right)\right|
            \le 
            |w'|_\infty
            N^{-\frac{2}{3}+\delta}
        \right) \to 1.
    \end{equation}
\end{lemma}

The following approximations are a consequence of the previous lemma and Lemma \ref{lem: sums}. 
\begin{lemma}\label{lem: G,H,I}
    Let $\varepsilon,\delta>0$ and $k \in\mathbb{N}$.
    It holds uniformly for all $l > \sqrt{2} + \varepsilon$ that
        \begin{equation}\label{eq: Hk Ik estimate}
        \left|\HNk{k}{\theta}{l} -
        \INk{k}{l} \right|
        = {{O}}\left(\frac{1}{N}\right),
    \end{equation}
    and
    \begin{align*}
        \HNk{k}{\theta}{l} & = \HNk{k}{\lambda}{l} + {{O}}_{\mathbb{P}}\left(N^{-\frac{2}{3}+\delta}\right),
        \numberthis\label{eq: Hk lambda theta identities}
        \\
        \GNk{k}{\theta}{l} & = \GNk{k}{\lambda}{l} + {{O}}_{\mathbb{P}}\left(N^{-\frac{2}{3}+\delta}\right).
        \numberthis\label{eq: Gk lambda theta identities}
    \end{align*}
\end{lemma}
\begin{proof}
    Let $w(l,\theta) = \tfrac{1}{l-\theta}$ and fix some $k\in\mathbb{N}$. By Lemma \ref{lem: sums}
    \begin{equation}\label{eq: uniform bound on s_theta(l) - s(l)}
        \left|\HNk{k}{\theta}{l} -
        \INk{k}{l} \right|
        \le \frac{\sup_{x\in [-\sqrt{2},\sqrt{2}]}|w^{(k+1)}(l,x)|}{N}
        \le \frac{1}{N} \frac{(k+1)!}{(l-\sqrt{2})^{k+2}}
        \le \frac{1}{N}\frac{(k+1)!}{\varepsilon^{k+2}} 
    \end{equation}
    for all $l \ge \sqrt{2}+\varepsilon$, which implies \eqref{eq: Hk Ik estimate}.
    On the event that $\lambda_N \le \sqrt{2} + \tfrac{\varepsilon}{2}$ we have by \eqref{eq: sum eval class loc} that for any $\delta > 0$
    \begin{equation*}
        \mathbb{P}
        \left(
        \forall l \ge \sqrt{2}+\varepsilon: \
        \left|\GNk{k}{\theta}{l} - \GNk{k}{\lambda}{l}\right| 
        \le  \frac{2(k+1)!}{\varepsilon^{k+2}} N^{-\frac{2}{3}+\delta}
        \right) \to 1,
    \end{equation*}
    implying \eqref{eq: Gk lambda theta identities}. The same argument for 
    $|\HNk{k}{\theta}{l} - \HNk{k}{\lambda}{l}|$ proves \eqref{eq: Hk lambda theta identities}.
\end{proof}

The following lemma gives a law of large numbers for sums over the classical locations or eigenvalues, weighted by the random $u_1^2,...,u_N^2$. It implies in particular that $\GNk{k}{\theta}{l} \to \INk{k}{l}$ and $\GNk{k}{\lambda}{l} \to \INk{k}{l}$ in probability.
\begin{lemma}\label{lem: u-fluc} 
    Let $\varepsilon>0$ and $w \in {C^1}([-\sqrt{2}-\varepsilon, \sqrt{2}+\varepsilon])$.
    Let $u$ be a random vector uniformly distributed on the sphere. Then as $N \rightarrow \infty$
            \begin{equation}\label{eq: u-fluc-theta}
                \sum_{i=1}^N w(\theta_{i/N}) u_i^2 
                \stackrel{\mathbb{P}}{\longrightarrow} 
                \int_{-\sqrt{2}}^{\sqrt{2}} w(x) \mu_{\text{sc}}(dx),
            \end{equation}
    and
            \begin{equation}\label{eq: u-fluc-lambda}
                \sum_{i=1}^N w(\lambda_i) u_i^2 
                \stackrel{\mathbb{P}}{\longrightarrow} 
                \int_{-\sqrt{2}}^{\sqrt{2}} w(x) \mu_{\text{sc}}(dx).
             \end{equation}
\end{lemma}
\begin{proof}
    Construct $u$ by setting $u_i = \frac{\tilde{u}_i}{|\tilde{u}|}$ with $\tilde{u}_1,...,\tilde{u}_N \sim \mathcal{N}(0,\tfrac{1}{N})$ i.i.d.. We then have
    \begin{align*}
        \mathbb{E}\left[\sum_{i=1}^N w(\theta_{i/N}) \tilde{u}_i^2 \right]
        =  \frac{1}{N} \sum_{i=1}^N w(\theta_{i/N})
        =  \int_{-\sqrt{2}}^{\sqrt{2}} w(x) \mu_{\text{sc}}(dx) + o(1),
    \end{align*}
    by Lemma \ref{lem: sums} and since $\text{Var}\left(N\tilde{u}_i^2\right) = 2$
    \begin{equation}\label{eq: w sum var}
        \text{Var}\left(\sum_{i=1}^N w(\theta_{i/N}) \tilde{u}_i^2\right)
        =
        \frac{2}{N^2} \sum_{i=1}^N w(\theta_{i/N}) 
        = {{O}}\left(\frac{1}{N}\right).
    \end{equation}
    Therefore
    \begin{align*}
        \sum_{i=1}^N w(\theta_{i/N}) \tilde{u}_i^2 
        \stackrel{\mathbb{P}}{\longrightarrow}
        \int_{-\sqrt{2}}^{\sqrt{2}} w(x) \mu_{\text{sc}}(dx)\text{ as }N\to\infty,
    \end{align*}
    and since also $|\tilde{u}| \to 1$ in probability the claim \eqref{eq: u-fluc-theta} follows.
    The second claim then follows from Lemma \ref{lem: sum eval class loc}.
\end{proof}
The previous lemma implies the following uniform convergence of $\GNk{k}{\theta}{l}$ to $\INk{k}{l}$.

\begin{lemma}\label{lem: GN IN uniformly o(1)}
Let $\varepsilon,\delta>0$ and $k\in\mathbb{N}$.
For any $\varepsilon>0,L>2$
\[
\sup_{l\in[\sqrt{2}+\varepsilon,L]}\left|\GNk{k}{\theta}{l}-\INk{k}{l}\right|=o_{\mathbb{P}}(1).
\]
\end{lemma}
\begin{proof} Firstly, by Lemma \ref{lem: u-fluc} and a union bound it holds for
all $\delta>0$ that
\[
    \lim_{N\to\infty}\sup_{l\in[\sqrt{2}+\varepsilon,L]\cap\delta\mathbb{Z}}\mathbb{P}\left(
\left|\GNk{k}{\theta}{l}-\INk{k}{l}\right| 
\ge \delta \right)=0.
\]
Secondly, since $l\to\frac{1}{(l-\theta_{i/N})^k}$ is Lipschitz for $l\ge\sqrt{2}+\varepsilon$ so are $\GNk{k}{\theta}{l}$ and $\INk{k}{l}$.
These two facts imply the claim.
\end{proof}

\begin{remark}\label{rem: GN IN uniformly o(1)}
    (a) Though we do not need it here, it is easy to argue that the convergence is uniform on $[\sqrt{2}+\varepsilon,\infty)$, since $\lim_{l\to\infty}s^{(k)}_\mu(l)=0$ for all $k$ and $\mu$ with compact support. (b) In Section \ref{subsec:fluctuations_of_s_lambda_u} we strengthen the bound to ${{O}}_{\mathbb{P}}(N^{-1/2})$, as this is needed to study the fluctuations of $L_N$ and $\tilde{L}_N$ (see \eqref{eq:s_lambda_s_sqrt_N}).
\end{remark}

The estimate \eqref{eq: Gk lambda theta identities} and Lemma \ref{lem: GN IN uniformly o(1)} together imply that for all $\varepsilon>0,L>2,$
\begin{equation}\label{eq: IN GN distance o(1)}
    \GN{\lambda}{l} \to s(l) \text{ uniformly in probability on }[\sqrt{2}+\varepsilon,L].
\end{equation}
The next lemma deduces from this that also $\GN{\lambda}{l}^{-1} \to \IN{l}^{-1}$ uniformly, and here we do take care to prove it for an unbounded interval.

\begin{lemma}\label{lem:inverse_s_conv}
    For all $\varepsilon>0$ it holds that
    \[
        \sup_{l \ge \sqrt{2} + \varepsilon }\left|\frac{1}{s_{\lambda,u}(l)}-\frac{1}{s(l)}\right| \overset{\mathbb{P}}{\to}0.
    \]
\end{lemma}
\begin{proof}
From \eqref{eq: useful integrals} it follows that $s(l)=l^{-1}+O\left(l^{-3}\right)$ for $l$
large. Similarly from \eqref{eq:stieltjes_large_l_expansion_second}
\[
    s_{\lambda,u}(l)=\frac{1}{l}+\frac{O\left(\sum_{i=1}^{N}u_{i}^{2}\lambda_{i}\right)}{l^{2}}+O\left(\frac{\max(|\lambda_1|,|\lambda_{N}|)^{3}}{l^{3}}\right),
\]
for all $u\in \mathcal{S}_{N-1}$ and $l\ge\lambda_{N}+1$. By Lemma \ref{lem: u-fluc} with $w(x)=x$ it holds that $\sum_{i=1}^{N}u_{i}^{2}\lambda_{i}\to \int_{-\sqrt{2}}^{\sqrt{2}} x \mu_{\rm{sc}}(dx)=0$
in probability, and since also $\lambda_{1},\lambda_N$ are stochastically bounded it follows that for each $\eta>0$ there is a large enough $L$ such that
    \[
        \lim_{N\to\infty}\mathbb{P}\left(\sup_{l\ge L}\left|\frac{1}{s_{\lambda,u}(l)}-\frac{1}{s(l)}\right|\ge\frac{\eta}{2}\right)=0.
    \]
    Furthermore \eqref{eq: IN GN distance o(1)} implies that
    \[
    \lim_{N\to\infty}\mathbb{P}\left(\sup_{l\in[\sqrt{2}+\varepsilon,L]}\left|\frac{1}{s_{\lambda,u}(l)}-\frac{1}{s(l)}\right|\ge\frac{\eta}{2}\right)=0,
    \]
    giving the claim.
\end{proof}

\subsection{Leading order estimate for Lagrange optimization}

We now use the laws of large numbers to study the optimization problem
\begin{equation}\label{eq:lagrange_opt}
    \inf_{l>\lambda_N} \left\{ l - \frac{\alpha^2}{\GN{\lambda}{l}} \right\},
\end{equation}
from \eqref{eq:max_slice_term_inf}. The law of large numbers $s_{\lambda,u}(l)\to s(l)$ leads us to consider the limiting optimization problem 
\begin{equation}\label{eq:lagrange_opt_lim}
    \inf_{l>\sqrt{2}}\left\{ l-\frac{\alpha^{2}}{s(l)}\right\}.
\end{equation}
The next lemma solves this limiting optimization.
\begin{lemma}
    For all $\alpha\in[-1,1]$
    \begin{equation}
        \inf_{l>\sqrt{2}}\left\{ l-\frac{\alpha^{2}}{s(l)}\right\} =\sqrt{2(1-\alpha^{2})},\label{eq: y limit}
    \end{equation}
    and if $\alpha\in(-1,1)$ the unique minimizer is 
    \begin{equation}
        \hat{l}(\alpha)=\frac{2-\alpha^{2}}{\sqrt{2(1-\alpha^{2})}},\label{eq: l hat def}
    \end{equation}
    while if $\alpha=\pm1$ the infimum is achieved for $l\to\infty$.
\end{lemma}

\begin{proof}
    We have
    \begin{equation}\label{eq:derivative_in_limit}
        \frac{d}{dl}\left\{ l-\frac{\alpha^{2}}{s(l)}\right\}
        = 1+\alpha^{2}\frac{s'(l)}{s(l)^2}
        \overset{\eqref{eq: useful integrals}}{=}
        1-\alpha^{2}\frac{1}{\left(l-\sqrt{l^{2}-2}\right)\sqrt{l^{2}-2}}
        = 1-\frac{\alpha^{2}}{1-x^{2}},
    \end{equation}
    where the last equality comes from the change of variables
    $l=\frac{1}{\sqrt{2}}\left(x+x^{-1}\right)$
    for which $\sqrt{l^{2}-2}=\frac{1}{\sqrt{2}}\left(x-x^{-1}\right)$. If $\alpha\in(-1,1)$
     the critical point equation thus has unique solution $x=\sqrt{1-\alpha^{2}}$
    which yields \eqref{eq: l hat def}. The claim for $\alpha^2=1$ follows from the general Lemma \ref{lem: l uniqueness general}, or since the derivative \eqref{eq:derivative_in_limit} is negative for all $l > \sqrt{2}$.
\end{proof}

We recognize on the r.h.s. of \eqref{eq: y limit} the term that \eqref{eq:max_slice_term_inf}
claims is the limit of \eqref{eq:lagrange_opt}. To prove \eqref{eq:max_slice_term_inf} we thus need to
approximate the random optimization \eqref{eq:lagrange_opt} by the limiting \eqref{eq:lagrange_opt_lim}.

From the explicit formula (\ref{eq: l hat def}) it follows that minimizer in the limiting problem \eqref{eq:lagrange_opt_lim} is bounded away from $\sqrt{2}$ if $\alpha$ is bounded away from zero,
and bounded if $\alpha$ is bounded away from $\pm1$. Formally, for
all $\delta$ there exists a $\varepsilon>0$ such that $\hat{l}(\alpha)\ge\sqrt{2}+\varepsilon$
if $\left|\alpha\right|\ge\delta$ and $\hat{l}(\alpha)\le\varepsilon^{-1}$
if $\left|\alpha\right|\le1-\delta$, and thus 
\begin{equation}\label{eq:leading_order_lim_for_alpha_small_large}
    \inf_{\lambda>\sqrt{2}}\left\{ l-\frac{\alpha^{2}}{s(l)}\right\} =
    \left\{ \begin{array}{rll}
    {\displaystyle \inf_{\lambda\ge\sqrt{2}+\varepsilon}} & { \left\{ l-\frac{\alpha^{2}}{s(l)}\right\} } & \text{\,if }\left|\alpha\right|\ge\delta,\\
    {\displaystyle \inf_{\lambda\in[\sqrt{2},\varepsilon^{-1}]}} & { \left\{ l-\frac{\alpha^{2}}{s(l)}\right\} } & \text{\,if }\left|\alpha\right|\le1-\delta.
    \end{array}\right.
\end{equation}
The next lemma shows that this also holds for the random optimization problem \eqref{eq:lagrange_opt}.
\begin{lemma}\label{lem: l minimizer epsilon distance}
Let $y(\alpha,l) = l - \frac{\alpha^2}{\GN{\lambda}{l}}$. For each $\delta>0$ there is an $\varepsilon>0$ such that
\begin{equation}\label{eq: l minimizer epsilon distance}
    \lim_{N\to \infty} \mathbb{P}\left(\inf_{l\ge\lambda_{N}}y(\alpha,l)=\inf_{l \ge \sqrt{2}+\varepsilon}y(\alpha,l), \ \forall\alpha: |\alpha| \ge  \delta \right)=1
\end{equation}
    and 
    \begin{equation}\label{eq: l minimizer finite for small alpha}
        \lim_{N\to \infty} \mathbb{P}\left(\inf_{l\ge\lambda_{N}}y(\alpha,l)=\inf_{l\in[\lambda_N,\varepsilon^{-1}]}y(\alpha,l), \ \forall\alpha: |\alpha| \le 1-\delta \right)=1.
    \end{equation}
\end{lemma}
\begin{proof}
For any $l>\sqrt{2}$ and $\alpha$
\[
    \partial_{l}y\left(\alpha,l\right)
    =
    1+\alpha^2\frac{\GNk{1}{\lambda}{l}}{\GN{\lambda}{l}^2}    
    \overset{\mathbb{P}}{\underset{\text{Lem }\ref{lem: u-fluc}}{\longrightarrow}}
    1+\alpha^{2}
    \frac{\INk{1}{l}}{\IN{l}^2}
    = 1-\frac{\alpha^{2}}{1-\frac{1}{2}s(l)^{2}} =: t(\alpha,l)
\]
where the final expression follows by \eqref{eq:derivative_in_limit}, since inverting the change
of variables $l=\frac{1}{\sqrt{2}}(x+x^{-1})$ used there yields $x=\frac{1}{\sqrt{2}}(l-\sqrt{l^{2}-2})=\frac{1}{\sqrt{2}}s(l)$. By \eqref{eq:s_props} the r.h.s. tends to $-\infty$ if $l\downarrow\sqrt{2}$ and $\alpha\ne0$,
and to $1-\alpha^{2}>0$ if $l\uparrow\infty$ and $\left|\alpha\right|<1$.
Thus there is an $\varepsilon>0$ small enough so that
\[
t(\delta,\sqrt{2} + \varepsilon)<0\quad\quad\text{and}\quad\quad t(1-\delta,\varepsilon^{-1})>0.
\]
Since $\frac{\GNk{1}{\lambda}{l}}{\GN{\lambda}{l}^2}$ is negative for all $l>\lambda_{N}$ (see e.g. \eqref{eq: phi derivative in l}) we have $\partial_{l}y(\alpha,\sqrt{2}+\varepsilon)\le\partial_{l}y(\delta,\sqrt{2}+\varepsilon)$
for $\left|\alpha\right|\ge\delta$ on the event $\sqrt{2}+\varepsilon>\lambda_{N}$
(which has probability tending to one). It follows that 
\[
    \lim_{N\to\infty}\mathbb{P}\left(\partial_{l}y(\alpha,\sqrt{2}+\varepsilon)<0,\quad\forall\alpha:\left|\alpha\right|\ge\delta\right)=1.
\]
Since $y(\alpha,l)$ is almost surely convex in $l>\lambda_{N}$ by Lemma \ref{lem: phi convexity} the
claim \eqref{eq: l minimizer epsilon distance} follows. The claim \eqref{eq: l minimizer finite for small alpha} follows similarly since $\partial_{l}y(\alpha,\sqrt{2}+\varepsilon) \ge \partial_{l}y(1-\delta,\sqrt{2}+\varepsilon)$
for $\left|\alpha\right|\le1-\delta$ (if $\sqrt{2}+\varepsilon>\lambda_{N}$), so that 
\[
    \lim_{N\to\infty}\mathbb{P}\left(\partial_{l}y(\alpha,\varepsilon^{-1})>0,\quad\forall\alpha:\left|\alpha\right|\le1-\delta\right)=1.
\]
\end{proof}

We can now compute \eqref{eq:lagrange_opt} for $\alpha$ bounded away from zero.
\begin{lemma}\label{lem:alpha_large}
For all $\delta>0$ 
    \begin{equation}\label{eq:alpha_large}
        \sup_{\alpha\in[-1,1]:\delta \le \left|\alpha\right| }
        \left|\inf_{l>\lambda_{N}}\left\{ l-\frac{\alpha^{2}}{s_{\lambda,u}(l)}\right\} -\sqrt{2(1-\alpha^{2})}\right|
        \overset{\mathbb{P}}{\to}0.
    \end{equation}
\end{lemma}

\begin{proof}
If we pick $\varepsilon$ small enough depending on $\delta$ then by Lemma \ref{lem:inverse_s_conv} and \eqref{eq: l minimizer epsilon distance}
\[                                                                                                                                              
    \sup_{\alpha\in[-1,1]:\left|\alpha\right|\ge\delta}\left|\inf_{l>\lambda_{N}}\left\{ l-\frac{\alpha^{2}}{s_{\lambda,u}(l)}\right\} -\inf_{l\ge \sqrt{2}+\varepsilon}\left\{ l-\frac{\alpha^{2}}{
        s(l)
    }\right\} \right|\overset{\mathbb{P}}{\to}0,
\]
while by \eqref{eq:leading_order_lim_for_alpha_small_large} and \eqref{eq: y limit}  also
\[
    \inf_{l \ge \sqrt{2}+\varepsilon}\left\{ l-\frac{\alpha^{2}}{s(l)}\right\} = \inf_{l > \sqrt{2}}\left\{ l-\frac{\alpha^{2}}{s(l)}\right\} = \sqrt{2(1-\alpha^2)}\text{ for all }|\alpha| \ge \delta.
\]
\end{proof}

Next we estimate \eqref{eq:lagrange_opt} for $\alpha$ close to zero.

\begin{lemma}
\label{lem:alpha_small}There is a universal constant $c$ such that for all $\delta>0$
\begin{equation}\label{eq:alpha_small}
    \lim_{N\to\infty} \mathbb{P}\left(\sup_{\alpha\in[-1,1]:\left|\alpha\right|\le\delta}\left|\inf_{l>\lambda_{N}}\left\{ l-\frac{\alpha^{2}}{s_{\lambda,u}(l)}\right\} -\sqrt{2(1-\alpha^{2})}\right|\ge c\delta\right)=0.
\end{equation}
\end{lemma}

\begin{proof}
If $\left|\alpha\right|\le\delta$ then 
\[
\sqrt{2}\ge\inf_{l>\lambda_{N}}\left\{ l-\frac{\alpha^{2}}{s_{\lambda,u}(l)}\right\} \ge\inf_{l>\lambda_{N}}\left\{ l-\frac{\delta^{2}}{s_{\lambda,u}(l)}\right\} \overset{\mathbb{P}}{\to}\sqrt{2(1-\delta^{2})}=\sqrt{2}+{{O}}(\delta),
\]
where we used Lemma \ref{lem:alpha_large}. Since also $\sqrt{2(1-\alpha^{2})}=\sqrt{2}+{{O}}(\delta)$ this implies \eqref{eq:alpha_small}.
\end{proof}

\begin{proof}[Proof of Proposition \ref{prop:max_slice_term}]
    The convergence \eqref{eq:max_slice_term_inf} is a consequence of Lemma \ref{lem:alpha_large} and Lemma \ref{lem:alpha_small}. By Corollary \ref{lem:estimate} this proves \eqref{eq:max_slice_term}, since $u_N\to 0$ in probability and $\lambda_1,\lambda_N$ are stochastically bounded (see Lemma \ref{lem: rigidity of eigenvalues}).
\end{proof}

This also completes the proofs of Theorem \ref{thm: sphere} (a) and Theorem \ref{thm: ball} (a), since as already mentioned the former follows from \eqref{eq: fix overlap} and Proposition \ref{prop:max_slice_term}, while the latter follows from \eqref{eq: fix overlap ball} and Proposition \ref{prop:max_slice_term}

\section{Examples: Leading order}\label{section: example leading order}

In this section we consider some important special cases where specific choices are made for $f$ and $g$
and characterize the maximizing $\alpha$ and $r$ as explicitly as possible. Later after proving Theorems \ref{thm: sphere} (b) resp. \ref{thm: ball} (b) about fluctuations we will see that they also apply to these examples.

Recall
$$
    \mathcal{B}(\alpha) = f(\alpha) + \sqrt{2}\beta\sqrt{1-\alpha^2}.
$$
We will first consider the ground state $L_N$ on the sphere for monomials $f(x)=h x^k$. Define for $h>0$
\begin{equation}\label{eq: beta critical}
    \beta_c(k,h) := 
        \begin{cases}
            \infty & \text{ for } k= 1,
            \\
            \sqrt{2}h & \text{ for } k= 2,
            \\
            \frac{h}{\sqrt{2}} \frac{k-1}{k-2}\left(1-\frac{1}{(k-1)^2}\right)^{\frac{k}{2}}
            & \text{ for } k \ge 3.
        \end{cases}
\end{equation}
Let also for $k \ge 3$
\begin{equation}\label{eq: betac tilde def}
\tilde{\beta}_{c}\left(k,h\right)=\frac{hk}{\sqrt{2}}\frac{\left(k-2\right)^{\frac{k-2}{2}}}{\left(k-1\right)^{\frac{k-1}{2}}} > \beta_{c}\left(k,h\right).
\end{equation}

The next lemma shows for monomial $f$ that $\mathcal{B}(\alpha)$ has a unique maximizer $\hat{\alpha}$ for $\beta \neq \beta_c(k,h)$, where $\hat{\alpha}=0$ if  $\beta > \beta_c(k,h)$ and $\hat{\alpha}>0$ if  $\beta < \beta_c(k,h)$.

\begin{lemma}[Ground state on sphere for monomials]\label{lem: monomial maxima}
When $k=1$ then for all $\beta>0$
\[
\sup_{\alpha\in\left[-1,1\right]}\mathcal{B}\left(\alpha\right)=
\sqrt{h^{2}+2\beta^{2}},
\]
and the unique local and global maximizer of $\mathcal{B}\left(\alpha\right)$
is $\alpha=\frac{h}{\sqrt{h^{2}+2\beta^{2}}}$.

When $k=2$ and $\beta\ge\beta_{c}\left(2,h\right)$ the unique local
and global maximizer of $\mathcal{B}\left(\alpha\right)$ is $\alpha=0$
and when $\beta<\beta_{c}\left(2,h\right)$ 
\[
\sup_{\alpha\in\left[-1,1\right]}\mathcal{B}\left(\alpha\right)=h+\frac{\beta^{2}}{2h},
\]
and the unique local and global maximizers of $\mathcal{B}\left(\alpha\right)$
are $\alpha=\pm\sqrt{1-\frac{\beta^{2}}{2h^{2}}}$.

When $k\ge3$ and $\beta\ge\tilde{\beta}_{c}\left(k,h\right)$ the unique local and global
maximizer of $\mathcal{B}\left(\alpha\right)$ is $\alpha=0$. When
$\beta<\tilde{\beta}_{c}\left(k,h\right)$ let $\hat{\alpha}$ be
the largest solution to 
\begin{equation}
\alpha^{2\left(k-2\right)}\left(1-\alpha^{2}\right)=2\left(\frac{\beta}{hk}\right)^{2},\label{eq: alpha eq}
\end{equation}
which is the unique solution to the equation in $\left(\sqrt{\frac{k-2}{k-1}},1\right)$.
Then $\alpha=0,\alpha=\hat{\alpha}$ are the only local maximizers of
$\mathcal{B}\left(\alpha\right)$ in $\left[0,1\right]$. When $\beta>\beta_{c}\left(k,h\right)$
the global maximizer is $\alpha=0$ and when $\beta=\beta_{c}\left(k,h\right)$
both $\alpha=0$ and $\alpha=\hat{\alpha}$ are global maximizers, and
when $\beta<\beta_{c}\left(k,h\right)$ the global maximizer in $\left[0,1\right]$
is $\hat{\alpha}$.

When $k\ge4$ and $k$ even then $\alpha=-\hat{\alpha}$ is also local
resp. global maximizer and the unique one in $[-1,0)$, and if $k\ge3$
and $k$ odd then there are no local maximizers in $[-1,0)$.
\end{lemma}
\begin{remark}
Also when $k=1,2$ and $\beta>\beta_{c}\left(k,h\right)$
the unique global maximizer is a solution of \eqref{eq: alpha eq} (in
fact the unique solution).
\end{remark}

\begin{proof}
    
Since $\mathcal{B}'\left(\alpha\right)\to-\infty$ for $\alpha\to\pm1$
a non-negative maximizer must exist and it must be a local maximizer
of $\mathcal{B}\left(\alpha\right)$ in $(-1,1)$. We have 
\[
\mathcal{B}'\left(\alpha\right)=hk\alpha^{k-1}-\sqrt{2}\beta\frac{\alpha}{\sqrt{1-\alpha^{2}}}.
\]
For $k$ odd we have $\mathcal{B}'\left(\alpha\right)<0$ for $\alpha\in\left(-1,0\right)$,
so there are no local maximizers in that interval. If $k$ is even and
thus $\mathcal{B}$ is symmetric, every local or global maximizer $-\alpha<0$
must correspond to $+\alpha>0$ that is also a local resp. global
maximizum of $\mathcal{B}$. Thus we may now restrict attention to
$\alpha\in\left[0,1\right]$.

For $k=1$ and all $\beta>0$ we have that $\mathcal{B}'\left(\alpha\right)=0\iff h-\sqrt{2}\beta\frac{\alpha}{\sqrt{1-\alpha^{2}}}=0$
has the unique solution $\frac{h}{\sqrt{h^{2}+2\beta^{2}}}$ which
must then be the unique local and global maximizer of $\mathcal{B}\left(\alpha\right)$,
and indeed $\mathcal{B}(\frac{h}{\sqrt{h^{2}+2\beta^{2}}})=h^{2}+2\beta^{2}$.
This completes the proof in the case $k=1$.

For $k\ge2$ we will use that
\[
\mathcal{B}''\left(\alpha\right)=hk\left(k-1\right)\alpha^{k-2}-\sqrt{2}\beta\frac{1}{\left(1-\alpha^{2}\right)^{3/2}}.
\]

When $k=2$ then $\mathcal{B}'\left(0\right)=0$ for all $\beta$.
If $\beta\ge\beta_{c}\left(k,h\right)$ then $\mathcal{B}'\left(\alpha\right)=0\iff2h\alpha-\sqrt{2}\beta\frac{\alpha}{\sqrt{1-\alpha^{2}}}=0$
has no non-zero solutions, so $\alpha=0$ is the unique local and
global maximizer. If $k=2$ and $\beta<\beta_{c}\left(k,h\right)$
then the unique positive solution of $\mathcal{B}'\left(\alpha\right)=0$
is $\sqrt{1-\frac{\beta^{2}}{2h^{2}}}$, and
\[
\mathcal{B}\left(\sqrt{1-\frac{\beta^{2}}{2h^{2}}}\right)=h+\frac{\beta^{2}}{2h^{2}}=\frac{\beta}{\sqrt{2}}\left(\frac{\sqrt{2}h}{\beta}+\frac{\beta}{\sqrt{2}h}\right)>\sqrt{2}\beta=\mathcal{B}\left(0\right),
\]
so this is the global maximum. Also $\mathcal{B}''\left(0\right)=2h-\sqrt{2}\beta>0$
so $\alpha=0$ is a local minimizer. This completes the proof in the
case $k=2$.

If $k\ge3$ then $\alpha=0$ is always a local maximizer of $\mathcal{B}\left(\alpha\right)$.
 Also the l.h.s. of (\ref{eq: alpha eq}) is maximized at $\alpha=\sqrt{\frac{k-2}{k-1}}$,
so when $\beta>\tilde{\beta}_{c}\left(k,h\right)$ then using \eqref{eq: betac tilde def} the l.h.s. of \eqref{eq: alpha eq} is
smaller than the r.h.s. for all $\alpha$, so the equation has no
solutions and $\alpha=0$ is the unique maximizer. When $\beta=\tilde{\beta}_{c}\left(k,h\right)$
it has a single solution at $\alpha=\sqrt{\frac{k-2}{k-1}}$ and otherwise
one in $\left(0,\sqrt{\frac{k-2}{k-1}}\right)$ and one in $\left(\sqrt{\frac{k-2}{k-1}},1\right)$.
At any solution $\alpha$ of $\mathcal{B}'\left(\alpha\right)=0$
we have that
\begin{align*}
\begin{array}{ccl}
\mathcal{B}''\left(\alpha\right) & = & \left(k-1\right)\sqrt{2}\beta\frac{1}{\sqrt{1-\alpha^{2}}}-\sqrt{2}\beta\frac{1}{\left(1-\alpha^{2}\right)^{3/2}}\\
 & = & \frac{\sqrt{2}\beta}{\sqrt{1-\alpha^{2}}}\left( k-1-\frac{1}{1-\alpha^{2}}\right) \begin{cases}
<0 & \text{ if }\alpha>\sqrt{\frac{k-2}{k-1}},\\
=0 & \text{ if }\alpha=\sqrt{\frac{k-2}{k-1}},\\
>0 & \text{ if }\alpha<\sqrt{\frac{k-2}{k-1}}.
\end{cases}
\end{array}
\end{align*}
This shows that when $\beta=\tilde{\beta}_{c}\left(k,h\right)$ we
have that $\alpha=\sqrt{\frac{k-2}{k-1}}$ is a saddle point (using that $\alpha=0$ is a local maximizer and $\mathcal{B}'\left(\alpha\right)\to-\infty$
for $\alpha\to1$), and when $\beta<\tilde{\beta}_{c}\left(k,h\right)$
the smaller solution is a local minimizer and the larger one is local
maximizer. It only remains to check which of the two local maximizers is the global maximizer when $\beta < \tilde{\beta}_c(k,h)$.

To this end note that
\[
\mathcal{B}\left(\alpha\right)>\mathcal{B}\left(0\right)\iff\frac{\alpha^{k}}{1-\sqrt{1-\alpha^{2}}}>\frac{\sqrt{2}\beta}{h}.
\]
The left-hand side is uniquely maximized at $\tilde{\alpha}=\frac{\sqrt{k\left(k-2\right)}}{k-1}$.
Thus if $\beta>\beta_{c}\left(k,h\right)$ so that $\frac{\tilde{\alpha}^{k}}{1-\sqrt{1-\tilde{\alpha}^{2}}}<\frac{\sqrt{2}\beta}{h}$
the global maximizer is $\alpha=0$, and if $\beta=\beta_{c}\left(k,h\right)$
we have $\mathcal{B}\left(\tilde{\alpha}\right)=\mathcal{B}\left(0\right)$
and $\mathcal{B}\left(\alpha\right)<\mathcal{B}\left(0\right)$ for
all $\alpha\in(0,1)\backslash\{\tilde{\alpha}\}$ so both $\alpha=0$
and $\alpha=\tilde{\alpha}$ are global maximizers, and the latter
is the aforementioned non-zero local maximizer. Lastly if $\beta<\beta_{c}\left(k,h\right)$
then the global maximizer is non-zero and is the aforementioned
non-zero local maximizer. This completes the proof for $k\ge3$.
\end{proof}

\begin{figure}[b]
	\centering
	\includegraphics[width=1\linewidth]{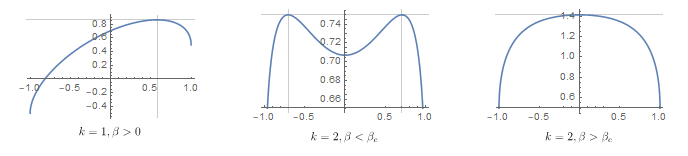}
	\label{fig:HTAP}
	\includegraphics[width=1\linewidth]{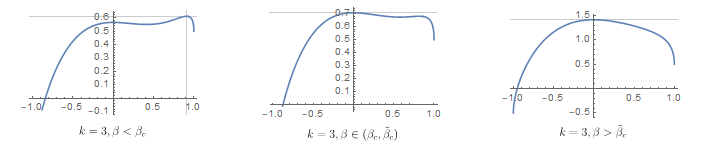}
    \caption{Plot of $\mathcal{B}(\alpha)$ for $\alpha\in [-1,1]$}
\end{figure}

\bigskip
We will now study an important special case of $\tilde{L}_N$. Recall the TAP free energy
$$ F_{\rm{TAP}}(m) = \beta H_N(m) + N f(u \cdot m) 
                + N g(|m|),$$
where $\beta\ge0$ and
$$
    g(x) = \frac{1}{2}\log\left(1-x^2\right) + \frac{\beta^2}{2}(1-x^2)^2 \text{ for }x\ge0.
$$
Let $q_P = \max( 1-\frac{1}{\sqrt{2\beta}},0)$ and define the \textit{Plefka region}
    \begin{equation}
        {\rm{Plef}}(\beta) = \left[\sqrt{q_P}, 1\right] \subset [0,1],
    \end{equation}
and denote its interior by $\rm{Plef}(\beta)^{\mathrm{o}}$.
In TAP analysis one is interested in the maximum of $F_{\rm{TAP}}$ for $m$ such that $|m|\in \rm{Plef}(\beta)$, that is in $\tilde{L}_N$ for this $g$ and $\mathcal{R}=\rm{Plef}(\beta)$.
Let $h>0,f(x)=h x^k$ for $k\ge1$ and define 
\begin{equation}
        \tilde{\mathcal{B}}(\alpha, r) = f(r \alpha) + \sqrt{2}\beta r^2 \sqrt{1-\alpha^2} + g(r),
\end{equation}
so that by \eqref{eq: L tilde leading order} 
$$
\frac{1}{N}\tilde{L}_N \overset{\mathbb{P}}{\to} \sup_{r\in{\rm{Plef}(\beta)},\alpha\in[-1,1]} \tilde{\mathcal{B}}(\alpha,r).$$
In the rest of the section we will compute the r.h.s. explicitly as possible, and show that except for critical values of $\beta,h$ it has a unique maximizer.

\begin{lemma}[TAP maximizer with linear external field]\label{lem: k=1 maximum on ball}
    Let $h > 0$, $\beta > 0$ and $f(x) = h x$. It holds hat
\begin{equation}
\sup_{r\in{\rm Plef}\left(\beta\right),\alpha\in\left[-1,1\right]}\tilde{\mathcal{B}}(\alpha,r)=\sup_{q\in[q_P ,1)}\mathscr{B}\left(q\right),\label{eq: tap linear}
\end{equation}
where 
\[
\mathscr{B}\left(q\right)=\sqrt{h^{2}q+2\beta^{2}q^{2}}+\frac{1}{2}\log(1-q)+\frac{\beta^{2}}{2}(1-q)^{2},
\]
is a concave function in $[q_P ,1)$
whose unique maximizer $\hat{q}$ is the unique solution to
\begin{equation}\label{eq: TAP linear qhat}
    \frac{q}{h^{2}+2q\beta^{2}}=(1-q)^2
\end{equation}
in $\left(q_P ,1\right)$.
Furthermore the unique maximizer of the l.h.s. of (\ref{eq: tap linear})
is $\hat{r}=\sqrt{\hat{q}}$ and $\hat{\alpha}=\frac{h}{\sqrt{h^{2}+2\beta^{2}\hat{q}}}$.
\end{lemma}
\begin{proof}
    We will first maximize $\tilde{\mathcal{B}}$ in $\alpha$ for fixed $r \ne 0$. Since 
    $$\partial_\alpha \tilde{\mathcal{B}}(\alpha,r)=h r - \sqrt{2}\beta r^2 \frac{\alpha}{\sqrt{1-\alpha^2}} \to-\infty\text{ for }\alpha\to \pm1$$
    a maximizer must exist and be a critical point. The critical point equation $\partial_\alpha \tilde{\mathcal{B}}(\alpha,r)=0$
    has the unique solution
    \begin{equation}\label{eq: alpha(r) for k=1}
        \alpha_r := \frac{h}{\sqrt{h^2 + 2\beta^2r^2}}
    \end{equation}
    which maximizes $\tilde{\mathcal{B}}(\cdot ,r)$. This implies
    \begin{equation}
        \sup_{\alpha \in [-1,1]} \tilde{\mathcal{B}}(\alpha,r) = \tilde{\mathcal{B}}(\alpha_r,r) = r \sqrt{h^2 + 2\beta^2 r^2} + g(r),
    \end{equation}
    (also when $r=0$ since then all three expressions are identically $\beta^2/2$). With the change of variables $q = r^2$ we get
    \begin{equation}
        \tilde{\mathcal{B}}(\alpha_r,r) = \mathscr{B}(q) 
        :=  t(q) + g(\sqrt{q})
    \end{equation}
    where
    $$ t(q)=\sqrt{h^2 q + 2\beta^2 q^2},$$
    and
    \begin{equation}\label{eq: g as func of q}
        g(\sqrt{q})=\frac{1}{2} \log(1-q)+\frac{\beta^2}{2} (1-q)^2.
    \end{equation}
    We have thus proved \eqref{eq: tap linear}.
    
    Furthermore we have
    $$t'(q)=\frac{h^2+4\beta^2 q}{2 \sqrt{h^2 q + 2 \beta^2 q^2}} \text{ and }t''(q)=\frac{2 \beta^2}{\sqrt{h^2 q + 2\beta^2 q^2}}-\frac{(h^2 + 4 \beta^2 q)^2}{4 (h^2 q + 2\beta^2 q^2)^{\frac{3}{2}}}.$$
    Since $2\beta^{2}(h^{2}q+2\beta^{2}q^{2})<(h^{2}+4\beta^{2}q)^{2}$ for all $q\in[0,1]$ one sees that $t''(q)<0$, so $t$ is strictly concave. Also
    $$\frac{\partial}{\partial q} g(\sqrt{q}) = -\beta^2 (1-q) -\frac{1}{2(1-q)} \ \text{ and } \ \frac{\partial^2}{\partial q^2} g(\sqrt{q}) = \beta^2 - \tfrac{1}{2(1-q)^2},$$
    and the latter is negative for $q\in(q_P ,1)$, so 
    \begin{equation}\label{eq: g of q conc}
        q\to g(\sqrt{q})\text{ is strictly concave in } [q_p,1].
    \end{equation}
    Thus also $\mathscr{B}(q)$ is strictly concave in $[q_P ,1)$. This implies that $\tilde{\mathcal{B}}(q)$ has a unique maximizer $\hat{q}$ in $[q_P ,1)$, and $\hat{r}=\sqrt{\hat{q}}$ is the unique maximizer of $r\to \tilde{\mathcal{B}}(\alpha_r,r)$ in ${\rm Plef}\left(\beta\right)$, and $(\sqrt{\hat{q}},\frac{h}{\sqrt{h^2+2\beta^2 \hat{q}}})$ is the unique maximizer of the l.h.s. of \eqref{eq: tap linear}.
    
    Thus it only remains to derive the equation \eqref{eq: TAP linear qhat} for $\hat{q}$. For this it suffices to note that with $v\left(x\right)=x+x^{-1}$
we have the identities
\begin{equation}\label{eq: t prime and g prime}
t'\left(q\right)=\frac{\beta}{\sqrt{2}}v\left(\frac{\sqrt{2q}\beta}{\sqrt{h^{2}+2q\beta^{2}}}\right)\text{ and }\frac{\partial}{\partial q}g(\sqrt{q})=-\frac{\beta}{\sqrt{2}} v\left(\sqrt{2}\beta\left(1-q\right)\right).
\end{equation}
Therefore the critical point equation $\mathscr{B}'\left(q\right)=0$
is equivalent to $v\left(\frac{\sqrt{2q}\beta}{\sqrt{h^{2}+2q\beta^{2}}}\right)=v(\sqrt{2}\beta\left(1-q\right))$ and since $v\left(x\right)$ is a bijection for $x\in\left[0,1\right]$ this is in turn equivalent to $\frac{\sqrt{2q}\beta}{\sqrt{h^{2}+2q\beta^{2}}}=\sqrt{2}\beta\left(1-q\right)$
and \eqref{eq: TAP linear qhat}. Since a solution to \eqref{eq: TAP linear qhat} always exists a unique critical point always exists in $(q_P,1)$, and by concavity it is the unique local and global maximum.
\end{proof}
For the cases $k\ge2$ the following fact will be useful.
\begin{lemma}\label{lem: L at alpha=0 strictly decreasing} For all $f,\beta,h$ it holds that $\tilde{\mathcal{B}}(0,r)$ is strictly decreasing in $r$.
\end{lemma}
\begin{proof}
    We have
    \begin{equation}
        \tilde{\mathcal{B}}(0,r) = \sqrt{2} \beta r^2 + \frac{\beta^2}{2} (1-r^2)^2 + \frac{1}{2} \log(1-r^2), 
    \end{equation}
    and
    \begin{equation}\label{eq: L deriv in alpha at r=0}
        \partial_r \tilde{\mathcal{B}}(0,r) = - 2 r \left( \beta^2 (1-r^2) - \sqrt{2}\beta + \frac{1}{2(1-r^2)} \right)
        =
        - 2 r \left(\beta\sqrt{1-r^2} - \tfrac{1}{\sqrt{2(1-r^2)}}\right)^2
        \le 0,
    \end{equation}
    with equality only at a single point, implying the claim.
\end{proof}

We are now ready to study the case $k=2$. Define for $\beta > 0$
\begin{equation}\label{eq: F beta def}
    \mathcal{F}(\beta) 
    = \sup_{r \in [\sqrt{q_P},1)} \tilde{\mathcal{B}}(0,r) =\tilde{\mathcal{B}}(0,\sqrt{q_P}) =
    \begin{cases}
        \frac{\beta^2}{2}  & \ \text{ for } \beta \le \frac{1}{\sqrt{2}},
        \\
        \sqrt{2}\beta-\frac{3}{4}-\frac{1}{2}\log(\sqrt{2}\beta) & \ \text{ for } \beta \ge \frac{1}{\sqrt{2}}.
    \end{cases}
\end{equation}
\begin{lemma}[TAP maximizer with quadratic spike]\label{lem: k=2 maximum on ball}
    Let $f(x)=hx^2$. If $h>\frac{1}{2}$ and $\beta < \sqrt{2}h$ then
    \begin{equation}\label{eq: maximum at k=2}
        \sup_{r\in\rm{Plef}(\beta),\alpha\in [-1,1]} \tilde{\mathcal{B}}(\alpha,r)
        =
        \frac{\beta^2}{8h^2}(4h-1) + h - \frac{1}{2}\left(1 + \log(2h)\right),
    \end{equation}
    and the unique maximizers of the l.h.s. are
    \begin{equation}
        \left(\sqrt{1-\tfrac{1}{2h}}, \pm \sqrt{1-\tfrac{\beta^2}{2h^2}}\right).
    \end{equation}
    If either $h \le \frac{1}{2}$ or $\beta \ge \sqrt{2}h$ then
    \begin{equation}\label{eq: maximum at k=2 border}
        \sup_{r\in\rm{Plef}(\beta),\alpha\in [-1,1]} \tilde{\mathcal{B}}(\alpha,r)
        = \mathcal{F}(\beta),
    \end{equation}
    where the maximum is attained at $(q_P,0)$ (uniquely if $\beta>\frac{1}{\sqrt{2}}$ and otherwise also on $\{0\}\times[0,1]$).
\end{lemma}
\begin{proof} We first maximize in $\alpha$ for fixed $r$. The critical point equation in $\alpha$ for $r$ fixed is
    \begin{equation}
        r^2\left(2h\alpha - \sqrt{2}\beta \frac{\alpha}{\sqrt{1-\alpha^2}}\right) = 0.
    \end{equation}
    Thus when $r\ne0$ the only critical points are $\alpha=0$  and
    if $\beta < \sqrt{2}h$ also
    \begin{equation}\label{eq: alpha(r) for k=2}
        \alpha_r = \pm\sqrt{1 - \frac{\beta^2}{2 h^2}}.
    \end{equation}
    Note that if $\beta < \sqrt{2}h$ and $r\ne0$ we also have
    $$
        \tilde{\mathcal{B}}(\alpha_r,r) - \tilde{\mathcal{B}}(0,r)
        = \left(\frac{2h^2 + \beta^2}{2h} - \sqrt{2}\beta\right) r^2
        = \frac{r^2}{2h} \left(\sqrt{2}h - \beta\right)^2
        > 0,
    $$
    so that the maximizing $\alpha$ for fixed $r\neq 0$ is
    $$
        \alpha = 
        \begin{cases}
            0 & \text{ if } \beta \ge \sqrt{2}h,
            \\
            \pm\sqrt{1 - \frac{\beta^2}{2 h^2}} & \text{ if } \beta < \sqrt{2}h.
        \end{cases}
    $$
    Thus with $q = r^2$ and recalling \eqref{eq: g as func of q} we have
    \begin{equation}\label{eq: tilde L(q) quadratic}
         \sup_{\alpha\in [-1,1]} \tilde{\mathcal{B}}(\alpha,r) = 
         \begin{cases}
            \mathscr{B}(q) :=  \frac{2h^2 + \beta^2}{2h} q + g(\sqrt{q})
             & \text{ if } \beta < \sqrt{2}h,
            \\
            \tilde{\mathcal{B}}(0,\sqrt{q})
             & \text{ if } \beta \ge \sqrt{2}h.
         \end{cases}
    \end{equation}
    If $\beta \ge \sqrt{2}h$ all claims thus follow by Lemma \ref{lem: L at alpha=0 strictly decreasing} and \eqref{eq: F beta def}.
    
    If $\beta < \sqrt{2}h$, since the first term $\mathscr{B}\left(q\right)$ is linear \eqref{eq: g of q conc} implies that $\mathscr{B}\left(q\right)$ is
    strictly concave in $q\in[q_{P},1)$, and so it has a unique maximizer.  
    Note that
    \[
    \mathscr{B}'\left(q\right)=\frac{2h^{2}+\beta^{2}}{2h}+\frac{\partial}{\partial q}g\left(\sqrt{q}\right)=\frac{\beta}{\sqrt{2}}\left(v\left(\frac{\beta}{\sqrt{2}h}\right)-s\left(\sqrt{2}\beta\left(1-q\right)\right)\right),
    \]
    recalling the second part of \eqref{eq: t prime and g prime} and the function $v\left(x\right)=x+x^{-1}$ from $(0,1]$
    to $[2,\infty)$ which is an increasing bijection. Therefore $\mathscr{B}'\left(q\right)=0$
    is equivalent to
    \[
    \frac{\beta}{\sqrt{2}h}=\sqrt{2}\beta\left(1-q\right)\iff q=1-\frac{1}{2h}.
    \]
    Now if $h>\frac{1}{2}$,
    we have that $1-\frac{1}{2h}\in\left(q_{P},1\right)$ so that $1-\frac{1}{2h}$
    is a critical point in $(q_{P},1)$ and by concavity it is the unique
    local and global maximum. It is easy to check that $\mathscr{B}\left(1-\frac{1}{2h}\right)$
    equals the r.h.s. of \eqref{eq: maximum at k=2}, completing the proof when $h>\frac{1}{2}$ and $\beta < \sqrt{2}h$. If $h \le \frac{1}{2}$ the maxmizer is $q=\sqrt{q}_P$, since $\mathscr{B}(q)\to-\infty$ for $q\to1$, and $\mathscr{B}(q)=\tilde{\mathcal{B}}(0,\sqrt{q_P})=\mathcal{F}(\beta)$, giving the claims.
\end{proof}
The result on maximizers of $\tilde{\mathcal{B}}$ for monomial $f$ with $k \ge 3$ is less explicit, and the analysis more complicated. We first show that the global maximum of $\tilde{\mathcal{B}}$ on $[0,1]\times[\sqrt{q_p},1]$ is either achieved at a critical point of in the interior $(\sqrt{q_p},1)\times(0,1)$ or at $(\sqrt{q_p},0)$.
\begin{lemma}\label{lem: L max interior or corner}
     For any $f\in C^1([-1,1])$ 
    we have that $\tilde{\mathcal{B}}(\alpha,r)$ for $(\alpha,r) \in [0,1]\times[\sqrt{q_P},1]$ is maximized
    in the interior $(\sqrt{q_P},1)\times(0,1)$
    or at the point $(\alpha,r) = (0,\sqrt{q_P})$.
\end{lemma}
\begin{proof}
    Note that we have $\tilde{\mathcal{B}}(\alpha,1) = -\infty$ and
    \begin{equation}\label{eq: alpha = 1 never maximizing}
        \frac{\partial}{\partial\alpha}\tilde{\mathcal{B}}(\alpha,r)
        = 
        r f'(r \alpha) - \frac{\sqrt{2}\beta r^2}{\sqrt{1-\alpha^2}}
        \longrightarrow
        -\infty 
        \quad \text{ as } \alpha\to 1,
    \end{equation}
    so $(\alpha,r)$ with $r=1$ or $\alpha=1$ can not be maximizers. Lemma \ref{lem: L at alpha=0 strictly decreasing} shows the only possible maximizer with $r\in\left[\sqrt{q_{P}},1\right],\alpha=0$
is $\left(\sqrt{q_{P}},0\right)$. If $\beta\le\frac{1}{\sqrt{2}}$
then $q_{P}=0$, and $\tilde{\mathcal{B}}\left(\alpha,0\right)=f\left(0\right)+g\left(0\right)$
for all $\alpha$, so if a point on the remaining boundary $r=\sqrt{q_{P}},\alpha\in\left[0,1\right]$
is a maximizer then so is $\left(\sqrt{q_{P}},0\right)$.

Lastly if $\beta>\frac{1}{\sqrt{2}}$ then any critical point of
    $$
        \tilde{\mathcal{B}}\left(\alpha, \sqrt{q_P} \right)
        =
        f(\sqrt{q_P} \alpha) + \sqrt{2}\beta q_P \sqrt{1-\alpha^2} + g(\sqrt{q_P})
    $$
    is a solution of
    \begin{equation}\label{eq: critical alpha on plefka border}
       \sqrt{q_P} f'(\sqrt{q_P} \alpha) - \sqrt{2} \beta q_P \frac{\alpha}{\sqrt{1-\alpha^2}} = 0
        \quad\Leftrightarrow\quad 
        f'(\sqrt{q_P} \alpha)  = \frac{\sqrt{2} \beta\sqrt{q_P} \alpha}{\sqrt{1-\alpha^2}}.
    \end{equation}
    However, in any such point the derivative of $\tilde{\mathcal{B}}$ in $r$ is
    \begin{equation}
       \begin{array}{rcl}
        \alpha f'(\sqrt{q_P} \alpha) + 2\sqrt{2}\beta\sqrt{q_P} \sqrt{1-\alpha^2} + g'(\sqrt{q_P})
        &\stackrel{\eqref{eq: critical alpha on plefka border}}{=}&
        \sqrt{2} \beta\sqrt{q_P}
        \left( 
            \frac{ \alpha^2}{\sqrt{1-\alpha^2}}
            +
            2 \sqrt{1-\alpha^2}
        \right) 
            -2\sqrt{2}\beta\sqrt{q_P}
        \\
        & = &
        \sqrt{2}\beta\sqrt{q_P}
        \left( \frac{2-\alpha^2}{\sqrt{1-\alpha^2}}
        - 2
        \right),
        \end{array}
    \end{equation}
    which is equal to zero for $\alpha = 0$ and positive for all $\alpha \in (0,1)$. Therefore, if some $\alpha > 0$ maximizes $\tilde{\mathcal{B}}\left(\alpha,\sqrt{q_P}\right)$ then there are larger values in the neighborhood of that point, and thus $(\sqrt{q_P} , \alpha)$ cannot be a global maximizer.
\end{proof}

Define
\begin{equation}
    h_c(k,\beta) =  
        \begin{cases}
            0  & \text{ for } k = 1,
            \\
            \min\{\tfrac{1}{2}, \tfrac{\beta}{\sqrt{2}}\}  & \text{ for } k = 2,
            \\
            \mathcal{W}(k,\beta)  & \text{ for } k \ge 3,
        \end{cases}
\end{equation}
where
\begin{equation}
    \mathcal{W}(k,\beta)
    =
    \inf_{r \in \rm{Plef}(\beta)}
    \left\{
    \tfrac{\mathcal{F}(\beta)-g(r)-2\beta^2 r^2(1-r^2)}{\left(r \sqrt{1 - 2\beta^2(1-r^2)^2} \right)^k}
    \right\}
    .
\end{equation}
We now show that if $h>h_c(k,\beta)$ for $k\ge 3$ then there is a unique maximizer in the interior $(q_P,1)\times(0,1)$, while for $h < h_c(k,\beta)$ the point $(\sqrt{q_P},0)$ is the unique maximizer.

\begin{lemma}[TAP maximizer with degree $k\ge3$ spike]\label{lem: k>=3 main lemma}
    Let $k \ge 3, \beta > 0,h>0$ and $f(x) = h x^k$. It holds that
    \begin{equation}\label{eq: k=3 sup}
        \sup_{r\in{\rm Plef}\left(\beta\right),\alpha\in\left[-1,1\right]}\tilde{\mathcal{B}}(\alpha,r)=\sup_{r\in{\rm Plef}\left(\beta\right)}\left\{ hr^{k}\left(1-2\beta^{2}\left(1-r^{2}\right)^{2}\right)^{\frac{k}{2}}+2\beta^{2}r^{2}\left(1-r^{2}\right)+g\left(r\right)\right\} .
    \end{equation}
If $h<h_{c}\left(k,\beta\right)$ then the unique maximizer
of the l.h.s. is $(\sqrt{q_{P}},0)$ and the l.h.s. equals $\mathcal{F}(\beta)$,
and if $h>h_{c}\left(k,\beta\right)$ it the unqiue maximizer is $(\hat{r},\sqrt{1-2\beta^{2}\left(1-\hat{r}^{2}\right)^{2}})$
where $\hat{r}$ is the largest of the two solutions of
    \begin{equation}\label{eq: r equation}
        (1-r^2) \left(r^2\left(1 - 2 \beta^2 (1-r^2)^2\right)\right)^{\frac{k-2}{2}}
        = \frac{1}{h k}
    \end{equation}
    in $(\sqrt{q_P},1)$.
\end{lemma}
\begin{proof}
    By Lemma \ref{lem: L max interior or corner} the maximizer of the l.h.s. of \eqref{eq: k=3 sup} is either $(\sqrt{q}_P,0)$ or a critical point of $\tilde{\mathcal{B}}$  in $(1,\sqrt{q}_P)\times(0,1)$. The critical point equations are
    \begin{align}
        0 \ = & \ h k \alpha^k r^{k-1} + 2\sqrt{2} \beta r \sqrt{1-\alpha^2} + g'(r)
        \numberthis\label{eq: k>=3 equation 1}
        \\
        0 \ = & \ h k \alpha^{k-1} r^{k} - \sqrt{2}\beta r^2 \tfrac{\alpha}{\sqrt{1-\alpha^2}}
        \numberthis\label{eq: k>=3 equation 2}.
    \end{align}
    Any solution to \eqref{eq: k>=3 equation 2} must satisfy $ h k \alpha^k r^{k-1} = r\sqrt{2}\beta r^2 \frac{\alpha^2}{1-\alpha^2}$, and plugging this into \eqref{eq: k>=3 equation 1} we get that any critical point must satisfy
    \begin{equation}\label{eq: k>=3 equation 1 transformed}
        \frac{\alpha^2}{\sqrt{1-\alpha^2}} + 2 \sqrt{1-\alpha^2} = c(r), 
    \end{equation}
    where
    $$ c(r)= \frac{g'(r)}{\sqrt{2}\beta r}=\frac{1}{\sqrt{2}\beta(1-r^2)} + \sqrt{2}\beta(1-r^2).$$ 
    The quadratic \eqref{eq: k>=3 equation 1 transformed} in $\alpha^2$ has the solutions $\frac{ -(c(r)^2-4) \pm c(r) \sqrt{c(r)^2-4}}{2}$ which are well-defined since $c(r)>2$ for $r>\sqrt{q}_P$. Since only one is non-negative and using $\sqrt{(x+x^{-1})^2-4}=x^{-1}-x$ for $x\in(0,1)$ we obtain that any critical point must satisfy
    \begin{equation}\label{eq: k>=3 alpha solution}
        \alpha^2 = \frac{c(r) \sqrt{c(r)^2-4}  - (c(r)^2-4)}{2} 
        = {1-2\beta^2(1-r^2)^2}.
    \end{equation}
    The r.h.s. lies in $[0,1]$ for all $\beta > 0$ and $r\in\rm{Plef}(\beta)$.
    Thus
    \begin{equation}\label{eq: k>=3 first result}
         \sup_{r\in{\rm Plef}\left(\beta\right),\alpha\in\left[-1,1\right]}\tilde{\mathcal{B}}(\alpha,r) = \sup_{r \in {\rm Plef}(\beta)} \tilde{\mathcal{B}}(\sqrt{1-2\beta^2(1-r^2)^2},r),
     \end{equation}
    noting that when $r$ is the left-end point $\sqrt{q_P}$ of $\rm{Plef}(\beta)$ the r.h.s. is $\tilde{\mathcal{B}}(0,\sqrt{q}_P)$. The r.h.s. of \eqref{eq: k>=3 first result} equals the r.h.s. of \eqref{eq: k=3 sup}, so \eqref{eq: k=3 sup} is proved.
    
    Next note that
    \begin{equation}\label{eq: W(k,beta) for beta small}
    \begin{array}{rcl}
        &  & \
        \exists r \in (0,1):
        \tilde{\mathcal{B}}(\sqrt{1-2\beta^2(1-r^2)^2},r) 
        > \tilde{\mathcal{B}}(0,\sqrt{q_P}) = \mathcal{F}(\beta)
        \\
        & \Leftrightarrow & \
        \exists r \in (0,1) :
        h  > 
        \tfrac{\mathcal{F}(\beta) - 2\beta^2 r^2 (1-r^2) - g(r)}{\left(r^2 (1-2\beta^2(1-r^2)^2)\right)^{\frac{k}{2}}}
        \\
        & \Leftrightarrow & \
        h  >  \mathcal{W}(k,\beta).
    \end{array}
    \end{equation}
    Thus indeed for $h < h_c(k,\beta)$ the unique maximizer is $(\sqrt{q}_P,0)$. When $h > h_c(k,\beta)$ the maximizer is a critical point $(\hat{r}, \sqrt{1-2\beta^2(1-\hat{r}^2)^2})$ in the interior $(\sqrt{q}_P,1)\times(0,1)$. It remains to characterize this point and prove its uniqueness.
    
    Firstly, plugging \eqref{eq: k>=3 alpha solution} into \eqref{eq: k>=3 equation 2} one sees that any critical point $(\alpha,r)$ of $\tilde{\mathcal{B}}$ and critical point of the expression on the r.h.s. of \eqref{eq: k=3 sup} with $r\in(\sqrt{q}_P,1)$ must satisfy \eqref{eq: r equation}. When $h>h_c(k,\beta)$ there is a local and global maximum, so the equation must have at least one solution.
    Let
    \begin{equation}\label{def: T(q)}
        T(q) =   (1-q) \left(q (1-2\beta^2(1-q)^2)\right)^{\frac{k-2}{2}},
    \end{equation}
    so that the l.h.s. of \eqref{eq: r equation} is $T(r^2)$. Note that $T(q)$ is non-negative for all $q\in (q_P,1)$ and zero for $q \in \{q_P,1\}$. Furthermore
    \[
    \begin{array}{ccl}
    \frac{\partial}{\partial q}\log T\left(q\right) & = &\displaystyle{ -\frac{1}{1-q}+\frac{k-2}{2}\frac{1}{q}-\left(k-2\right)\frac{2\beta^{2}\left(1-q\right)}{1-2\beta^{2}\left(1-q\right)^{2}}}\\
     & = & \displaystyle{\frac{k-2-kq-\left(k-2\right)q\frac{2\beta^{2}\left(1-q\right)^{2}}{1-2\beta^{2}\left(1-q\right)^{2}}}{2q\left(1-q\right)}}\\
     & = & \displaystyle{\frac{k-2-q\left\{ k-\left(k-2\right)\left(\frac{1}{1-2\beta^{2}\left(1-q\right)^{2}}-1\right)\right\} }{2q\left(1-q\right)}}.
    \end{array}
    \]
    Since $\frac{1}{1-2\beta^{2}\left(1-q\right)^{2}}-1$ is negative
    and decreasing in $(q_{P},1)$, we have that the numerator is decreasing.
    Therefore $\frac{\partial}{\partial q}\log T\left(q\right)$ can switch sign only
    once in $\left(q_{P},1\right)$, showing that $T\left(q\right)$ has
    exactly one critical point in $\left(q_{P},1\right)$, so the equation
    \eqref{eq: r equation} has zero, one or two solutions. We have already excluded the possibility of it having zero solutions. Thus the expression on the r.h.s. of \eqref{eq: k=3 sup} has one or two critical points, of which at least one is a local maximum.
    
    To determine the number and type of the critical point(s) it is useful to note that the expression on the r.h.s. of \eqref{eq: k=3 sup} is always decreasing in $r$ in a neighborhood of $q_{P}$. Indeed when $\beta<\frac{1}{\sqrt{2}}$ so that $q_{P}=0$ this
    follows by expanding the expression around $r=0$ as $\frac{\beta^{2}}{2}+(\beta^{2}-\frac{1}{2})r^{2}+O(r^{3})$.
    When $\beta=\frac{1}{\sqrt{2}}$ similarly the expression expands as $\frac{\beta^{2}}{2}-r^{4}+O(r^{5})$.
    When $\beta>\frac{1}{\sqrt{2}}$ we can make the change of variables
    $1-2\beta^{2}\left(1-r^{2}\right)=z$ and expand the expression around $z=0$
    as $\mathcal{F}(\beta)+\frac{1}{2}(1-\sqrt{2}\beta)z+O(z^{2})$,
    which is decreasing in $z$ in neighborhood of $0$ and therefore
    decreasing in $r$ in a neighborhood of $\sqrt{q_{P}}$.
    
    Thus since the expression is decreasing in a neighbourhood of $r=\sqrt{q_{P}}$ the left-most critical point cannot be a local
    maximum. Thus there are two critical points and \eqref{eq: r equation} has two solutions, the
    smaller which corresponds to a local minimum, and the larger of which
    corresponds to a local maximum which is also the global maximum.
\end{proof}

\bigskip

\section{Fluctuations}\label{section: subleading order fluctuations}

In this section we prove Theorem \ref{thm: sphere} (b) and Theorem \ref{thm: ball} (b) about the  fluctuations of $L_{N}$ resp. $\tilde{L}_{N}$. We do so by studying the fluctuations of minimax expressions of the type
$$\sup_y \inf_l h(y,l,s_{\lambda,u}(l)).$$
The next lemma  shows that under the assumptions of Theorem \ref{thm: sphere} (b) and Theorem \ref{thm: ball} (b) the quantities $L_{N}$ and $\tilde{L}_{N}$ equal such minimax expressions with probability tending to one. Recall $\mathcal{B}(\alpha)$ and $\mathcal{B}(\alpha, r)$ from \eqref{def: B(alpha)} and \eqref{def: tilde B_r}.

\begin{lemma}\label{lem: L_N in terms of h}
    (a) If $\mathcal{B}(\alpha)$ has finitely many global maximizers $\hat{\alpha}_i$, $i=1,...,m$ which are all non-zero then for all $\varepsilon>0$ small enough 
    \begin{equation}\label{eq: high probability equality L_N}
        \lim_{N\to\infty}\mathbb{P}\left(\frac{1}{N}L_{N}=
        \max_{i=1,\ldots,m}\sup_{\alpha\in\left[\hat{\alpha}_i-\varepsilon,\hat{\alpha}_i+\varepsilon\right]}\inf_{l\in[\sqrt{2}+\varepsilon,\varepsilon^{-1}]}h(\alpha,l,s_{\lambda,u}(l))\right)=1,
    \end{equation}
    where
    \[          
        h(\alpha,l,g)=f(\alpha)+\beta\left(l-\frac{\alpha^{2}}{g}\right).
    \]
    
    (b) If $\tilde{\mathcal{B}}(\alpha,r)$ has finitely many global maximizers $(\hat{\alpha}_i,\hat{r}_i)$, $i=1,...,m$, all lying in the interior $[-1,1]\times\mathcal{R}$ with $\hat{\alpha}_i ,\hat{r}_i \neq 0$, then for all $\varepsilon>0$ small enough
    \begin{equation}\label{eq: high probability equality tilde L_N}
        \lim_{N\to\infty}\mathbb{P}\left(
        \frac{1}{N}\tilde{L}_{N}=
        \max_{i=1,...,n}\sup_{\alpha\in\left[\hat{\alpha}_i-\varepsilon,\hat{\alpha}_i+\varepsilon\right],r\in[\hat{r}_i-\varepsilon,\hat{r}_i+\varepsilon]}\inf_{l\in[\sqrt{2}+\varepsilon,\varepsilon^{-1}]}h((\alpha,r),l,s_{\lambda,u}(l))
    \right)=1,
    \end{equation}
    where
    \[
        h((\alpha,r),l,g)=f(\alpha r)+g(r)+\beta r^{2}\left(l-\frac{\alpha^{2}}{g}\right).
    \]
\end{lemma}
\begin{proof}
    By \eqref{eq: fix overlap} we have
    \begin{equation*}
        \frac{1}{N}L_N
        =
        \sup_{\alpha \in [-1,1]} \left\{
        f(\alpha) + \beta
        \sup_{\substack{|\sigma| = 1 \\ \sigma\cdot u = \alpha}} \sum_{i=1}^N \lambda_i \sigma_i^2
        \right\}
    \end{equation*}
    and by Proposition \ref{prop:max_slice_term}
    \begin{equation*}
        f(\alpha) + \beta
        \sup_{\substack{|\sigma| = 1 \\ \sigma\cdot u = \alpha}} \sum_{i=1}^N \lambda_i \sigma_i^2
        =
        \mathcal{B}(\alpha) + o_{\mathbb{P}}(1)
    \end{equation*}
    for all $\alpha\in[-1,1]$ uniformly, so for any $\varepsilon>0$ a global maximizer $\alpha^*$ of the l.h.s. must lie in a $\varepsilon$-neighborhood of one of the $\hat{\alpha}_i\neq 0$ with probability tending to $1$. 
    Thus by Lemma~\ref{lem: simplify}
    \begin{equation*}
        \frac{1}{N}L_N = \max_{i=1,...,m}\sup_{\alpha\in\left[\hat{\alpha}_i-\varepsilon,\hat{\alpha}_i+\varepsilon\right]}\inf_{l>\lambda_{N}}\left\{ l-\frac{\alpha^{2}}{  \GN{\lambda}{l} } \right\},
    \end{equation*}
    with probability tending to one. Since $f\in C^1([-1,1])$ and the derivative of $\sqrt{2(1-\alpha^2)}$  diverges for $\alpha^2\to 1$, neither $1$ nor $-1$ can be a maximizer, so the $\alpha^*_i$ are bounded away from $\pm1$ with probability tending to one. By Lemma \ref{lem: l minimizer epsilon distance} the minimizer in $l$ of $h(\hat{\alpha}_i,l,s_{\lambda,u}(l))$ must lie in $[\sqrt{2}+\varepsilon,\varepsilon^{-1}]$ with probability tending to one for each $i$, after possibly decreasing $\varepsilon$, proving (a).
    
    The claim (b) follows similarly using \eqref{eq: fix overlap ball} and Proposition \ref{prop:max_slice_term}.
\end{proof}

\subsection{General minimax optimization involving \texorpdfstring{$s_{\lambda,u}$}{TEXT}}

In the rest of the section we will study the fluctuations of $\inf_{y\in\mathcal{Y}}\sup_{l\in\mathcal{L}}h(y,l,s_{\lambda,u}(l))$
under the assumptions that 
\begin{equation}\label{eq:YLcompact}
    \mathcal{Y}\subset\mathbb{R}^{n},\mathcal{L}\subset(\sqrt{2},\infty),\mathcal{G}\text{ are compact with } s(\mathcal{L}) \subset \mathcal{G}^{o}
\end{equation}
(where $A^{o}$ denotes the interior of a set $A$)
\begin{equation}\label{eq:hdiff}
    h:\mathcal{Y}\times\mathcal{L}\times\mathcal{G}\to\mathbb{R}\text{ is three times continuously differentiable,}
\end{equation}
\begin{equation}\label{eq:y_unique_max}
    y\to \mathcal{B}(y)\text{ is uniquely maximized at a }\hat{y}\in\mathcal{Y}^{o}\text{, where }\mathcal{B}(y) = \inf_{l\in\mathcal{L}}h(y,l,s(l))\},
\end{equation}
\begin{equation}\label{eq:l_unique_max}
    l\to h(\hat{y},l,s(l))\text{ is uniquely minimized at a }\hat{l}\in\mathcal{L}^{o},
\end{equation}
\begin{equation}\label{eq:hllpos}
    \partial_{ll}h(\hat{y},l,s(l))|_{l=\hat{l}}>0,
\end{equation}
\begin{equation}\label{eq:Hessynegdef}
    \nabla^2 \mathcal{B}(\hat{y}) \text{ is negative definite.}
\end{equation}

The existence of the derivatives in \eqref{eq:Hessynegdef} is guaranteed by the formula \eqref{eq: y limit} for the specific $h$ from Lemma \ref{lem: L_N in terms of h} (a) (b). It also follows from the other assumptions by the implicit function theorem. The latter argument is included in the following two lemmas, which will be needed also later.
\begin{lemma}\label{lem:use_imp_func_thm}
    Let $n\ge1,A\subset\mathbb{R}^{n}$,
    $\eta>0$ and $t:A\times[-\eta,\eta]\to\mathbb{R}$ be twice continuously
    differentiable. If $\partial_{bb}t(a,b)>0$ for all $a\in A,b\in[-\eta,\eta]$, 
    and $\partial_{b}t(a,-\eta)<0,\partial_{b}t(a,\eta)>0,$ for all $a\in A$
    then, ${\rm argmin}_{b\in [-\eta,\eta]}t(a,b)$ is unique for all $a\in A$
    and $b^{*}(a)={\rm argmin}_{b \in [-\eta,\eta]}t(a,b)$ is continuously
    differentiable in $A$ with
    \begin{equation}\label{eq: gradient of minimizer}
        \nabla b^{*}(a)=-\frac{\partial_{b}\nabla_{a}t(a,b)}{\partial_{bb}t(a,b)}|_{b=b^*(a)}
    \end{equation}
    for all $a\in A$.
    Furthermore for all $a\in A$
    \begin{equation}\label{eq: gradient of t}
        \nabla_{a}\left\{ t(a,b^{*}(a))\right\} =\left\{ \nabla_{a}t\right\} (a,b^{*}(a)),
    \end{equation}
    and
    \begin{equation}\label{eq: hessian of t}
        \nabla_{a}^{2}\left\{ t(a,b^{*}(a))\right\} = \nabla_{a}^{2}t(a,b^{*}(a))-\frac{\partial_{b}\{ \nabla_{a}t\}(a,b^{*}(a))\left(\partial_{b}\{\nabla_{a}t\}(a,b^{*}(a))\right)^{T}}{\partial_{bb}t(a,b)}.
    \end{equation}
\end{lemma}
\begin{proof}
    The assumption $\partial_{bb}t(a,b)>0$ implies that ${\rm \text{argmin}}_{b\in [-\eta,\eta]}t(a,b)$
    is unique. Then the assumption $\partial_{b}t(a,-\eta)<0,\partial_{b}t(a,\eta)>0$,
    implies that $b^{*}(a)$ lies in $(-\eta,\eta$) and is the unique
    solution of $\partial_{b}t(a,b)=0$ in this interval. Finally by the
    implicit function theorem applied to $\partial_{b}t(a,b)=0$ the solution
    $b^{*}(a)$ to this equation for $b$ is continuously differentiable
    and satisfies $\nabla b^{*}(a)=-\frac{\partial_{b}\nabla_{a}t(a,b)}{\partial_{bb}t(a,b)}$,
     using again that $\partial_{bb}t(a,b)>0$. 
     Furthermore
     \begin{equation}\label{eq: gradient of t calc}
        \nabla_{a}\left\{ t(a,b^{*}(a))\right\}
        =
        \left\{ \nabla_{a}t\right\} (a,b^{*}(a))
        + \underbrace{\partial_b t(a,b^*(a))}_{=0}
    \end{equation}
    for all $a\in\mathcal{A}$, which shows \eqref{eq: gradient D(y) in h terms}. By taking the derivative of \eqref{eq: gradient of t calc} one obtains 
    \begin{equation*}
        \nabla_{a}^{2}\left\{ t(a,b^{*}(a))\right\} = 
        \nabla_{a}^{2}t(a,b^{*}(a))- \partial_{b}\{ \nabla_{a}t\}(a,b^{*}(a)) \ \nabla b^{*}(a)^T
    \end{equation*}
    and by using \eqref{eq: gradient of minimizer} this shows \eqref{eq: hessian of t}.
\end{proof}
\bigskip

Applied to $(y,l)\to h(y,l,s(l))$ the lemma yields that $\mathcal{B}(y)$ is differentiable in a neighborhood and the following relation between the derivatives of $\mathcal{B}(y)$ and the derivatives of $h(y,l,s(l))$.

\begin{lemma}\label{lem: derivatives of B(y)}
    Assume \eqref{eq:YLcompact}-\eqref{eq:Hessynegdef}.
    Then there is a neighborhood $\mathcal{U}$ of $\hat{y}$ such that $\inf_{l\in\mathcal{L}} h(y,l,s(l))$ is uniquely maximized at a $\hat{l}(y)$ for $y\in \mathcal{U}$, $\mathcal{B}$ from \eqref{eq:y_unique_max} is three times continuously differentiable in $\mathcal{U}$, and for all $y\in\mathcal{U}$
    \begin{equation}\label{eq: gradient D(y) in h terms}
        \nabla\mathcal{B}(y)
        =
        \nabla_y h(y,\hat{l}(y),s(\hat{l}(y)))
    \end{equation}
    and 
    \begin{equation}\label{eq: hessian D(y) in h terms}
        \nabla^2\mathcal{B}(y)
        =
        \nabla_y^2 h(y,\hat{l}(y),s(\hat{l}(y))) 
        - \frac{ \partial_l\left\{\nabla_y h(y,l,s(l))\right\} \left(\partial_l\left\{\nabla_y h(y,l,s(l))\right\}\right)^T}{\partial_{ll} h(y,l,s(l))}\big|_{l=\hat{l}(y)}.
    \end{equation}
\end{lemma}
\begin{proof}
    Using \eqref{eq:hdiff} and \eqref{eq:hllpos} it follows that there is a neighborhood $\mathcal{U}$ of $\hat{y}$ and $[\hat{l}-\eta, \hat{l}+\eta]$ of $\hat{l}$ where $\partial_{ll}h(y,l,s(l))>0$ for all $y\in\mathcal{U},l \in [\hat{l}-\eta, \hat{l}+\eta]$, and by \eqref{eq:l_unique_max} one can in addition ensure that  $\partial_{l} h(y,l,s(l))|_{l = \hat{l} - \eta}<0$ and $\partial_{l} h(y,l,s(l))|_{l = \hat{l} + \eta}>0$. By Lemma \ref{lem:use_imp_func_thm} applied to $t(a,b)=h(\hat{y}+a,\hat{l}+b,s(\hat{l}+b))$ one obtains \eqref{eq: gradient D(y) in h terms} and \eqref{eq: hessian D(y) in h terms}. Since all terms on the r.h.s. of \eqref{eq: hessian D(y) in h terms} are  continuously differentiable it follows that $\mathcal{B} \in C^3(\mathcal{U})$.
\end{proof}
\bigskip

\subsection{Fluctuations of \texorpdfstring{$s_{\lambda,u}$}{TEXT} around \texorpdfstring{$s$}{TEXT}}\label{subsec:fluctuations_of_s_lambda_u}
We will calculate the fluctuations of $\sup_{y \in \mathcal{Y}}\inf_{l \in \mathcal{L}}h(\alpha,l,\GN{\lambda}{l})$ by quadratically expanding $h$ around $(\hat{y},\hat{l},s(\hat{l}))$.  To this end we start by studying the fluctuations of $\GNk{k}{\lambda}{l}$ around $\INk{k}{l}$. Note that for all $l\in\mathcal{L}$
\begin{equation}\label{eq:s_lambda_u_decomp}
    s_{\lambda,u}^{(k)}(l) =
    s^{(k)}(l)+\frac{1}{\sqrt{N}}W_{N}^{(k)}(l)+\frac{1}{N}\Lambda_{N}^{(k)}(l)+\frac{1}{N}R_{N}^{(k)}(l),
\end{equation}
where
\begin{equation}\label{def: W_N(l)}
    W_{N}(l)=\frac{1}{\sqrt{N}}\sum_{i=1}^{N}\frac{Nu_{i}^{2}-1}{l-\theta_{i/N}},  
\end{equation}
and
\begin{equation*}
    \Lambda_{N}(l)=\sum_{i=1}^{N}\frac{1}{l-\lambda_{i}}-Ns(l),
\end{equation*}
as well as
\begin{equation*}
    R_{N}(l)=\sum_{i=1}^{N}\left(Nu_{i}^{2}-1\right)\left(\frac{1}{l-\lambda_{i}}-\frac{1}{l-\theta_{i/N}}\right).
\end{equation*}
We also define
\begin{equation*}
    U_N(l) = \frac{1}{\sqrt{N}} \sum_{i=1}^N \frac{Nu_i^2-1}{l - \lambda_i},
\end{equation*}
which equals $U_N$ from Theorem \ref{thm: sphere} for $l=\frac{2-\hat{\alpha}^{2}}{\hat{z}}$ (with $\hat{\alpha}$ and $\hat{z}$ as in the theorem), recalling from below \eqref{eq: LN in diagonalizing basis} that $u$ is the vector $v$ in the diagonalizing basis of $J$ and $\lambda_i$ are the eigenvalues of $\frac{1}{\sqrt{N}} J_N$. The derivative $U'_N(l)$ for $l=\frac{2-\hat{\alpha}^{2}}{\hat{z}}$ also equals $U'_N$ from Theorem \ref{thm: sphere}. Later we will use that
\begin{equation}\label{eq: U_N-W_N=R_N}
    U_N^{(k)}(l)-W_N^{(k)}(l) = \frac{1}{\sqrt{N}}R_N^{(k)}(l).
\end{equation}
The next lemma shows that the error term $R_N^{(k)}$ in \eqref{eq:s_lambda_u_decomp} and \eqref{eq: U_N-W_N=R_N} is small.
\bigskip

\begin{lemma}\label{lem:s_lambda_u_exp_rest_term}
    For all $k$ and $\varepsilon>0$ it holds that $\sup_{l\ge \sqrt{2}+\varepsilon}|R_{N}^{(k)}(l)|=o_{\mathbb{P}}\left(1\right)$.
\end{lemma}
\begin{proof}
    Let $w(l,x) = \frac{1}{l-x}$ and denote by $w^{(k)}(l,x)$ the $k$-th derivative in $l$.
    Let $\delta > 0$ and define the event
    \begin{equation}\label{eq: E event}
        \mathcal{E}_\delta = 
        \left\{
            \sup_{i=1,...,N} |\lambda_i - \theta_{i/N}|
            \le N^{-\frac{2}{3}+\delta}
        \right\},
    \end{equation}
    whose probability converges to one for any choice of $\delta$ by Lemma \ref{lem: rigidity of eigenvalues}, and define the $\sigma$-algebra
    $$
        \sigma_\Lambda = \sigma(\lambda_1,...,\lambda_N).
    $$
    First consider
    $$
        X := 
        \frac{1}{N}\sum_{i=1}^N \left(N\tilde{u}_i^2 -1\right) 
        \left(w^{(k)}(l,\lambda_i)-w^{(k)}(l,\theta_{i/N})\right),
    $$
     where $\tilde{u}_i$ are i.i.d with law $\mathcal{N}(0,1)$ and independent of $J$, as in the proof of Lemma \ref{lem: u-fluc}.
    Then $\mathbb{E}[X|\sigma_\Lambda] = 0$ and
    \begin{equation*}
    \begin{array}{rcl}
        \mathbb{E}[X^2|\sigma_\Lambda] 1_{\mathcal{E}_\delta}
        &=&
        \frac{1}{N^2}
        \mathbb{E}\left[ 
            \left(
                \sum_{i=1}^N \left(N\tilde{u}_i^2 -1\right) 
                \left(w^{(k)}(l,\lambda_i)-w^{(k)}(l,\theta_{i/N})\right)
            \right)^2
            \big | \sigma_\Lambda
        \right] 1_{\mathcal{E}_\delta}
        \\
        &=&
        \frac{2}{N^2} \sum_{i=1}^N  \left(w^{(k)}(l,\lambda_i)-w^{(k)}(l,\theta_{i/N})\right)^2 1_{\mathcal{E}_\delta}
        \\
        & \overset{\eqref{eq: E event}}{\le} &
        \frac{2|w^{(k+1)}|_{\infty}}{N^2} N N^{-\frac{4}{3}+2\delta} 1_{\mathcal{E}_\delta}
         =
        N^{-7/3 + 2 \delta} 1_{\mathcal{E}_\delta},
    \end{array}
    \end{equation*}
    which implies via Chebyshev's inequality that
    \begin{equation*}
    \begin{array}{rcl}
        \mathbb{P}\left(|X| \ge \tfrac{1}{N \log N}\right)
        & = &
        \mathbb{E}\left[\mathbb{P}\left(|X| \ge \tfrac{1}{N \log N}|\sigma_\Lambda \right)\right]
        \\
        & \le &
        \mathbb{E}\left[\mathbb{P}\left(|X| \ge \tfrac{1}{N \log N}|\sigma_\Lambda \right)1_{\mathcal{E}_\delta}\right] + \mathbb{P}(\mathcal{E}^c_\delta)
        \\
        & \le &
        \mathbb{E}\left[
        \mathbb{E}[X^2|\sigma_\Lambda] (N \log N)^2
        1_{\mathcal{E}_\delta}\right] + \mathbb{P}(\mathcal{E}^c_\delta)
        \\
        & \le &
        (\log N)^2 N^{-\frac{1}{3} + 2\delta} + \mathbb{P}(\mathcal{E}^c_\delta).
    \end{array}
    \end{equation*}
    By choosing $\delta < \tfrac{1}{6}$ this probability converges to zero, and so $X = o_{\mathbb{P}}(\tfrac{1}{N})$. \\
    Constructing the vector $u$ via $u = \tilde{u}/|\tilde{u}|$ we then have
    \begin{equation*}
        \begin{array}{rcl}
            && \sum_{i=1}^N {u}_i^2 \left(w^{(k)}(l,\lambda_i)-w^{(k)}(l,\theta_{i/N})\right) - 
            \sum_{i=1}^N \tilde{u}_i^2 \left(w^{(k)}(l,\lambda_i) -w^{(k)}(l,\theta_{i/N})\right)
            {\color{white}\bigg|}
            \\
            &=&
            (1 - |\tilde{u}|^{2})
            \sum_{i=1}^N u_i^2 \left(w^{(k)}(l,\lambda_i)-w^{(k)}(l,\theta_{i/N})\right),
        \end{array}
    \end{equation*}
    and
    \begin{equation*}
    \begin{array}{l}
        \mathbb{P}\left(
        \left|(1 - |\tilde{u}|^2)
        \sum_{i=1}^N u_i^2 \left(w^{(k)}(l,\lambda_i)-w^{(k)}(l,\theta_{i/N})\right)\right|
        \le 
        \frac{1}{N \log N}
        \right)
        \\
        \ge 
        \mathbb{P}\left(
        \left|1 - |\tilde{u}|^2\right| \le \frac{1}{N^{\frac{1}{3}+\delta} \log N}
        , \
        \left|\sum_{i=1}^N u_i^2 \left(w^{(k)}(l,\lambda_i)-w^{(k)}(l,\theta_{i/N})\right)\right|
        \le N^{-\frac{2}{3}+\delta}
        \right)
        \\
        \ge
        \mathbb{P}\left(
        \left|1 - |\tilde{u}|^2\right| \le \frac{1}{N^{\frac{1}{3}+\delta} \log N}\right)
        -
        \mathbb{P}\left(
        \displaystyle{\sup_{i=1,...,N}} |\lambda_i - \theta_{i/N}|
        \ge \tfrac{1}{|w^{(k+1)}|_{\infty}} N^{-\frac{2}{3}+\delta}\right),
    \end{array}
    \end{equation*}
    (if $|w'|_\infty=0$ the claim of the lemma is of course trivial) which for $\delta < \frac{1}{6}$ converges to $1$ by the CLT on the first probability and Lemma \ref{lem: rigidity of eigenvalues} on the second.
\end{proof}
\bigskip

It thus holds that
\begin{equation}\label{eq:s_lambda_u_decomp_o}
    s_{\lambda,u}^{(k)}(l)=s^{(k)}(l)+\frac{1}{\sqrt{N}}W_{N}^{(k)}(l)+\frac{1}{N}\Lambda_{N}^{(k)}(l)+o_{\mathbb{P}}(N^{-1})\text{ uniformly in }l\ge\sqrt{2}+\varepsilon,
\end{equation}
for any $\varepsilon>0$ and $k\in\mathbb{N}$. The next lemma shows that $W^{(k)}(l)$ of \eqref{eq:s_lambda_u_decomp_o} is of order $O_{\mathbb{P}}(1)$ uniformly.

\begin{lemma}\label{lem:U_N_UB}
    For all $k$ and $\varepsilon>0$ it holds that $\sup_{l \ge \sqrt{2}+\varepsilon}|W_{N}^{(k)}(l)|=O_{\mathbb{P}}(1)$.
\end{lemma}
\begin{proof}
    We construct $u$ by setting $u_i = \frac{\tilde{u}_i}{|\tilde{u}|}$ with $\tilde{u}_1,...,\tilde{u}_N \sim \mathcal{N}(0,\tfrac{1}{N})$ i.i.d.. We then have
    \begin{equation}
        \begin{array}{cclcl}
            W_{N}(l) & = & {\displaystyle \frac{1}{\sqrt{N}}\frac{1}{\left|\tilde{u}\right|^{2}}\sum_{i=1}^{N}\frac{N\tilde{u}_{i}^{2}-\left|\tilde{u}\right|^{2}}{l-\theta_{i/N}}} & = & {\displaystyle \frac{1}{\sqrt{N}}\frac{1}{\left|\tilde{u}\right|^{2}}\sum_{i=1}^{N}\left(N\tilde{u}_{i}^{2}-\left|\tilde{u}\right|^{2}\right)\left(\frac{1}{l-\theta_{i/N}}-s_{\theta}(l)\right)}\\
             &  &  & = & {\displaystyle \frac{1}{\sqrt{N}}\frac{1}{\left|\tilde{u}\right|^{2}}\sum_{i=1}^{N}\left(N\tilde{u}_{i}^{2}-1\right)\left(\frac{1}{l-\theta_{i/N}}-s_{\theta}(l)\right)}
        \end{array}
    \end{equation}
    and similarly
    \begin{align*}
        W_{N}^{(k)}(l)
        \numberthis\label{eq: WN in tilde u form}
        =& 
        \frac{1}{|\tilde{u}|^{2}} \frac{1}{\sqrt{N}}\sum_{i=1}^{N}(N\tilde{u}_{i}^{2}-1)
        \left(\frac{k!(-1)^k}{(l-\theta_{i/N})^{k+1}}-\HNk{k}{\theta}{l}\right).
    \end{align*}
    Note that $\frac{1}{|\tilde{u}|^{2}} = 1 + o_{\mathbb{P}}(1)$ and that by using \eqref{eq: Hk Ik estimate} and a CLT we have
    $$\frac{1}{\sqrt{N}}\sum_{i=1}^{N}(N\tilde{u}_{i}^{2}-1)\HNk{k}{\theta}{l} = O_{\mathbb{P}}(1).$$
    Thus it only remains to show that
    $\tfrac{1}{\sqrt{N}}\sum_{i=1}^N \left(N\tilde{u}_i^2 -1 \right) w^{(k)}(\theta_{i/N},l)= {{O}}_{\mathbb{P}}\left(\frac{1}{\sqrt{N}}\right)$, i.e.
    \begin{equation}\label{eq: sums uniform order} 
        \lim_{z\to\infty} \lim_{N \to \infty}
        \mathbb{P}\left( \sup_{l \ge \sqrt{2}+\varepsilon}
        \frac{1}{N}\sum_{i=1}^N \frac{N\tilde{u}^2_i - 1}{(l  - \theta_{i/N})^k} 
        \ge z \right) = 0.
    \end{equation}
    
    Note that for $x\in(0,1)$
    \begin{equation}
        \frac{1}{(1 - x)^k} = \sum_{j=0}^\infty x^j C_j(k)
    \end{equation}
    where $C_j(k) = \frac{k(k+1)\ldots(k+1-j)}{j!}$, so that that we have for $l  \ge \sqrt{2}+\varepsilon$ and all $x \in [-\sqrt{2}-\tfrac{\varepsilon}{2}, \sqrt{2}+\tfrac{\varepsilon}{2}]$
    \begin{align*}
        \frac{1}{N}\sum_{i=1}^N \frac{N\tilde{u}_i^2 - 1}{(l  - \theta_{i/N})^k}
        = &
        \frac{1}{N}\sum_{j=0}^\infty \sum_{i=1}^N \frac{C_j(k) \theta_{i/N}^j}{l ^{j+k}}
            \left(N\tilde{u}_i^2 - 1\right)
            \\
        = &
        \frac{1}{N}\sum_{j=0}^\infty
            \frac{C_j(k)}{l ^k}\left(\frac{\sqrt{2} + \tfrac{\varepsilon}{2}}{l }\right)^{j}
        \sum_{i=1}^N
            \left(N\tilde{u}_i^2 - 1\right) \left(\frac{\theta_{i/N}}{\sqrt{2} + \tfrac{\varepsilon}{2}}\right)^j
            .
    \end{align*}
    Let
    $$
        \psi_N(j) = \frac{1}{\sqrt{N}} \sum_{i=1}^N
            \left(N\tilde{u}_i^2 - 1\right) \left(\frac{\theta_{i/N}}{\sqrt{2} + \tfrac{\varepsilon}{2}}\right)^j.
    $$
    Since
    $$
        C_j(k) = {k+j-1 \choose k-1} \le (j+k)^{k-1} 
        \quad \text{ and } \quad
        \bigg |
            \frac{\sqrt{2} + \frac{\varepsilon}{2}}{l }
        \bigg | < q 
    $$ 
    for some $q \in (0,1)$, there exists some $c_1 = c_1(k,q) > 0$ such that for fixed $k\in\mathbb{N}_0$
    $$
        \frac{C_j(k)}{l ^k}\left(\frac{\sqrt{2} + \tfrac{\varepsilon}{2}}{l }\right)^{j}
        \le c_1 q^j
    $$
    uniformly for all $j$ and $l  > \sqrt{2} + \varepsilon$, and so
    \begin{align*}
        \bigg | 
            \frac{1}{\sqrt{N}}\sum_{i=1}^N \frac{N\tilde{u}_i^2 - 1}{(l  - \theta_{i/N})^k}
        \bigg |
        \le &
        \sum_{j=0}^\infty c_1 q^j |\psi_N(j)|.
    \end{align*}
    We have
    \begin{align*}
        \text{Var}(\psi_N(j)) = 
        \sum_{i=1}^N
        \left(\frac{\theta_{i/N}}{\sqrt{2} + \tfrac{\varepsilon}{2}}\right)^{2j}
        \mathbb{E}\left[(1-N\tilde{u}_i^2)^2\right] 
        =
        \frac{2}{N} \sum_{i=1}^N\left(\frac{\theta_{i/N}}{\sqrt{2} + \tfrac{\varepsilon}{2}}\right)^{2j}
        \le
        2 \left(\frac{\sqrt{2}}{\sqrt{2}+\frac{\varepsilon}{2}}\right)^{2j}.
    \end{align*}
    For any $x \in \mathbb{R}^+$ via Chebyshev's inequality 
    \begin{align*}
        \mathbb{P}\left(
            \exists j\ge 1:
            |\psi_N(j)| \ge x
        \right)
        \le &
        \sum_{j=1}^\infty
        \frac{\text{Var}(\psi_N(j))}{x^2}
        \le
        \frac{1}{x^2} \sum_{j=1}^\infty 2 \left(\frac{\sqrt{2}}{\sqrt{2}+ \frac{\varepsilon}{2}}\right)^{2j}
        \le
        \frac{c_2}{x^2}
        \numberthis\label{eq: psi finite}
    \end{align*}
    for some $c_2 = c_2(\varepsilon) >0$.  
    Thus the probability in \eqref{eq: sums uniform order} 
    is bounded from above by
    \begin{align*}
        \mathbb{P}\left(
            \sum_{j=0}^\infty c_1 q^j |\psi_N(j)| \ge z
        \right)
        \le & \
        \mathbb{P}\left(
            \sup_{j} |\psi_N(j)| \ge \frac{z}{2 c_1 \sum_{j=0}^\infty q^j}
        \right)
        \stackrel{\eqref{eq: psi finite}}{\le}
        \frac{4 c_1^2 c_2}{(1-q)^2} \frac{1}{z^2}
    \end{align*}
    for all $N$. Taking the limts $N\to\infty$ then $z\to\infty$ completes the proof.
\end{proof}
\bigskip

The following lemma shows that $\Lambda_N^{(k)}(l)$ from \eqref{eq:s_lambda_u_decomp_o} is of order $O_{\mathbb{P}}(1)$ for fixed $l$, and the suboptimal but sufficient bound $O_{\mathbb{P}}(N^{\frac{2}{5}})$ uniformly in $l$.

\begin{lemma}\label{lem:lambda_N_UB}
    For all $k$ and $l$ it holds that $|\Lambda_{N}^{(k)}(l)|=O_{\mathbb{P}}(1)$, and $\sup_{l \ge \sqrt{2}+\varepsilon}|\Lambda_{N}^{(k)}(l)|=O_{\mathbb{P}}(N^{\frac{2}{5}})$ for all $\varepsilon>0$.
\end{lemma}
\begin{proof}
    Lemma \ref{lem: ev fluctuation} implies that $|\Lambda_{N}^{(k)}(l)|=O_{\mathbb{P}}(1)$. \\
    Let $w(l,x) = \frac{1}{l-x}$ and let $w^{(k)}$ denote the $k$-th derivative in $l$.
    It holds that
    \begin{equation}\label{eq: Lambda rough decomposition}
        \Lambda_{N}^{(k)}(l) = 
        \left(\sum_{i=1}^{N}w^{(k)}\left(l,\lambda_i\right)-\sum_{i=1}^{N}w^{(k)}\left(l,\theta_{i/N}\right)\right)
        +
        \left(\sum_{i=1}^{N}w^{(k)}\left(l,\theta_{i/N}\right)-Ns(l)\right),
    \end{equation}
    where the left most term on the r.h.s. is bounded by
    \begin{equation*}
        N |w^{(k+1)}(l) 1_{\{l\ge\sqrt{2}+\varepsilon\}}|_{\infty} \sup_{i=1,...,N}|\lambda_i-\theta_{i/N}|,
    \end{equation*}
    which is of order $O_{\mathbb{P}}(N^{\frac{2}{5}})$ by Lemma \ref{lem: rigidity of eigenvalues}. The right-most term of \eqref{eq: Lambda rough decomposition} is of order $O_{\mathbb{P}}(1)$ by Lemma \ref{lem: sums}.
\end{proof}

In particular we have from \eqref{eq:s_lambda_u_decomp_o} that
\begin{equation}\label{eq:s_lambda_s_sqrt_N}
    s_{\lambda,u}^{(k)}(l)=s^{(k)}(l) + O_{\mathbb{P}}(N^{-1/2})\text{ uniformly in }l \ge \sqrt{2}+\varepsilon,
\end{equation}
for any $\varepsilon>0$ and $k\in\mathbb{N}$.

\subsection{Quadratic expansion and fluctuations of minimax}

We are now ready to expand $h(y,l,\GN{\lambda}{l})$ quadratically around $(\hat{y},\hat{l},s(\hat{l}))$. To formulate the result one needs to take various partial derivatives
of $h$, such as $\partial_{l}\left\{ \{\partial_{g}h\}(y,l,s(l))\right\} |_{(\hat{y},\hat{l})}$.
To keep the typographical size of expressions manageable we define
the shorthand notation
\begin{equation}
h_{\underset{j\text{ times}}{\underbrace{g\ldots g}}}(y,l,s(l))=\left\{ \partial_{g}^{j}h\right\} (y,l,s(l))\label{eq:shorthand_hg}
\end{equation}
for first taking the $g$ derivative $j$ times and
then substituting $s(l)$ for $g$, and 
\begin{equation}
h_{g\ldots g}=h_{g\ldots g}(\hat{y},\hat{l},s(\hat{l}))\text{ for }j\in\{0,1,\ldots\},\label{eq:shorthand_hg_at_hat}
\end{equation}
for in addition substituting $(\hat{y},\hat{l})$ for $(y,l)$ at
the end. Furthermore for $V=\{l\},V=\{y\}$ or $V=\left\{ l,y\right\} $
the notation 
\begin{equation}
h_{V,g\ldots g}=\nabla_{V}\left\{ h_{g\ldots g}(y,l,s(l))\right\} |_{(\hat{y},\hat{l})}\in\mathbb{R}^{\left|V\right|},\label{eq:shorthand_gradient}
\end{equation}
is the gradient (viewed as a column vector) in some combination of
$l$ and $y$ \emph{after} taking $g$ derivatives and substituting
$s(l)$, evaluated at $(\hat{y},\hat{l})$.
Lastly for $V,V'=\{l\},\{y\}$ or $\left\{ y,l\right\} $
\begin{equation}
h_{V',V,g\ldots g}=\nabla_{V'}\nabla_{V}\left\{ h_{g\ldots g}(y,l,s(l))\right\} |_{(\hat{y},\hat{l})}\in\mathbb{R}^{\left|V\right|\times\left|V'\right|},\label{eq:shorthand_mixed_deriv}
\end{equation}
is a matrix of mixed derivatives in $y,l$ obtained in the same way.
Then e.g. $h_{l,g}=h_{\left\{ l\right\} ,g}=\partial_{l}\{\{\partial_{g}h\}(y,l,s(l))\}|_{(\hat{y},\hat{l})}$,
or $h_{y}=h_{\{y\}}=\left\{ \nabla_{y}h\right\} (\hat{y},\hat{l},s(\hat{l}))\in\mathbb{R}^{n}$
or $h_{y,g}=h_{\{y\},g}=\left\{ \nabla_{y}\partial_{g}\right\} h(\hat{y},\hat{l},s(\hat{l}))\in\mathbb{R}^{n}$.
In the statement and proof below $h,h_{g},h_{gg},h_{l,g}\in\mathbb{R}$,
$h_{\{y,l\},g}\in\mathbb{R}^{n+1}$ (column vector), $h_{y,g}\in\mathbb{R}^{n}$
(column vector) and $h_{\{y,l\},\{y,l\}}\in\mathbb{R}^{(n+1)\times(n+1)}$
(matrix) appear.

Similarly, we write for short
\begin{equation}\label{eq:U_N_lambda_N_at_hatl}
    W_N^{(k)} = W_N^{(k)}(\hat{l}) \quad \text{ and } \quad \Lambda_N^{(k)} = \Lambda_N^{(k)}(\hat{l}) \quad \text{ for } \quad k \in \mathbb{N}.
\end{equation}

We now state the quadratic expansion.

\begin{lemma}\label{lem:h_quad_expansion}
    Let $h,\mathcal{Y},\mathcal{L}$
    be as in \eqref{eq:YLcompact}-\eqref{eq:Hessynegdef}. Writing $\Delta=(y-\hat{y},        l-\hat{l})^T \in \mathbb{R}^{n+1}$ (a column vector)   
    it holds that
    \begin{equation}\label{eq:h_quad_expansion}
        h(y,l,s_{\lambda,u}(l)) = p_N(\Delta) + O_{\mathbb{P}}\left(\left|\Delta\right|^{3}\right) + o_{\mathbb{P}}(N^{-1}),
    \end{equation}
    uniformly in all $(y,l)\in\mathcal{Y}\times\mathcal{L}$, for the random quadratic
    \begin{equation}\label{eq:p_N}
        p_{N}(\Delta)=h+\frac{h_{g}}{\sqrt{N}}A_{N}+\frac{1}{N}C_{N}+\frac{1}{\sqrt{N}}\Delta\cdot V_{N}+\frac{1}{2}\Delta^{T}D\Delta,
    \end{equation}
    where the
    sequences $C_{N},V_N$ of random variables are stochastically bounded and given by
    \begin{equation}\label{eq:delta_v_D}
        C_{N}=\Lambda_N h_{g}+\frac{W_{N}^{2}}{2}h_{gg},
        \quad
        V_{N}=\left(\begin{matrix}h_{y,g} & 0\\
        h_{l,g} & h_{g}
        \end{matrix}\right)\left(\begin{matrix}W_{N}\\
        W_{N}^{'}
        \end{matrix}\right)\in\mathbb{R}^{n+1},\quad D=h_{\{y,l\},\{y,l\}}\in\mathbb{R}^{(n+1)\times(n+1)}.
    \end{equation}
\end{lemma}

\begin{proof}
    We start by Taylor expanding in $s_{\lambda,u}(l)-s(l)$ and obtain
    \begin{equation}\label{eq:h_exp_in_g}
    \begin{array}{rcl}
        h(y,l,s_{\lambda,u}(l))
        &=&h(y,l,s(l))+\partial_{g}h(y,l,s(l))(s_{\lambda,u}(l)-s(l))
        \\
        &&+\frac{1}{2}\partial_{gg}h(y,l,s(l))\left(s_{\lambda,u}(l)-s(l)\right)^{2}+O\left(\left|s_{\lambda,u}(l)-s(l)\right|^{3}\right),
    \end{array}
    \end{equation}
    where we used \eqref{eq:hdiff} and therefore the constant in the $O$ term depends on $h,\mathcal{Y},\mathcal{L}$
    (as in several estimates below). Using \eqref{eq:s_lambda_u_decomp_o} and Lemmas \ref{lem:U_N_UB} and \ref{lem:lambda_N_UB} it follows that
    \begin{equation}\label{eq:h_exp_in_g_A_B}
        \begin{array}{rcl}
            h(y,l,s_{\lambda,u}(l))&=&
            h(y,l,s(l))+\partial_{g}h(y,l,s(l))\left(\frac{1}{\sqrt{N}}W_{N}(l)+\frac{1}{N}\Lambda_{N}(l)\right)
            \\
            &&+\frac{1}{2}\partial_{gg}h(y,l,s(l))\frac{1}{N}W_{N}(l)^{2}+o_{\mathbb{P}}\left(N^{-1}\right),
        \end{array}
    \end{equation}
    uniformly in $l\in\mathcal{L}$.
    
    Next we Taylor expand $h(y,l,s(l))$ around $(\hat{y},\hat{l})$,
    giving with the shorthand notation (\ref{eq:shorthand_hg_at_hat}) 
    \begin{equation}\label{eq:h_exp_in_y_l}
        h(y,l,s(l))=h+h_{\{y,l\}}\cdot\Delta+\frac{1}{2}\Delta^{T}h_{\{y,l\},\{y,l\}}\Delta+O(\left|\Delta\right|^{3}).
    \end{equation}
    Note that $h_l = \partial_l \{ h(y,l,s(l)) \}|_{y=\hat{y},l=\hat{l}} = 0$ by \eqref{eq:l_unique_max}, and $h_y = \nabla\mathcal{B}(\hat{y}) =0$ by \eqref{eq:y_unique_max} and \eqref{eq: gradient D(y) in h terms}, so 
    \begin{equation}\label{eq:h_grad_at_hats}
        h_{\{y,l\}} = (h_y, h_l)=0.
    \end{equation}
    Similarly Taylor
    expanding $\partial_{g}h(y,s(l))$ and $\partial_{gg}h(y,s(l))$ around
    $(\hat{y},\hat{l})$ gives 
    \begin{equation}\label{eq:hg_exp_in_y_l}
        \partial_{g}h(y,l,s(l))=h_{g}+h_{\{y,l\},g}\cdot\Delta+O(\left|\Delta\right|^{2})\text{ and }\partial_{gg}h(y,l,s(l))=h_{gg}+O(\left|\Delta\right|).
    \end{equation}
    Finally Taylor expanding $W_{N}(l)$ around $\hat{l}$ and using Lemma \ref{lem:U_N_UB} gives that
    \begin{equation}
        W_{N}(l) = W_{N} + W_{N}' \Delta_l + {{O}}_{\mathbb{P}}(|\Delta^2|),
    \end{equation}
    (recall \eqref{eq:U_N_lambda_N_at_hatl}) uniformly in $l\in\mathcal{L}$, and using Lemma \ref{lem:lambda_N_UB} that 
        $\Lambda_{N}(l) = \Lambda_{N}+O_{\mathbb{P}}(\left|\Delta\right| + \left|\Delta\right|^2 N^{\frac{2}{5} })$, so that
    \begin{equation}\label{eq:B_N_exp}
        \Lambda_{N}(l)
        =
        \Lambda_{N}+O_{\mathbb{P}}(\left|\Delta\right| + N \left|\Delta\right|^2 N^{-\frac{1}{2} })
    \end{equation}
    
    Combining (\ref{eq:h_exp_in_g_A_B})-(\ref{eq:B_N_exp}) and noting
    that $\left|\Delta\right|^{a}(N^{-1/2})^{b}=O(\left|\Delta\right|^{3}+N^{-3/2})$
    for $a+b\le3$ we obtain (\ref{eq:h_quad_expansion}).
\end{proof}

\bigskip

The next lemma computes the minimax of $p_N(\Delta)$ from \eqref{eq:p_N}.

\begin{lemma}\label{lem:max_fluct_quad}
    For any $h$ satisfying \eqref{eq:YLcompact}-\eqref{eq:Hessynegdef}
    there exist constants $E_{1},E_{2}$ and a stochastically bounded
    sequence of random variables $F_{N}$ such that $p_N$ from \eqref{eq:p_N} a.s. satisfies   \begin{equation}\label{eq:max_fluct_quad}  
        \sup_{y\in\mathbb{R}^n}\inf_{l\in\mathbb{R}}p_N(\Delta_y,\Delta_l) =E_{1}+\frac{1}{\sqrt{N}}E_{2}W_{N}+\frac{1}{N}F_{N}.
    \end{equation}
    
    Furthermore, $E_{1},E_{2},F_{N}$ are explicit in terms of the derivatives
    of $h$ at $\hat{y},\hat{l}$ and equal
    \begin{equation}\label{eq:E1E2E3}
        E_{1}=h=h(\hat{y},s(\hat{l}))=\mathcal{B}(\hat{y}),\quad\quad E_{2}=h_{g},\quad\quad E_{3}=h_{gg},\quad(E_{1},E_{2},E_{3}\in\mathbb{R}),
    \end{equation}
    and
    \begin{equation}\label{eq:FN}
        F_{N}=E_{2}\Lambda_{N}-\frac{1}{2}\left(\begin{matrix}W_{N}\\
        W_{N}^{'}
        \end{matrix}\right)^{T}G\left(\begin{matrix}W_{N}\\
        W_{N}^{'}
        \end{matrix}\right)\in\mathbb{R},
    \end{equation}
    where
    \begin{equation}\label{eq:GHJ}
        G=H-\left(\begin{matrix}E_{3} & 0\\
        0 & 0
        \end{matrix}\right)\in\mathbb{R}^{2\times2},\quad H=K^{T}JK+\frac{ww^{T}}{h_{l,l}}\in\mathbb{R}^{2\times2},\quad J= \nabla^{2}\mathcal{B}(\hat{y})\in\mathbb{R}^{n\times n},
    \end{equation}
    \begin{equation}\label{eq:KLw}
        K=L-\frac{h_{l,y}w^{T}}{h_{l,l}}\in\mathbb{R}^{n\times2},\quad w=\left(\begin{matrix}h_{l,g}\\
        E_{2}
        \end{matrix}\right)=\left(\begin{matrix}h_{l,g}\\
        h_{g}
        \end{matrix}\right)\in\mathbb{R}^{2\times1},\quad L=\left(\begin{matrix}h_{y,g} & 0\end{matrix}\right)\in\mathbb{R}^{n\times2},
    \end{equation}
    where we view $w$ as a column vector, and recall from (\ref{eq:shorthand_hg})-(\ref{eq:shorthand_gradient})
    that $h,h_{g},h_{gg},h_{l,l},h_{l,g}\in\mathbb{R}$ are scalars, that
    $h_{y,g},h_{l,y}\in\mathbb{R}^{n\times1}$ are column vectors and
    $h_{y,l}\in\mathbb{R}^{1\times n}$ is a row vector.
\end{lemma}

\begin{proof}
    The expressions $D$
    and $V_N$ from (\ref{eq:delta_v_D}) can be written as
    \begin{equation}
    D=h_{\{y,l\},\{y,l\}}=\left(\begin{matrix}h_{y,y} & h_{l,y}\\
    h_{y,l} & h_{l,l}
    \end{matrix}\right)\text{ and }V_{N}=\left(\begin{matrix}V_{y,N}\\
    V_{l,N}
    \end{matrix}\right)\overset{\eqref{eq:KLw}}{=}\underset{\in\mathbb{R}^{(n+1)\times2}}{\underbrace{\left(\begin{matrix}\begin{array}{c}
    L\end{array}\\
    w^{T}
    \end{matrix}\right)}}\left(\begin{matrix}W_{N}\\
    W_{N}^{'}
    \end{matrix}\right),\label{eq:Dv}
    \end{equation}
     (for $V_{y,N}\in\mathbb{R}^{n}$ and $V_{l,N}\in\mathbb{R}$; note
    that $h_{y,l}=h_{l,y}^{T}$) and $\Delta$ from (\ref{eq:delta_v_D})
    as $\Delta=\left(\begin{matrix}\Delta_{y}\\
    \Delta_{l}
    \end{matrix}\right)$ for $\Delta_{l}\in\mathbb{R}$ and row vector $\Delta_{y}\in\mathbb{R}^{n}$.
    With this notation $p_N$ can be written as
    \[
        p_N(\Delta_{y},\Delta_{l})=X_{N}+\frac{1}{\sqrt{N}}\Delta_{y}\cdot v_{y}+\frac{1}{\sqrt{N}}\Delta_{l}v_{l}+\frac{1}{2}\Delta_{y}^{T}h_{y,y}\Delta_{y}+\Delta_{l}h_{y,l}\Delta_{y}+\frac{1}{2}\Delta_{l}^{2}h_{l,l},
    \]
    for (recalling (\ref{eq:E1E2E3}))
    \begin{equation}\label{eq:XN_def}
        X_{N}=E_{1}+\frac{E_{2}}{\sqrt{N}}W_{N}+\frac{1}{N}C_{N}.
    \end{equation}
    Collecting the terms involving $\Delta_{l}$
    we can furthermore write
    \begin{equation}\label{eq:quad_in_delta}
        p_N(\Delta_{y},\Delta_{l})=X_{N}+\frac{1}{\sqrt{N}}\Delta_{y}\cdot V_{y,l}+\frac{1}{2}\Delta_{y}^{T}h_{y,y}\Delta_{y}+\Delta_{l}\left(\frac{1}{\sqrt{N}}V_{l,N}+\Delta_{y}^{T}h_{l,y}\right)+\frac{1}{2}\Delta_{l}^{2}h_{l,l}.
    \end{equation}
    Recalling that $h_{l,l}$ is positive by \eqref{eq:hllpos} the quadratic
    $\Delta_{l}\to p_N(\Delta_{y},\Delta_{l})$ with $\Delta_{y}$ fixed
    is minimized by 
    \begin{equation}
        \hat{\Delta}_{l}=-\frac{1}{h_{l,l}}\left(\frac{1}{\sqrt{N}}V_{l,N}+\Delta_{y}^{T}h_{l,y}\right),\label{eq:hat_delta_l}
    \end{equation}
    and plugging this into (\ref{eq:quad_in_delta}) gives 
    \begin{equation}\label{eq:minimzed_in_l}
        \begin{array}{lcl}
            p_N(\Delta_{y}):=p_N(\Delta_{y},\hat{\Delta}_{l}) & = & X_{N}+\frac{1}{\sqrt{N}}\Delta_{y}^{T}V_{y,N}+\frac{1}{2}\Delta_{y}^{T}h_{y,y}\Delta_{y}-\frac{1}{2}\frac{1}{h_{l,l}}\left(\frac{1}{\sqrt{N}}V_{l,N}+\Delta_{y}^{T}h_{l,y}\right)^{2}\\
             & = & X_{N}+\frac{1}{\sqrt{N}}\Delta_{y}^{T}\left(V_{y,N}-\frac{1}{h_{l,l}}V_{l,N}h_{l,y}\right)+\frac{1}{2}\Delta_{y}^{T}\left(h_{y,y}-\frac{h_{l,y}h_{l,y}^{T}}{h_{l,l}}\right)\Delta_{y}-\frac{1}{2}\frac{1}{N}\frac{1}{h_{l,l}}V_{l,N}^{2}\\
             & = & X_{N}+\frac{1}{\sqrt{N}}\Delta_{y}^{T}K\left(\begin{matrix}W_{N}\\
            W_{N}^{'}
            \end{matrix}\right)+\frac{1}{2}\Delta_{y}^{T}J\Delta_{y}-\frac{1}{2}\frac{1}{N}\frac{1}{h_{l,l}}\left(\begin{matrix}W_{N}\\
            W_{N}^{'}
            \end{matrix}\right)^{T}ww^{T}\left(\begin{matrix}W_{N}\\
            W_{N}^{'}
            \end{matrix}\right),
        \end{array}
    \end{equation}
    where the last representation follows since by Lemma \ref{lem: derivatives of B(y)}
    \[
        J=h_{y,y}-\frac{h_{l,y}h_{l,y}^{T}}{h_{ll}},
    \]
    and
    \[
        V_{y,N}-\frac{1}{h_{ll}}V_{l,N}h_{l,y}\overset{\eqref{eq:KLw},\eqref{eq:Dv}}{=}\left(L-\frac{h_{l,y}w^{T}}{h_{ll}}\right)\left(\begin{matrix}W_{N}\\
        W_{N}^{'}
        \end{matrix}\right)=K\left(\begin{matrix}W_{N}\\
        W_{N}^{'}
        \end{matrix}\right),
    \]
    and 
    \[
        V_{l,N}^{2}\overset{\eqref{eq:KLw},\eqref{eq:Dv}}{=}\left(w^{T}\left(\begin{matrix}W_{N}\\
        W_{N}^{'}
        \end{matrix}\right)\right)^{2}=\left(\begin{matrix}W_{N}\\
        W_{N}^{'}
        \end{matrix}\right)^{T}ww^{T}\left(\begin{matrix}W_{N}\\
        W_{N}^{'}
        \end{matrix}\right).
    \]
    
    We now maximize $p_N(\Delta_{y})$ in $\Delta_{y}$. Recall that $J$
    is negative definite by assumption. It is easily seen that $p_N(\Delta_{y})$
    is maximized by 
    \begin{equation}\label{eq:hat_delta_y}
        \hat{\Delta}_{y}=-\frac{1}{\sqrt{N}}J^{-1}
        \begin{pmatrix}
        W_{N}\\
        W_{N}^{'}
        \end{pmatrix},
    \end{equation}
    and plugging this in yields
    \begin{equation}
        \begin{array}{ccl}
            p_N(\hat{\Delta}_{y}) & = & \begin{array}{c}
            X_{N}-\frac{1}{2}\frac{1}{N}\left(\begin{matrix}W_{N}\\
            W_{N}^{'}
            \end{matrix}\right)^{T}K^{T}J^{-1}K\left(\begin{matrix}W_{N}\\
            W_{N}^{'}
            \end{matrix}\right)-\frac{1}{2}\frac{1}{N}\frac{1}{h_{ll}}\left(\begin{matrix}W_{N}\\
            W_{N}^{'}
            \end{matrix}\right)^{T}ww^{T}\left(\begin{matrix}W_{N}\\
            W_{N}^{'}
            \end{matrix}\right)\end{array}\\
             & \overset{\eqref{eq:GHJ}}{=} & X_{N}-\frac{1}{N}\frac{1}{2}\left(\begin{matrix}W_{N}\\
            W_{N}^{'}
            \end{matrix}\right)^{T}H\left(\begin{matrix}W_{N}\\
            W_{N}^{'}
            \end{matrix}\right)\\
             & \overset{(*)}{=} & E_{1}+\frac{1}{\sqrt{N}}E_{2}W_{N}+\frac{1}{N}F_{N}\quad\quad(*: \text{ by \eqref{eq:delta_v_D}, \eqref{eq:GHJ}, \eqref{eq:XN_def}})
        \end{array}
    \end{equation}
    Thus have we have proved \eqref{eq:max_fluct_quad}.
\end{proof}
\bigskip

The following lemma shows that we can reduce the optimization region $\mathcal{Y}\times\mathcal{L}$ to a small neighborhood of $(\hat{y},\hat{l})$. Let
\[
    \mathcal{Y}(\varepsilon)=\left\{ y\in\mathcal{Y}:|y-\hat{y}|<\varepsilon\right\} \quad\quad\text{ and }\quad\quad\mathcal{L}(\varepsilon)=\left\{ l\in\mathcal{L}:|l-\hat{l}|<\varepsilon\right\}.
\]
\begin{lemma}\label{lem:localize}
    For all $h$ that satisfy \eqref{eq:YLcompact}-\eqref{eq:Hessynegdef},
    and all $\varepsilon_{1}>0$ there is a $\delta=\delta(\varepsilon_{1})$
    such that if $0<\varepsilon_{2}\le\delta$ then
    \[
        \lim_{N\to\infty}\mathbb{P}\left(\sup_{y\in\mathcal{Y}}\inf_{l\in\mathcal{L}}h(y,l,s_{\lambda,u}(l))=\sup_{y\in\mathcal{Y}(\varepsilon_{2})}\inf_{l\in\mathcal{L}(\varepsilon_{1})}h(y,l,s_{\lambda,u}(l))\right)=1.
    \]
\end{lemma}

\begin{proof}
    By the continuity of $h$, the compactness of $\mathcal{L}\backslash\mathcal{L}(\varepsilon_{1})$
    and (\ref{eq:l_unique_max}) it holds for any $\varepsilon_{1}>0$
    that
    \begin{equation}
        \inf_{l\in\mathcal{L}\backslash\mathcal{L}(\varepsilon_{1})}h(\hat{y},l,s(l))>h(\hat{y},\hat{l},s(\hat{l})).\label{eq:hl}
    \end{equation}
    Using uniform continuity of $h$ on the compact $\mathcal{Y}\times(\mathcal{L}\backslash\mathcal{L}(\varepsilon_{1}))$
    there is some $\delta>0$ such that if $0<\varepsilon_{2}\le\delta$
    then in addition
    \begin{equation}\label{eq:h_inf_lb}
        \inf_{l\in\mathcal{L}\backslash\mathcal{L}(\varepsilon_{1})}h(y,l,s(l))>h(y,\hat{l},s(\hat{l}))\text{ for all }y\in\mathcal{Y}(\varepsilon_{2}).
    \end{equation}
    By Lemma \ref{lem: GN IN uniformly o(1)}, \eqref{eq:hdiff} and compactness it follows that $h(y,l,s_{\lambda,u}(l))\to h(y,l,s(l))$ in probability uniformly in $\mathcal{Y}\times\mathcal{L}$, so that \eqref{eq:h_inf_lb} holds with $s_{\lambda,u}$ in place of $s$, with probability tending to one. This implies that \begin{equation}\label{eq:P1}
        \lim_{N\to\infty}\mathbb{P}\left(\inf_{l\in\mathcal{L}}h(y,l,s_{\lambda,u}(l))=\inf_{l\in\mathcal{L}(\varepsilon_{1})}h(y,l,s_{\lambda,u}(l))\text{ for all }y\in\mathcal{Y}(\varepsilon_{2})\right)=1.
    \end{equation}
    Similarly to (\ref{eq:hl}) it also follows from (\ref{eq:y_unique_max})
    that
        \[
        \sup_{y\in\mathcal{Y}\backslash\mathcal{Y}(\varepsilon_{2})}\inf_{l\in\mathcal{L}}h(y,l,s(l))<h(\hat{y},\hat{l},s(\hat{l})),
    \]
    and similary by the uniform convergence of $h(y,l,s_{\lambda,u}(l))$ it
    follows that
    \begin{equation}\label{eq:P2}
        \lim_{N\to\infty}\mathbb{P}\left(\sup_{y\in\mathcal{Y}}\inf_{l\in\mathcal{L}}h(y,l,s_{\lambda,u}(l))=\sup_{y\in\mathcal{Y}(\varepsilon_{2})}\inf_{l\in\mathcal{L}}h(y,l,s_{\lambda,u}(l))\right)=1.
    \end{equation}
    The claim then follows by (\ref{eq:P1}) and (\ref{eq:P2}).
\end{proof}
\bigskip

Let $\partial_{a}$ denote the directional derivative in the direction
of a vector $a$. The next lemma gives conditions under which the optimizer of a minimax is given by a unique critical point.
\begin{lemma}\label{lem:minimax_crit}
    Let $n\ge1,d>0,\eta>0$, $A(d)=\left\{ a\in\mathbb{R}^{n}:\left|a\right|\le d\right\}$
    and $t:A(d)\times[-\eta,\eta]\to\mathbb{R}$ be twice continuously differentiable.
    Assume $\partial_{bb}t(a,b)>0$ for all $a\in A(d),b\in[-\eta,\eta]$,
    and $\partial_{b}t(a,\eta)>0,\partial_{b}t(a,-\eta)<0$ for all $a\in A(d)$,
    and 
    \[
        \lambda_{\max}\left(\nabla_{a}^{2}t(a,b)-\frac{1}{\partial_{bb}t(a,b)}\partial_{b}\nabla_{a}t(a,b)\left(\partial_{b}\nabla_{a}t(a,b)\right)^{T}\right)<0
    \]
    for $a\in A(d),b\in[-\eta,\eta]$ (where $\lambda_{\max}$ denotes the largest eigenvalue), and that $\partial_{a}\left\{ \inf_{b\in[-\eta,\eta]}t(a,b)\right\}$
    exists and is negative for all $a$ with $\left|a\right|=d$. Then
    $t$ has a unique critical point in $A(d)\times[-\eta,\eta]$ and $\sup_{a\in A(d)}\inf_{b\in[-\eta,\eta]}t(a,b)$
    is uniquely achieved at this critical point.
\end{lemma}

\begin{proof}
    This $t$ satisfies the assumptions of Lemma \ref{lem:use_imp_func_thm},
    so the map $a\to b^{*}(a):={\rm argmin}_{b\in[-\eta,\eta]}t(a,b)$
    is well defined and continuously differentiable, and $\nabla_{a}\left\{ t(a,b^{*}(a))\right\} =\left\{ \nabla_{a}t\right\} (a,b^{*}(a))$
    and
    \[
        \nabla_{a}^{2}\left\{ t(a,b^{*}(a))\right\} =\nabla_{a}^{2}t(a,b^{*}(a))-\frac{1}{\partial_{bb}t(a,b)}\partial_{b}\nabla_{a}t(a,b^{*}(a))\left(\partial_{b}\nabla_{a}t(a,b^{*}(a))\right)^{T}\text{ for all }a\in A(d).
    \]
    By assumption this is negative-definite for all $a\in A(d)$, implying
    that if $a\to t(a,b^{*}(a))$ is concave and therefore if not maximized on the boundary of
    $A(d)$, it has a unique critical point in the interior which is the
    maximizer. Since the assumption $\partial_{a}\left\{ \inf_{b\in[-\eta,\eta]}t(a,b)\right\} < 0$
    rules out the maximizer lying on the boundary, and $(a,b)$ is a critical
    point of $t$ iff $b=b^{*}(a)$ and $a$ is a critical point of $a\to t(a,b^{*}(a))$,
    this proves the claim.
\end{proof}

We can now strengthen Lemma \ref{lem:localize}.
\begin{lemma}\label{lem:localize_more}
    It holds that
    \[
        \lim_{N\to\infty}\mathbb{P}\left(\sup_{y\in\mathcal{Y}}\inf_{l\in\mathcal{L}}h(y,l,s_{\lambda,u}(l))=\sup_{y\in\mathcal{Y}\left(\frac{\log N}{\sqrt{N}}\right)}\inf_{l\in\mathcal{L}\left(\frac{(\log N)^{2}}{\sqrt{N}}\right)}h(y,l,s_{\lambda,u}(l))\right)=1.
    \]
\end{lemma}

\begin{proof}
    By  Lemma \ref{lem:localize} there is for for each $\varepsilon_{2}>0$
    small enough an $\varepsilon_{1}>0$ small enough so that 
    \begin{equation}
        \lim_{N\to\infty}\mathbb{P}\left(\sup_{y\in\mathcal{Y}}\inf_{l\in\mathcal{L}}h(y,l,s_{\lambda,u}(l))=\sup_{y\in\mathcal{Y}(\varepsilon_{2})}\inf_{l\in\mathcal{L}(\varepsilon_{1})}h(y,l,s_{\lambda,u}(l))\right)=1.\label{eq:to_eps_nbd}
    \end{equation}
    
    Furthermore, for each $\varepsilon_{2}>0$
    small enough, there is an $\varepsilon_{1}>0$ small enough such that
    \[
        \partial_{ll}\left\{ h(y,l,s(l))\right\} >0, \
        \partial_{l}\left\{ h(y,l,s(l))\right\}\big|_{\hat{l}-\varepsilon_{1}} >0, \
        \partial_{l}\left\{ h(y,l,s(l))\right\}\big|_{\hat{l}-\varepsilon_{1}} <0,
    \]    
    for all $y\in\mathcal{Y}(\varepsilon_{2}),l\in\mathcal{L}(\varepsilon_{1})$ (see \eqref{eq:l_unique_max} and \eqref{eq:hllpos}),
    and
    \begin{equation}
        \lambda_{\max}\left(\nabla_{y}^{2}h(y,l,s(l))-\frac{\partial_{l}\nabla_{y}h(y,l,s(l))(\partial_{l}\nabla_{y}h(y,l,s(l)))^{T}}{\partial_{ll}h(y,l,s(l))}\right )<0,\label{eq:apa}
    \end{equation}
    for all $y\in\mathcal{Y}(\varepsilon_{2}),l\in\mathcal{L}(\varepsilon_{1})$ (see \eqref{eq:Hessynegdef}),
    and since $\hat{y}$ is the unique maximum (see \eqref{eq:y_unique_max})
    \[
    \partial_{(\hat{y}-y)}\left\{ h(y,l,s(l))\right\} >0\text{\,for all }y\text{ s.t.}\left|y-\hat{y}\right|=\varepsilon_{2}.
    \]
    By Lemma \ref{lem: GN IN uniformly o(1)} and \eqref{eq:hdiff} the
    same holds with $s_{\lambda,u}(l)$ in place of $s(l)$, on an event
    with probability tending to one. Therefore by applying Lemma \ref{lem:minimax_crit}
    to $(a,b)\to h(\hat{y}+a,\hat{l}+b,s_{\lambda,u}(\hat{l}+b))$ on this event one obtains that
    \begin{equation}
    \lim_{N\to\infty}\mathbb{P}\left(\begin{array}{c}
    (y,l)\to h(y,l,s_{\lambda,u}(l))\text{\,has a unique critical point }(y^{*},l^{*})\text{\,in }\mathcal{Y}(\varepsilon_{2})\times\mathcal{L}(\varepsilon_{1})\\
    \text{ and }\sup_{y\in\mathcal{Y}(\varepsilon_{2})}\inf_{l\in\mathcal{L}(\varepsilon_{1})}h(y,l,s_{\lambda,u}(l))\text{ is achieved at }(y^{*},l^{*})
    \end{array}\right)=1.\label{eq:minimax_is_crit}
    \end{equation}
    
    By the Schur complement formula and (\ref{eq:apa}) it holds that
    $\nabla_{y,l}^{2}\left\{ h(y,l,s(l))\right\} $ is non-degenerate,
    and $\nabla_{y,l} \{ h(\hat{y},\hat{l},s(\hat{l}))\} = h_{\{y,l\}} = 0$ as stated in \eqref{eq:h_grad_at_hats},
    so for $\varepsilon_{2},\varepsilon_{1}>0$ small enough there is
    a constant $c$ such that
    \[
    \left|\nabla_{y,l}\left\{ h(y,l,s(l))\right\} \right|\ge c(|y-\hat{y}|^{2}+|l-\hat{l}|^{2})\text{\,for all }
    (y,l)\in\mathcal{Y}(\varepsilon_{2})\times\mathcal{L}(\varepsilon_{1}).
    \]
    Since $\left|\nabla h(y,l,s(l))-\nabla h(y,l,s_{\lambda,u}(l))\right|=O_{\mathbb{P}}(N^{-1/2})$
    by \eqref{eq:s_lambda_s_sqrt_N} it follows that
    \begin{equation}
    \lim_{N\to\infty}\mathbb{P}\left(\begin{array}{c}
    h(y,l,s_{\lambda,u}(l))\text{ has no critical point in }\mathcal{Y}(\varepsilon_{2})\times\mathcal{L}(\varepsilon_{1})\\
    \text{ with }\left|y-\hat{y}\right|\ge\frac{\log N}{N^{1/2}},|l-\hat{l}|\ge\frac{(\log N)^{2}}{N^{1/2}}
    \end{array}\right)=1.\label{eq:no_crit}
    \end{equation}
    Since we can pick $\varepsilon_{2}>0$ and then $\varepsilon_{1}>0$
    small enough so that (\ref{eq:to_eps_nbd}), (\ref{eq:minimax_is_crit}),
    (\ref{eq:no_crit}) hold simultaneously the claim follows.
\end{proof}
\bigskip

We can now prove a version of Lemma \ref{lem:max_fluct_quad} for the actual function  $h(y,l,s_{\lambda,u}(l))$ rather than its quadratic expansion.

\begin{proposition}\label{prop:max_fluct_gen}
    For any $h$ satisfying \eqref{eq:YLcompact}-\eqref{eq:Hessynegdef} it holds that
    \begin{equation}\label{eq:max_fluct_gen}  
        \sup_{y\in\mathcal{Y}}\inf_{l\in\mathcal{L}}h(y,l,s_{\lambda,u}(l)) = E_{1}+\frac{1}{\sqrt{N}}E_{2}W_{N}+\frac{1}{N}F_{N}+o_{\mathbb{P}}\left(\frac{1}{N}\right),
    \end{equation}
    for $E_1,E_2,E_3,F_N$ as in \eqref{eq:E1E2E3}-\eqref{eq:FN}. 
\end{proposition}
\begin{proof}
    By Lemma \ref{lem:h_quad_expansion} and Lemma \ref{lem:localize_more}
    \[
        \sup_{y\in\mathcal{Y}}\inf_{l\in\mathcal{L}}h(y,l,s_{\lambda,u}(l))=\sup_{y\in\mathcal{Y}(N^{-1/2}\log N)}\inf_{l\in\mathcal{L}(N^{-1/2}(\log N)^{2})}p_{N}(\Delta_{y},\Delta_{l})+o_{\mathbb{P}}(N^{-1}).
    \]
    Recall from the proof of Lemma \ref{lem:max_fluct_quad} that $\hat{\Delta}_{l}$ from \eqref{eq:hat_delta_l} is the minimizer of
    $\inf_{l\in\mathbb{R}}p_{N}(\Delta_{y},\Delta_{l})$. Note that for
    all $y\in\mathcal{Y}(N^{-1/2}\log N)$ it holds that $\mathbb{P}(|\hat{\Delta}_{l}|\le N^{-1/2}(\log N)^{2})\to1,$ so
    \[
        \lim_{N\to\infty}\mathbb{P}\left(\inf_{l\in\mathcal{L}(N^{-1/2}(\log N)^{2})}p_{N}(\Delta_{y},\Delta_{l})=p_{N}(\Delta_{y},\hat{\Delta}_{l})\right)=1.
    \]
    Similarly recall that $\hat{\Delta}_{y}$ from \eqref{eq:hat_delta_y} is the maximizer
    of $\sup_{y\in\mathbb{R}^{n}}\inf_{l\in\mathbb{R}}p_{N}(\Delta_{y},\Delta_{l})$
    and note that $\mathbb{P}(|\hat{\Delta}_{y}|\le N^{-1/2}\log N)\to1$
    so that furthermore
    \[
        \lim_{N\to\infty}\mathbb{P}\left(\sup_{y\in\mathcal{Y}(N^{-1/2}\log N)}\inf_{l\in\mathcal{L}(N^{-1/2}(\log N)^{2})}p_{N}(\Delta_{y},\Delta_{l})=p_{N}(\hat{\Delta}_{y},\hat{\Delta}_{l})\right)=1.
    \]
    Thus the claim follows from \eqref{eq:max_fluct_quad}.
\end{proof}
\bigskip

The next lemma computes the distributional limit of $(W_{N}(l),W_{N}'(l))$. 

\begin{lemma}\label{lem:A Y joint conv for general w}
    For all $l$ it holds that
    \begin{equation}\label{eq:A Y joint conv}
        (W_{N}(l), W'_{N}(l)) \overset{d}{\to} (U(l),U'(l)),
    \end{equation} 
    where
    $(U(l),U'(l))$
    is a centered Gaussian vector with covariance matrix
    \begin{equation*}
        \Sigma =
        \begin{pmatrix}
        -2s'(l)-2s(l)^2 & 
        -s''(l)-2s(l) s'(l) \\
        -s''(l)-2s(l) s'(l) & 
        -\tfrac{1}{3}s'''(l)-2s'(l)^2
        \end{pmatrix}
        .
    \end{equation*}
\end{lemma}
\begin{proof}
    Define
    \begin{equation*}
        \tilde{W}_{N}^{(k)}(l)
        =
        \frac{1}{\sqrt{N}}\sum_{i=1}^{N}(N\tilde{u}_{i}^{2}-1)
        \left(\frac{k!(-1)^k}{(l-\theta_{i/N})^{k+1}}-\HNk{k}{\theta}{l}\right).
    \end{equation*}
    Then
    \begin{align*}
        \mathbb{E}\left[\tilde{W}^{(k)}_{N}(l)\tilde{W}^{(k')}_{N}(l)\right]
        \numberthis\label{eq: tilde WNk variance}
        &=
        \frac{2}{N}\sum_{i=1}^N 
        \left(\frac{k!(-1)^k}{(l-\theta_{i/N})^{k+1}}-\HNk{k}{\theta}{l}\right)
        \left(\frac{{k'}!(-1)^{k'}}{(l-\theta_{i/N})^{{k'}+1}}-\HNk{{k'}}{\theta}{l}\right)
        \\
        &=
        2\left(\frac{1}{N}\sum_{i=1}^N \frac{k! k'! (-1)^{k+k'}}{(l-\theta_{i/N})^{k+k'+2}}
        -\HNk{k}{\theta}{l} \HNk{k'}{\theta}{l}\right)
        \\ 
        &\stackrel{\eqref{eq: sum class loc int}}{\longrightarrow}
        - 2\frac{k! k'!}{(k+k'+1)!}s^{(k+k'+1)}(l)- 2s^{(k)}(l)s^{(k')}(l).
    \end{align*}
    Note that for all $t=(t_1,t_2)\in\mathbb{R}^2$
    \begin{equation*}
        \mathbb{E}\left[
            t_1 \tilde{W}_{N}(l) + t_2 \tilde{W}'_{N}(l) 
        \right] \longrightarrow 0 ,
    \end{equation*}
    and by \eqref{eq: tilde WNk variance}
    \begin{equation*}
        \mathbb{E}\left[
            \left(t_1 \tilde{W}_{N}(l) + t_2 \tilde{W}'_{N}(l) \right)^2
        \right] \longrightarrow t^T \Sigma t.
    \end{equation*}    
    Therefore 
    \begin{equation*}
        t_1 \tilde{W}_{N}(l) + t_2 \tilde{W}'_{N}(l)
        \stackrel{d}{\longrightarrow}
        t_1 U(l) + t_2 U'(l)
        \sim
        \mathcal{N}\left(0, t^T\Sigma t\right)
    \end{equation*}
    by Lyapunov's CLT 
    (see Lindeberg's theorem \cite[Theorem 7.3.1 and Lyapunov's condition p. 307-309]{probability_theory}; note that $\sum_{i=1}^NE[|(N\tilde{u}^2_i-1) (t_1/(l-\theta_{i/N})-t_2/(l-\theta_{i/N})) |^{3}] = {{O}}(N)$ while ${\rm{Var}}(t_1 \tilde{W}_{N}(l) + t_2 \tilde{W}'_{N}(l) )^{3/2} = (\sum_{i=1}^N 2 (t_1/(l-\theta_{i/N})-t_2/(l-\theta_{i/N}))^2)^{3/2} = {{O}}(N^{3/2})$, so Lyapunov's condition is satisfied).
    \\
    By \eqref{eq: WN in tilde u form} and Slutzky's theorem thus also
    \begin{equation*}
        t_1 {W}_{N}(l) + t_2 {W}'_{N}(l)
        \stackrel{d}{\longrightarrow}
        t_1 U(l) + t_2 U'(l)
        \sim
        \mathcal{N}\left(0, t^T\Sigma t\right)
    \end{equation*}
    for all $t\in\mathbb{R}^3$.
    By the Cramér–Wold theorem \cite[Corollary 6.5]{kallenberg} one obtains
    the joint convergence \eqref{eq:A Y joint conv}.
\end{proof}
\bigskip 
We also compute the distributional limit of $\Lambda_N$.

\bigskip 
\begin{lemma}\label{lem: gaussian error term}
    For any $l > \sqrt{2}$
    $$
        \Lambda_N(l)
        \stackrel{d}{\longrightarrow} 
        \mathcal{N}\left(\frac{l-\sqrt{l^2-2}}{2(l^2-2)}, \frac{1}{(l^2-2)^2} \right)
    $$
    as $N\to\infty$.
\end{lemma}
\begin{proof}
By Lemma \ref{lem: ev fluctuation} the random variable $\Lambda_N$ converges 
in law to a normal distribution with mean
\begin{equation}\label{eq: mean}
            m(w) =  \frac{\frac{1}{l-\sqrt{2}} + \frac{1}{l+\sqrt{2}}}{4} -
                    \frac{1}{2\pi} \int_{-\sqrt{2}}^{\sqrt{2}} 
                     \frac{1}{(l-x)\sqrt{2-x^2}} dx  
\end{equation}
    and variance
\begin{equation}\label{eq: variance}
        v(w) =      \frac{1}{2\pi^2}
                    \int_{-\sqrt{2}}^{\sqrt{2}}\int_{-\sqrt{2}}^{\sqrt{2}}
                    \left(\frac{\frac{1}{l-x} - \frac{1}{l-y}}{x-y}\right)^2
                    \frac{2 - x y}{\sqrt{2-x^2}\sqrt{2-y^2}}
                    dx dy
\end{equation}
    with $w(x) = \tfrac{1}{l-x}$.
    It only remains to compute the integrals in \eqref{eq: mean}-\eqref{eq: variance}.
    
    First, note that for any $k\in\mathbb{N}$ by integration by parts
    \begin{equation}\label{eq: preliminary integration by parts}
        \int_{-\sqrt{2}}^{\sqrt{2}}
            \frac{1}{(l-x)^k} \frac{x}{\pi\sqrt{2-x^2}}dx
        =
        \int_{-\sqrt{2}}^{\sqrt{2}}
            \frac{k}{(l-x)^{k+1}} \mu_{\text{sc}}(dx)
        =
        \frac{(-1)^{k}}{(k-1)!} \INk{k}{l}
    \end{equation}
    and also
    \begin{equation}\label{eq: preliminary k=1 integral}
    \begin{array}{rcl}
            \int_{-\sqrt{2}}^{\sqrt{2}}
            \frac{1}{l-x} \frac{1}{\pi\sqrt{2-x^2}}dx
             & =  &
            \frac{1}{l}
            \int_{-\sqrt{2}}^{\sqrt{2}}
            \left(\frac{x}{l-x} + 1 \right) \frac{1}{\pi} \frac{1}{\sqrt{2-x^2}}dx 
            \\
            & \stackrel{\eqref{eq: preliminary integration by parts}}{=} &
             \frac{1}{l}\left({-\INk{1}{l}} + 1\right)
             \\
            &\stackrel{\eqref{eq: useful integrals}}{=}&
            \frac{1}{\sqrt{l^2-2}}
    \end{array}
    \end{equation}
    as well as
    \begin{equation}\label{eq: preliminary k=2 integral}
    \begin{array}{rcl}
        \int_{-\sqrt{2}}^{\sqrt{2}}
            \frac{1}{(l-x)^2} \frac{1}{\pi\sqrt{2-x^2}}dx
        & = &
        \frac{1}{l}\int_{-\sqrt{2}}^{\sqrt{2}}
            \frac{1}{l-x}\left(\frac{x}{l-x} + 1 \right) \frac{1}{\pi\sqrt{2-x^2}}dx
        \\
        &\stackrel{\eqref{eq: preliminary integration by parts},\eqref{eq: preliminary k=1 integral}}{=} &
        \frac{1}{l}\left(\INk{2}{l} + 
                        \frac{1}{\sqrt{l^2-2}}
                    \right)
        \\
        &\stackrel{\eqref{eq: useful integrals}}{=}&
        \frac{l}{(l^2-2)^{\frac{3}{2}}}
        .
    \end{array}
    \end{equation}
    Therefore the expectation of the limiting distribution is
    \begin{align*}
        &
        m(w)
        \stackrel{\eqref{eq: preliminary k=1 integral}}{=} 
         \frac{l}{2(l^2-2)} - \frac{1}{2 \sqrt{l^2-2}}
         =
        \frac{l - \sqrt{l^2-2}}{2(l^2-2)}.
        \numberthis\label{eq: gaussianlimit expectation}
    \end{align*}
    The variance on the other hand is given by
    \begin{align*}
        &
        \frac{1}{2\pi^2}
        \int_{-\sqrt{2}}^{\sqrt{2}}\int_{-\sqrt{2}}^{\sqrt{2}}
                \frac{1}{(l-x)^2(l-y)^2}
                \frac{2 - x y}{\sqrt{2-x^2}\sqrt{2-y^2}}
                dx dy
        \\
        = & 
        \frac{1}{2\pi}\int_{-\sqrt{2}}^{\sqrt{2}}
        \frac{1}{(l-y)^2 \sqrt{2-y^2}}
        \int_{-\sqrt{2}}^{\sqrt{2}}
        \frac{1}{(l-x)^2}
        \frac{2 - x y}{\pi\sqrt{2-x^2}}
        dx
        dy
        ,
        \numberthis\label{eq: gaussianlimit variance}
    \end{align*}
    where the inner integral is by \eqref{eq: preliminary integration by parts} and \eqref{eq: preliminary k=2 integral}
    \begin{equation}\label{eq: gaussianlimit variance inner integral}
        \frac{2l}{(l^2-2)^{\frac{3}{2}}} - 2 y \INk{2}{l}
        \stackrel{\eqref{eq: useful integrals}}{=}
        \frac{2(l-y)}{(l^2-2)^{\frac{3}{2}}}.
    \end{equation}
    Therefore the variance is
    \begin{align*}
        \frac{1}{\pi}\int_{-\sqrt{2}}^{\sqrt{2}}
        \frac{1}{(l-y)^2 \sqrt{2-y^2}}
        \frac{l-y}{(l^2-2)^{\frac{3}{2}}}
        dy
        = 
        \frac{1}{(l^2-2)^{\frac{3}{2}}} \
        \int_{-\sqrt{2}}^{\sqrt{2}}
        \frac{1}{l-y} \frac{1}{\pi \sqrt{2-y^2}} dy
        \stackrel{\eqref{eq: preliminary k=1 integral}}{=} \frac{1}{(l^2-2)^2}.
    \end{align*}
\end{proof}

\subsection{Derivation of main fluctuation results}

Now we are ready to prove Theorem \ref{thm: sphere} (b) and Theorem \ref{thm: ball} (b).
Before giving the proof, we state the following simplified versions of \eqref{eq: useful integrals} using
\eqref{eq: l hat def}:
\begin{equation}\label{eq: useful integrals 2}
    \INk{k}{\hat{l}(\alpha)} = \left\{ \begin{array}{l@{\quad}lll}
\hat{l}(\alpha)-\sqrt{\hat{l}(\alpha)^{2}-2} & = & \sqrt{2(1-\alpha^{2})} & \text{ for }k=0,\\
-\frac{\hat{l}-\sqrt{\hat{l}(\alpha)^{2}-2}}{\sqrt{\hat{l}(\alpha)^{2}-2}} & = & -\frac{2(1-\alpha^{2})}{\alpha^{2}} & \text{ for }k=1,\\
\frac{2}{(\hat{l}(\alpha)^{2}-2)^{\frac{3}{2}}} & = & \frac{2(2(1-\alpha^{2}))^{\frac{3}{2}}}{\alpha^{6}} & \text{ for }k=2,\\
-\frac{\hat{6l}(\alpha)}{(\hat{l}(\alpha)^{2}-2)^{\frac{5}{2}}} & = & 
-\frac{24(2-\alpha^{2})(1-\alpha^{2})^{2}}{\alpha^{10}} 
& \text{ for }k=3.
\end{array}\right.
\end{equation}
Using this with $\alpha = \hat{\alpha}$
\begin{equation}\label{eq:I(hatl)etc}
\IN{\hat{l}}=\hat{z},\quad\INd{\hat{l}}=-\frac{\hat{z}^{2}}{\hat{\alpha}^{2}},\quad\INdd{\hat{l}}=2\frac{\hat{z}^{3}}{\hat{\alpha}^{6}},
\quad\INddd{\hat{l}}=-\frac{6(2-\alpha^{2}) \hat{z}^4}{\hat{\alpha}^{10}},
\quad \text{ where } \hat{z}=\sqrt{2\left(1-\hat{\alpha}^2\right)}.
\end{equation}    
\bigskip

\begin{proof}[Proof of Theorem \ref{thm: sphere} (b)]
    Applying Lemma \ref{lem: L_N in terms of h} and Proposition \ref{prop:max_fluct_gen} with 
    \begin{equation}
    h(\alpha,l,g)=f(\alpha)+\beta\left(l-\frac{\alpha^{2}}{g}\right),\label{eq:sphere_h}
    \end{equation}
    we obtain
    \begin{equation}\label{eq: apply proposition to LN}
        \frac{1}{N} L_N = 
        E_{1}+\frac{1}{\sqrt{N}}E_{2}W_{N}+\frac{1}{N}F_{N}+o_{\mathbb{P}}\left(\frac{1}{N}\right).
    \end{equation}
    Note that $$U^{(k)}_N - W^{(k)}_N \stackrel{\eqref{eq: U_N-W_N=R_N}}{=} \frac{1}{\sqrt{N}}R_N^{(k)}(l) = o_{\mathbb{P}}(N^{-1/2})$$ by Lemma \ref{lem:s_lambda_u_exp_rest_term}.
    It follows that $(U_N,U_N')$ and $(W_N,W_N')$ have the same limit, and that we can swap all $W_N$ for $U_N$ and $W_N'$ for $U'_N$ in \eqref{eq: apply proposition to LN} at the cost of a negligible error.
    
    The remainder of the proof will revolve around computing $E_{1},E_{2},E_{3},J,L,W,K,G$ of Proposition \ref{prop:max_fluct_gen}. Note
    first that
    \[
    E_{1}=\mathcal{B}(\hat{\alpha})\quad\text{ and }\quad J=\frac{1}{\mathcal{B}''(\hat{\alpha})}.
    \]
    Furthermore for the $h$ in (\ref{eq:sphere_h}) we obtain with $\hat{z}=\sqrt{2(1-\hat{\alpha}^{2})}$
    \begin{equation*}
        E_{2}=h_{g}=\frac{\beta\hat{\alpha}^{2}}{s(\hat{l})^{2}}
        =\frac{\beta\hat{\alpha}^{2}}{\hat{z}^{2}}
        = \kappa
        ,
        \quad
        E_{3}=h_{gg}=-\frac{2\beta\hat{\alpha}^{2}}{s(\hat{l})^{3}}
        = -\frac{2\beta\hat{\alpha}^{2}}{\hat{z}^{3}}
    \end{equation*}
    as well as
    \begin{equation*}
    \begin{array}{rclcrcl}
         h_{l,\alpha}&=& \frac{2\beta\hat{\alpha}s'(\hat{l})}{s(\hat{l})^{2}} = -\frac{2\beta}{\hat{\alpha}},
        &\quad& 
        h_{ll} &= &\beta\hat{\alpha}^{2}\left(\frac{s''(\hat{l})}{s(\hat{l})^2}-\frac{2 s'(\hat{l})^2}{s(\hat{l})^3}\right) 
        = \beta \frac{\hat{z}^{3}}{\hat{\alpha}^{4}},
        \\
        h_{l,g}&=&-\frac{2\beta\hat{\alpha}^{2} s'(\hat{l})}{s(\hat{l})^3} 
        = \frac{2\beta}{\hat{z}},
        &\quad&
        h_{\alpha,g}&=&\frac{2\beta\hat{\alpha}}{s(\hat{l})^2}
        = \frac{2\beta\hat{\alpha}}{\hat{z}^2},
    \end{array}
    \end{equation*}
    which gives
    \begin{equation*}
        L=\left(\begin{matrix}h_{\alpha,g} & 0\end{matrix}\right)
        = \beta\begin{pmatrix}
            \frac{2\hat{\alpha}}{\hat{z}^2} & 0
        \end{pmatrix}
        \in\mathbb{R}^{1\times2},
        \quad
        w=\left(\begin{matrix}h_{l,g}\\
        E_{2}
        \end{matrix}\right)=\left(\begin{matrix}h_{l,g}\\
        h_{g}
        \end{matrix}\right)
        = \frac{\beta}{\hat{z}}\begin{pmatrix}
            2\\
            \frac{\hat{\alpha}^{2}}{\hat{z}}
        \end{pmatrix}
        \in\mathbb{R}^{2},
    \end{equation*}
    \begin{equation*}
        \begin{array}{rcl}
             K&=&L-\frac{h_{l,\alpha}w^{T}}{h_{ll}}
            =
            \beta
            \begin{pmatrix}
            \frac{2\hat{\alpha}}{\hat{z}^2} & 0
            \end{pmatrix}
            +
            \frac{2\beta\hat{\alpha}^3}{\hat{z}^4}
            \begin{pmatrix}
            2&
            \frac{\hat{\alpha}^{2}}{\hat{z}}
            \end{pmatrix}
            =
            \frac{2\beta\hat{\alpha}}{\hat{z}^4}
        \begin{pmatrix}
            2 
            &
            \frac{\hat{\alpha}^{4}}{\hat{z}}
        \end{pmatrix}
        \in\mathbb{R}^{1\times2}
        \end{array}
    \end{equation*}
    \[
    H=\frac{K^{T}K}{\mathcal{B}''(\hat{\alpha})}
    = \frac{8 \beta^2 \hat{\alpha}^2}{\hat{z}^8 \mathcal{B}''(\hat{\alpha})}
    \begin{pmatrix}
            2
            &
            \frac{\hat{\alpha}^4}{\hat{z}}
            \\
            \frac{\hat{\alpha}^4}{\hat{z}}
            &
            \frac{\hat{\alpha}^{8}}{2\hat{z}^2}
    \end{pmatrix}
    \in\mathbb{R}^{2\times2},
    \]
    \[
    G=H-\left(\begin{matrix}E_{3} & 0\\
    0 & 0
    \end{matrix}\right)
    = \frac{8 \beta^2 \hat{\alpha}^2}{\hat{z}^8 \mathcal{B}''(\hat{\alpha})}
    \begin{pmatrix}
            2
            &
            \frac{\hat{\alpha}^4}{\hat{z}}
            \\
            \frac{\hat{\alpha}^4}{\hat{z}}
            &
            \frac{\hat{\alpha}^{8}}{2\hat{z}^2}
    \end{pmatrix}
    +
    \frac{2\beta\hat{\alpha}^{2}}{\hat{z}^{3}}
    \begin{pmatrix}
            1
            &
            0
            \\
            0
            &
            0
    \end{pmatrix}
    ,
    \]
    and finally
    \[
        F_{N}=E_{2}\Lambda_{N} - \frac{1}{2}\left(\begin{matrix}U_{N}\\
        U_{N}^{'}
        \end{matrix}\right)^{T}G\left(\begin{matrix}U_{N}\\
        U_{N}^{'}
        \end{matrix}\right).
    \]
    This proves \eqref{eq: theorem1-equation}.

    The joint convergence in law of $W_{N},W'_{N},\Lambda_N$ follows from Lemma \ref{lem:A Y joint conv for general w}, Lemma \ref{lem: gaussian error term} and since $\Lambda_N$ is independent from $(W_{N},W'_{N})$ for all $N$.
    Note that using \eqref{eq: useful integrals 2} the matrix $\Sigma$ can be simplified to
    \begin{equation*}
        \Sigma =
        \begin{pmatrix}
        \tfrac{4 (1-\hat{\alpha}^2)^2}{\hat{\alpha}^2} & 
        -\tfrac{4 \sqrt{2} \sqrt{1-\hat{\alpha}^2}^5 (1+\hat{\alpha}^2)}{\hat{\alpha}^6}  \\
        -\tfrac{4 \sqrt{2} \sqrt{1-\hat{\alpha}^2}^5 (1+\hat{\alpha}^2)}{\hat{\alpha}^6} & 
        \frac{8(1-\hat{\alpha}^2)^3(2+\hat{\alpha}^2+\hat{\alpha}^4)}{\hat{\alpha}^{10}} 
        \end{pmatrix},
    \end{equation*} 
    while the limiting distribution of $\Lambda_N$ is given by
    \begin{equation}\label{eq: Lambda_N distribution}
        \mathcal{N}\left( \tfrac{\hat{l}(\hat{\alpha})-\sqrt{\hat{l}(\hat{\alpha})^2-2}}{2(\hat{l}(\hat{\alpha})^2-2)}, \tfrac{1}{(\hat{l}(\hat{\alpha})^2-2)^2}  \right)
        \stackrel{\eqref{eq: l hat def}}{=}
        \mathcal{N}\left(2\tfrac{\sqrt{2} (1-\hat{\alpha}^2)^{\frac{3}{2}}}{\hat{\alpha}^4}
            , \tfrac{4(1-\hat{\alpha}^2)^2}{\hat{\alpha}^8}\right)
        .
    \end{equation}
\end{proof}
\bigskip

\begin{proof}[Proof of Theorem \ref{thm: ball} (b)]
As in the previous proof we apply Lemma \ref{lem: L_N in terms of h} and Proposition \ref{prop:max_fluct_gen},
this time with 
\begin{equation}
h((\alpha,r),l,g)=f(\alpha r)+ g(r) + \beta r^{2}\left(l-\frac{\alpha^{2}}{g}\right),\label{eq:h_ball}
\end{equation}
and also use that $U^{(k)}_N- W^{(k)}_N = o_{\mathbb{P}}(N^{-1/2})$ to exchange $W_N,W_N'$ for $U_N, U_N'$, yielding
\begin{align*}
    \frac{1}{N}\tilde{L}_N
    =
    E_1 + \frac{1}{\sqrt{N}} E_2 W_N + \frac{1}{N}F_N + o_{\mathbb{P}}\left(\frac{1}{N}\right).
\end{align*}
Now
\[
E_{1}=\tilde{\mathcal{B}}(\hat{\alpha},\hat{r})\quad\text{ and }\quad J=\left(\nabla^{2}\tilde{\mathcal{B}}((\hat{\alpha},\hat{r}))\right)^{-1},
\]
and for the $h$ in (\ref{eq:h_ball}) with $\hat{z}=\sqrt{2(1-\hat{\alpha}^2)}$
\begin{equation*}
    E_{2}=h_{g}=\beta\frac{\hat{r}^{2}\hat{\alpha}^{2}}{s(\hat{l})^{2}}=\beta\frac{\hat{r}^{2}\hat{\alpha}^{2}}{\hat{z}^{2}}
    = \tilde{\kappa},
    \quad 
    E_{3}=h_{gg}=-2\beta\frac{\hat{r}^{2}\hat{\alpha}^{2}}{s(\hat{l})^{3}}=-2\beta\frac{\hat{r}^{2}\hat{\alpha}^{2}}{\hat{z}^{3}},
\end{equation*}
\[
h_{l,y}=\left(\begin{matrix}2\beta \hat{r}^{2}\frac{\hat{\alpha}}{s(l)^{2}}s'(l)\\
2\beta \hat{r}\left(1+\frac{\hat{\alpha}^{2}}{s(l)^{2}}s'(l)\right)
\end{matrix}\right)=\left(\begin{matrix}2\beta \hat{r}^{2}\frac{\hat{\alpha}}{\hat{z}^{2}}s'(l)\\
2\beta \hat{r}\left(1+\frac{\hat{\alpha}^{2}}{\hat{z}^{2}}s'(l)\right)
\end{matrix}\right)=\left(\begin{matrix}-2\beta \hat{r}^{2}\frac{\hat{\alpha}}{\hat{z}^{2}}\frac{\hat{z}^{2}}{\hat{\alpha}^{2}}\\
2\beta \hat{r}\left(1-\frac{\hat{\alpha}^{2}}{\hat{z}^{2}}\frac{\hat{z}^{2}}{\hat{\alpha}^{2}}\right)
\end{matrix}\right)=\left(\begin{matrix}-\frac{2\beta \hat{r}^{2}}{\hat{\alpha}}\\
0
\end{matrix}\right),
\]
\[
h_{ll}=-2\beta \hat{r}^{2}\frac{\hat{\alpha}^{2}}{s(l)^{3}}s'(l)^{2}+\beta \hat{r}^{2}\frac{\hat{\alpha}^{2}}{s(l)^{2}}s''(l)=
\beta\hat{r}^2 \left(\frac{2\hat{z}}{\hat{\alpha}^{4}}-\frac{2\hat{z}}{\hat{\alpha}^{2}}\right)
=
\beta \frac{\hat{r}^2\hat{z}^3}{\hat{\alpha}^{4}},
\]
\begin{equation*}
    h_{l,g}=-\frac{2\beta\hat{\alpha}^{2}\hat{r}^{2} s'(\hat{l})}{s(\hat{l})^3} = \frac{2\beta\hat{r}^{2}}{\hat{z}}
    , \quad 
    h_{\alpha,g}=\frac{2\beta\hat{\alpha}\hat{r}^{2}}{s(\hat{l})^2}
    =\frac{2\beta\hat{\alpha}\hat{r}^{2}}{\hat{z}^2}
    , \quad 
    h_{r,g}=\frac{2\beta\hat{\alpha}^{2}\hat{r}}{s(\hat{l})^2}
    =\frac{2\beta\hat{\alpha}^{2}\hat{r}}{\hat{z}^2},
\end{equation*}
which gives
\begin{equation*}
    L=\left(\begin{matrix}h_{y,g} & 0\end{matrix}\right)
    = 
    \frac{2\beta\hat{\alpha}\hat{r}}{\hat{z}^2}
    \begin{pmatrix}
        \hat{r}
        & 0 \\
        \hat{\alpha}
        & 0
    \end{pmatrix}
    \in\mathbb{R}^{2\times2},
    \quad 
    w=\left(\begin{matrix}h_{l,g}\\
    h_{g}
    \end{matrix}\right)
    =
    \frac{\beta\hat{r}^2}{\hat{z}}
    \begin{pmatrix}
        2
        \\
        \frac{\hat{\alpha}^2}{\hat{z}}
    \end{pmatrix}
    \in\mathbb{R}^{2}
\end{equation*}
\begin{equation*}
    K=L-\frac{h_{l,y}w^{T}}{h_{ll}}
    =
    L-
    \frac{\hat{\alpha}^{4}}{2\beta\hat{r}^2\hat{z}^3}
    \frac{-2\beta^2\hat{r}^4}{\hat{z}}
    \begin{pmatrix}
        \frac{2}{\hat{\alpha}}
        &
        \frac{\hat{\alpha}}{\hat{z}}
        \\
        0 & 0
    \end{pmatrix}
    = 
    \frac{2\beta\hat{r}\hat{\alpha}}{\hat{z}^2}
    \begin{pmatrix}
        \frac{2\hat{r}}{\hat{z}^2}
        &  \frac{\hat{r}\hat{\alpha}^4}{\hat{z}^3}
        \\ \hat{\alpha}
        & 0
    \end{pmatrix}
    \in\mathbb{R}^{2\times2}.
\end{equation*}
Furthermore we obtain
\begin{equation*}
    H=K^{T}\left(\nabla^{2}\mathcal{B}((\hat{\alpha},\hat{r}))\right)^{-1}K,
    \quad 
    G = 
    H-\left(\begin{matrix}E_{3} & 0\\
    0 & 0
    \end{matrix}\right)
    =
    K^{T}\left(\nabla^{2}\mathcal{B}((\hat{\alpha},\hat{r}))\right)^{-1}K
    +
    \begin{pmatrix}
        2\beta\frac{\hat{r}^{2}\hat{\alpha}^{2}}{\hat{z}^{3}}
        & 0 \\ 0 & 0
    \end{pmatrix},
\end{equation*}
and finally
\[
F_{N}=E_{2}\Lambda_{N}
-\frac{1}{2}\left(\begin{matrix}W_{N}\\
W_{N}^{'}
\end{matrix}\right)^{T}G\left(\begin{matrix}W_{N}\\
W_{N}^{'}
\end{matrix}\right).
\]
This proves \eqref{eq: ball result equation}.
\end{proof}
\bigskip

\begin{remark}\label{rem: alternative theorems}
    Theorem \ref{thm: sphere} (b) and Theorem \ref{thm: ball} (b) were stated in terms of the sums $U_N,U'_N$ over the random vector $u$ with weakly dependent but not independent entries. It may be more natural to write the result instead in terms of sums of truly independent summands. This can be done if one constructs $u$ from i.i.d. $\tilde{u}_1,...,\tilde{u}_N$ as we did in the proofs of Lemma \ref{lem:U_N_UB} and Lemma \ref{lem:A Y joint conv for general w}. If we define
    \begin{align*}
        X_N^{(k)} & =
        \frac{1}{\sqrt{N}}\sum_{i=1}^{N}(N\tilde{u}_{i}^{2}-1)
        \left(\frac{k!(-1)^k}{(\hat{l}-\theta_{i/N})^{k+1}}-\INk{k}{l}\right) \label{eq:X}\numberthis
        \\
        Y_N &= \frac{1}{\sqrt{N}}\sum_{i=1}^{N}(N\tilde{u}_{i}^{2}-1)
    \end{align*}
    one can verify that
    \begin{equation}\label{eq:WX}
        W^{(k)}_N = X^{(k)}_N - \frac{1}{\sqrt{N}} X^{(k)}_N Y_N + o_{\mathbb{P}}(N^{-1/2}).
    \end{equation}
    Theorem \ref{thm: sphere} (b) can then be reformulated as
    \begin{equation}\label{eq: theorem1-equation alt}
        L_N -
            N \mathcal{B}(\hat{\alpha}) 
            - \sqrt{N}\kappa X_N
            - \left( \kappa\Lambda_N - \kappa X_N Y_N 
                -\frac{1}{2}\begin{pmatrix}X_{N}\\
                X_{N}^{'}
                \end{pmatrix}^{T}G\begin{pmatrix}X_{N}\\
                X_{N}^{'}
                \end{pmatrix}
                \right)
                \overset{\mathbb{P}}{\longrightarrow}0,
    \end{equation}
    where the random variables satisfy
    $$
        (X_N, X'_N, Y_{N}, \Lambda_{N})
        \stackrel{d}{\longrightarrow}
        (X,X',Y,\Lambda),
    $$
    where with $\hat{z} = \sqrt{2(1-\hat{\alpha}^2)}$
    \begin{alignat*}{3}
        X \sim \mathcal{N}\left(0, \frac{\hat{z}^4}{\hat{\alpha}^2}\right)
        , \qquad
        & 
        X' \sim \mathcal{N}\left(0, \frac{\hat{z}^6(2+\hat{\alpha}^2+\hat{\alpha}^4)}{\hat{\alpha}^{10}}\right)
        , 
        \numberthis\label{eq: theorem1-variables alt}
        \\
        Y \sim \mathcal{N}\left(0, 2 \right)
        , \qquad
        &
        \Lambda \sim \mathcal{N}\left(\frac{\hat{z}^3}{2\hat{\alpha}^4}
            , \frac{\hat{z}^4}{\hat{\alpha}^8}\right)
        , 
    \end{alignat*}
    with $(X,X')$, $Y$ and $\Lambda$ mutually independent and
    \begin{align*}
        {\rm{Cov}}(X,X') & = 
            -\frac{\hat{z}^5 (1+\hat{\alpha}^2)}{\hat{\alpha}^6}.
    \end{align*}
    The constant $\kappa$ and matrix $G$ are the same as before. Comparing the estimate \eqref{eq: theorem1-equation} in terms of $U_N,U_N'$ and \eqref{eq: theorem1-equation alt} one sees that the extra term $\kappa X_N Y_N $ of order one appears, which arises from the $N^{-1/2}$ correction in \eqref{eq:WX}. Note furthermore that \eqref{eq: theorem1-equation alt} would remain true if one defined $X^{(k)}_N$ with the random eigenvalues $\lambda_i$ instead of deterministic classical locations $\theta_{i/N}$ in \eqref{eq:X}.
    
    Similarly Theorem \ref{thm: ball} (b) can be formulated as
    \begin{equation}\label{eq: ball result equation alt}
        \tilde{L}_N -
        \tilde{\mathcal{B}}(\hat{\alpha},\hat{r}) 
        -
        \sqrt{N} \kappa X_N
        - \left( \kappa \Lambda_N - \kappa X_N Y_N 
        - \frac{1}{2} 
            \begin{pmatrix}X_{N}\\
                X_{N}^{'}
                \end{pmatrix}^{T}\tilde{G}\begin{pmatrix}X_{N}\\
                X_{N}^{'}
                \end{pmatrix}
        \right)
       \overset{\mathbb{P}}{\longrightarrow} 0,
    \end{equation} 
    where $X_N, X_N, Y_N, \Lambda_N$ are as in \eqref{eq: theorem1-equation alt}.
\end{remark}

\section{Examples: Subleading order}\label{section: example fluctuations}
We showed in Section \ref{section: example leading order} that for $f(x) = h x^k$ and $\beta < \beta_c(k,h)$ (see \eqref{eq: beta critical}) the function $\mathcal{B}(\alpha)$ has a unique maximizer in $[0,1]$. Theorem \ref{thm: sphere} (b) requires also that $\mathcal{B}''(\hat{\alpha})<0$, which the next lemma shows is always satisfied.

\begin{lemma}\label{lem: monomial B double prime}
    Let $h\in\mathbb{R}^+$, $k\in\mathbb{N}$, $f(x) = h x^k$. If $\beta < \beta_c(k,h)$ then all global maximizers
    $\hat{\alpha} \in \argmax_{\alpha\in(-1,1)} \mathcal{B}(\alpha)$
    satisfy $\mathcal{B}''(\hat{\alpha}) < 0$.
\end{lemma}
\begin{proof}
    By Lemma \ref{lem: monomial maxima} there is a unique maximizer $\hat{\alpha} \in (0,1)$ for $\beta < \beta_c(k,h)$. Note that we must have  $\mathcal{B}''(\hat{\alpha}) \le 0$, so we only have to prove that $\mathcal{B}''(\hat{\alpha}) \neq 0$.
    In the case $k=1$ we have $\mathcal{B}''(\alpha) = -\frac{\sqrt{2}\beta}{(1-\alpha^2)^{\frac{3}{2}}} < 0$ for all $\alpha \in (-1,1)$. In the case $k=2$ we have by Lemma \ref{lem: monomial maxima} that $\hat{\alpha}^2 = 1-\frac{\beta^2}{2h^2}$ and thus
    $$
        \mathcal{B}''(\hat{\alpha}) 
        = 2h - \frac{\sqrt{2}\beta}{\left(1-\left(1-\frac{\beta^2}{2h^2}\right)\right)^{\frac{3}{2}}}
        = \frac{2h}{\beta^2}(\beta^2 - 2h^2) < 0
    $$
   for all $\beta < \beta_c(2) = \sqrt{2} h$. In the case $k\ge 3$ 
    note that for any critical $\alpha \in (0,1)$
   \begin{equation}
       \mathcal{B}''(\alpha) 
       \stackrel{\eqref{eq: alpha eq}}{=}
       (k-1) \frac{\sqrt{2}\beta}{\sqrt{1-\alpha^2}} - \frac{\sqrt{2}\beta}{(1-\alpha^2)^{\frac{3}{2}}}
       =
       \frac{\sqrt{2}\beta}{\sqrt{1-\alpha^2}} \left(k-1 - \frac{1}{1-\alpha^2}\right),
   \end{equation}
   which can only be equal to zero if $\alpha^2 = \frac{k-2}{k-1}$. Thus, it remains to show that $\alpha = \sqrt{\frac{k-2}{k-1}}$ is not the global maximizer of $\mathcal{B}$. 
   Now suppose we have
   \begin{equation}
       \mathcal{B}'\left(\sqrt{\tfrac{k-2}{k-1}}\right) = \mathcal{B}''\left(\sqrt{\tfrac{k-2}{k-1}}\right) = 0,
   \end{equation}
   and note that for any critical $\alpha \in (0,1)$ the third derivative is
   \begin{equation}
   \begin{array}{rcl}
       \mathcal{B}'''(\alpha) 
       &=& 
       h k (k-1) (k-2) \alpha^{k-3} - 3 \sqrt{2} \beta \frac{\alpha}{(1-\alpha^2)^{\frac{5}{2}}}
       \\
       &\stackrel{\eqref{eq: alpha eq}}{=}&
       \frac{\sqrt{2}\beta \alpha}{\sqrt{1-\alpha^2}}
       \left(\frac{(k-1)(k-2)}{\alpha^2} - \frac{3}{(1-\alpha^2)^2}\right).
    \end{array}   
   \end{equation}
   Then for $\alpha = \sqrt{\tfrac{k-2}{k-1}}$
   \begin{equation}
       \mathcal{B}'''\left(\sqrt{\tfrac{k-2}{k-1}}\right) = 
       \sqrt{2(k-2)}\beta
       \left((k-1)^2 - 3 (k-1)^2\right) = - 2 \sqrt{2(k-2)}\beta(k-1)^2  < 0,
   \end{equation}
   which means that a critical $\alpha = \sqrt{\tfrac{k-2}{k-1}}$ is a saddle point and not a  maximizer.
\end{proof}

\bigskip
From Lemma \ref{lem: monomial maxima} and Lemma \ref{lem: monomial B double prime} it follows that one can apply Theorem \ref{thm: sphere} (b) for all $f(x) = h x^k$ whenever $\beta < \beta_c(k)$. In the linear and quadratic case one can obtain the following more explicit results.

\bigskip

\begin{corollary}\label{cor: linear fluctutation}
    Let $h \in \mathbb{R}\setminus\{0\}$ and $f(x) = h x$. 
    Then
    $$
        L_N 
        -
            N \sqrt{h^2 + 2\beta^2}
            - \sqrt{N}\kappa U_N
            - \left( \kappa\Lambda_N 
                -\frac{1}{2}\begin{pmatrix}U_{N}\\
                U_{N}^{'}
                \end{pmatrix}^{T}G\begin{pmatrix}U_{N}\\
                U_{N}^{'}
                \end{pmatrix}
                \right)
        \stackrel{\mathbb{P}}{\longrightarrow} 0
    $$
    with $\kappa$ and $G$ given by
    \begin{align*}
        \kappa & = \frac{h^2}{4\beta}, &
        G_{11} & = -\frac{h^4\sqrt{h^2+2\beta^2}}{8\beta^4}
        , &
        \\
        G_{12} = G_{21} & = -\frac{h^6}{32\beta^5}
        , &
        G_{22} & = \frac{h^{10}}{2^7 \beta^6 (h^2+2\beta^2)^{\frac{3}{2}}}
        , &
    \end{align*}
    and the joint convergence
     \begin{align*}
        U_{N} & \rightarrow U \sim \mathcal{N}\left(   0 ,  \tfrac{16 \beta^4}{h^2 (h^2 + 2\beta^2)}  \right), & 
        U'_{N} & \rightarrow U' \sim \mathcal{N}\left(  0  ,  \tfrac{2^7 \beta^6 \left(4\beta^4+5\beta^2 h^2 +                                                                    2h^4\right)}{h^{10}}  \right), &
        \\& &
        \Lambda_N & \rightarrow \Lambda \sim \mathcal{N}\left(  \tfrac{4\beta^3 \sqrt{h^2+2\beta^2}}{h^4}  ,  \tfrac{16\beta^4 (h^2 + 2\beta^2)^2}{h^8}  \right), &
    \end{align*}
    in distribution, where $(U,U')$ and $\Lambda$ are independent and
    \begin{equation}
        \text{Cov}(U,U') = -\frac{2^6 \beta^5 (h^2 + \beta^2)}{h^6 \sqrt{h^2+2\beta^2}}.
    \end{equation}
\end{corollary}
\bigskip

\begin{remark}\label{rem:CS17_remark}
    Note that it follows from Corollary \ref{cor: linear fluctutation} that
    \begin{equation}
        \frac{1}{\sqrt{N}}\left(L_N - \sqrt{h^2+2\beta^2} N\right) \longrightarrow \mathcal{N}\left(0, \frac{\beta^2 h^2}{h^2+2\beta^2}\right),
    \end{equation}
    which coincides with the results from \cite{chen2017parisi}.
        To see this let $\gamma_2 = \beta$, $\gamma_p = 0$ for $p>2$ and 
        \begin{equation}
            \xi(s) = 
                    \beta^2 s^2
        \end{equation}
        in \cite[Theorem 5]{chen2017parisi}.
        By \cite[Proposition 1]{chen2017parisi} we then have
        \begin{equation}
            L_0 = \frac{1}{\sqrt{\xi'(1) + h^2}} = \frac{1}{\sqrt{2 \beta^2 + h^2}},
        \end{equation}
        and by \cite[Theorem 3]{chen2017parisi} the function $u_t : (0,1)\to\mathbb{R}$ is the solution of
        \begin{equation}
            L_0^2 (t \xi'(u_t) + h^2) = u_t
            \quad \Leftrightarrow \quad 
            u_t (2\beta^2 + h^2) = t 2 \beta^2 u_t + h^2,
        \end{equation}
        which is
        \begin{equation}
            u_t = \frac{h^2}{2 \beta^2 (1-t) + h^2}.
        \end{equation}
        This in turn gives us by \cite[Theorem 5]{chen2017parisi} that
        \begin{equation}
            \sqrt{N} \left(L_N - \sqrt{2\beta^2 + h^2}\right)
            \longrightarrow 
            \mathcal{N}_N(0, \chi)
        \end{equation}
        with
        \begin{equation}
            \chi = \int_0^1 \xi(u_t) dt
                 = \int_0^1 2 \beta^2 \left(\frac{h^2}{2 \beta^2 (1-t) + h^2}\right)^2 dt
                 = \frac{\beta^2 h^2}{2\beta^2 + h^2}.
        \end{equation}
\end{remark}
\bigskip

\begin{corollary}\label{cor: quadratic fluctutation}
    Let $h \in \mathbb{R}^+$ and $f(x) = h x^2$. If $\beta < \frac{h}{\sqrt{2}}$ then 
    $$
        L_N - 
        \frac{2h^2 + \beta^2}{2 h}
        - \sqrt{N}{\kappa} U_N
            - \left( {\kappa}\Lambda_N 
                -\frac{1}{2}\begin{pmatrix}U_{N}\\
                U_{N}^{'}
                \end{pmatrix}^{T}{G}\begin{pmatrix}U_{N}\\
                U_{N}^{'}
                \end{pmatrix}
                \right)
        \stackrel{\mathbb{P}}{\longrightarrow} 0
    $$
    with constants
    \begin{align*}
        \tilde{\kappa} & = 
                    \frac{2h^2-\beta^2}{2\beta}, &
        G_{11} & = -\frac{h(4h^4-2h^2\beta^2+\beta^4)}{\beta^4}
        , &
        \\
        G_{12} = G_{21} & = -\frac{h^2(2h^2-\beta^2)}{2\beta^5}
        , &
        G_{22} & = -\frac{(2h^2-\beta^2)^4}{16 h \beta^6}
        , &
    \end{align*}
    and the joint convergence in law
     \begin{align*}
        U_N & \rightarrow U \sim \mathcal{N}\left(   0 , \tfrac{2\beta^4}{h^2(2h^2 - \beta^2)}  \right), & 
        U'_N & \rightarrow U' \sim \mathcal{N}\left(  0  ,  \tfrac{8\beta^6 (16h^4 - 6\beta^2h^2 + \beta^4)}{(2h^2-\beta^2)^5}  \right), &
        \\ & &
        \Lambda_N & \rightarrow \Lambda \sim \mathcal{N}\left(  \tfrac{2 h \beta^3}{(2h^2-\beta^2)^2}  ,  \tfrac{16 h^4 \beta^4}{(2h^2-\beta^2)^4}  \right), &
    \end{align*}
    where $(U,U')$ and $\Lambda$ are independent and
    \begin{equation}
        \text{Cov}(U,U') = -\frac{4 \beta^5 (4h^2 - \beta^2)}{h (2h^2-\beta^2)^3} .
    \end{equation}
\end{corollary}
\bigskip
Note that it follows from Corollary \ref{cor: quadratic fluctutation} that
\begin{equation}
    \frac{1}{\sqrt{N}}\left(L_N - \frac{2h^2 + \beta^2}{2 h} N\right) \longrightarrow \mathcal{N}\left(0, \frac{\beta^2(2h^2-\beta^2)}{2h^2}\right).
\end{equation}

\bigskip 
Recall
\begin{equation}
    \tilde{\mathcal{B}}(\alpha,r) = f(r \alpha) + \sqrt{2}\beta r^2 \sqrt{1-\alpha^2}.
\end{equation}
The next lemma will show that the remaining requirements for Theorem \ref{thm: ball} (b) are also satisfied for monomial $f$ with $h > h_c(k,\beta)$.
\bigskip

\begin{lemma}\label{lem: second derivative on ball}
    Let $k \in \mathbb{N}$, $\beta>0$, $h > h_c(k,\beta)$.
    Then there is a unique maximizer $(\hat{\alpha},\hat{r}) \in (-1,1)\times\rm{Plef}(\beta)^{\mathrm{o}}$ of 
        $\tilde{\mathcal{B}}(\alpha,r) + g(r)$
    and
    \begin{equation}
        \nabla^2
        \tilde{\mathcal{B}}(\hat{\alpha},\hat{r})
        \quad\text{ is negative definite.}
    \end{equation}
\end{lemma}
\begin{proof}
    We know from Lemma \ref{lem: k=1 maximum on ball}, Lemma \ref{lem: k=2 maximum on ball} and Lemma \ref{lem: k>=3 main lemma} that there is a unique $(\hat{\alpha},\hat{r}) \in (0,1)\times\rm{Plef}(\beta)^{\mathrm{o}}$.
    We will show the negative definiteness of $\nabla^2 \tilde{\mathcal{B}}(\hat{\alpha},\hat{r})$ by showing that the determinant is positive while the trace is negative.
    \\
    \textbf{Trace:} Let us first look at
    $\nabla^2\tilde{\mathcal{B}}(\hat{\alpha},\hat{r}) = \tilde{\mathcal{B}}''(\hat{\alpha},\hat{r})$.
    Since $(\hat{\alpha},\hat{r})$ is a maximizer it must hold $\tilde{\mathcal{B}}''(\hat{\alpha},\hat{r})\le 0$.
    We have
    \begin{equation*}
        \partial_{\alpha\alpha}\tilde{\mathcal{B}}(\alpha,r)
        =
        r^2 f''(r \alpha) - \frac{\sqrt{2}\beta r^2}{(1-\alpha^2)^{\frac{3}{2}}},
    \end{equation*}
    which is negative for $k=1$ for all $(\alpha,r)$, while for $k \ge 2$ the critical point equation implies
    \begin{equation}\label{eq: critical point equation}
        h k (\hat{r}\hat{\alpha})^{k-2} = \frac{\sqrt{2}\beta}{\sqrt{1-\hat{\alpha}^2}}
    \end{equation}
    and thus 
    \begin{equation*}
    \begin{array}{rcl}
        \partial_{\alpha\alpha} \tilde{\mathcal{B}}(\hat{\alpha},\hat{r})
        & = &
        h k (k-1) \hat{r}^k \hat{\alpha}^{k-2} - \frac{\sqrt{2}\beta \hat{r}^2}{(1-\hat{\alpha}^2)^{\frac{3}{2}}}
        \\
        & \stackrel{\eqref{eq: critical point equation}}{=} &
        (k-1) \hat{r}^2 \frac{\sqrt{2}\beta}{\sqrt{1-\hat{\alpha}^2}}
        - \frac{\sqrt{2}\beta \hat{r}^2}{(1-\hat{\alpha}^2)^{\frac{3}{2}}}
        \\
        &=&
        \frac{\sqrt{2}\beta \hat{r}^2}{(1-\hat{\alpha}^2)^{\frac{3}{2}}}
        \left((k-1) (1-\hat{\alpha}^2) - 1\right).
    \end{array}
    \end{equation*}
    This is obviously negative for $k=2$, while for $k=3$ it can be zero if 
    $\hat{\alpha}^2 = \frac{k-2}{k-1}$. But if we had $\hat{\alpha}^2 = \frac{k-2}{k-1}$ we would obtain
    \begin{equation*}
    \begin{array}{rcl}
         \tilde{\mathcal{B}}'''(\hat{\alpha},\hat{r})
         &=&
         h k (k-1) (k-2) \hat{r}^k \hat{\alpha}^{k-3} - \frac{3\sqrt{2}\beta \hat{r}^2\hat{\alpha}}{(1-\hat{\alpha}^2)^{\frac{5}{2}}}
         \\
         &\stackrel{\eqref{eq: critical point equation}}{=}&
         (k-1)(k-2) \hat{r}^2\hat{\alpha}^{-1} \frac{\sqrt{2}\beta}{\sqrt{1-\hat{\alpha}^2}}
         - \frac{3\sqrt{2}\beta \hat{r}^2\hat{\alpha}}{(1-\hat{\alpha}^2)^{\frac{5}{2}}}
         \\
         &=&
         \frac{\sqrt{2}\beta \hat{r}^2}{\hat{\alpha}(1-\hat{\alpha}^2)^{\frac{5}{2}}}
         \left(
            (k-1)(k-2)(1-\hat{\alpha}^2)^2 - 3\hat{\alpha}^2
         \right)
         \\
         &=&
         - \frac{2(k-2)}{k-1} < 0,
    \end{array}
    \end{equation*}
    which would make this a saddle point and not a maximum, so it must hold that $\partial_{\alpha\alpha}\tilde{\mathcal{B}}(\alpha,r) < 0$. Since we also have $\partial_{r r}\tilde{\mathcal{B}}(\alpha,r) \le 0$ the trace is negative.
    \\
    \textbf{Determinant:} 
    The Hessian of $\tilde{\mathcal{B}}$ is given by
    \begin{equation*}
        \nabla^2\tilde{\mathcal{B}}(\alpha,r)
        =
        \begin{pmatrix}
            h k (k-1) r^{k}\alpha^{k-2} - \sqrt{2}\beta\frac{r^2}{(1-\alpha^2)^{\frac{3}{2}}}
            & 
            h k^2 (r \alpha)^{k-1} - 2\sqrt{2}\beta \frac{r\alpha}{\sqrt{1-\alpha^2}}
            \\
            h k^2 (r \alpha)^{k-1} - 2\sqrt{2}\beta \frac{r\alpha}{\sqrt{1-\alpha^2}}
            & 
            h k (k-1) r^{k-2}\alpha^k + 2\sqrt{2}\beta\sqrt{1-\alpha^2} + g''(r)
        \end{pmatrix}
        .
    \end{equation*}
    Using \eqref{eq: critical point equation} it follows that if $\alpha, r$ are critical points that
    \begin{equation*}
    \begin{array}{rcl}
        \nabla^2\tilde{\mathcal{B}}(\alpha,r)
        &=&
        \begin{pmatrix}
            (k-1)r^2 \frac{\sqrt{2}\beta}{\sqrt{1-\alpha^2}} - \sqrt{2}\beta\frac{r^2}{(1-\alpha^2)^{\frac{3}{2}}}
            & 
            k (r \alpha) \frac{\sqrt{2}\beta}{\sqrt{1-\alpha^2}} - 2\sqrt{2}\beta \frac{r\alpha}{\sqrt{1-\alpha^2}}
            \\
            k (r \alpha) \frac{\sqrt{2}\beta}{\sqrt{1-\alpha^2}} - 2\sqrt{2}\beta \frac{r\alpha}{\sqrt{1-\alpha^2}}
            & 
           (k-1) \alpha^2 \frac{\sqrt{2}\beta}{\sqrt{1-\alpha^2}} + 2\sqrt{2}\beta\sqrt{1-\alpha^2} + g''(r)
        \end{pmatrix},
        \\
        &=&
        \frac{\sqrt{2}\beta}{\sqrt{1-\alpha^2}}
        \begin{pmatrix}
            r^2 \left((k-1) - \frac{1}{1-\alpha^2}\right)
            &
            (k-2) r \alpha
            \\
            (k-2) r \alpha
            &
            (k-1) \alpha^2 + 2(1-\alpha^2) + \frac{g''(r)\sqrt{1-\alpha^2}}{\sqrt{2}\beta}
        \end{pmatrix}.
    \end{array}
    \end{equation*}
    For $k=2$
    \begin{equation*}
    \begin{array}{rcl}
        \det \nabla^2 \tilde{\mathcal{B}}(\alpha,r)
        &=&
        \frac{2\beta^2}{1-\alpha^2}
        \left(
            \left(\alpha^2 + 2(1-\alpha^2) + \frac{g''(r)\sqrt{1-\alpha^2}}{\sqrt{2}\beta}\right)
            \left(
            r^2 \left(1 - \frac{1}{1-\alpha^2}\right)
            \right)
        \right)
    \end{array}
    \end{equation*}
    Since $\beta < \sqrt{2}h$ and we have $\hat{\alpha}^2 = 1-\frac{\beta^2}{2h^2}$ (by Lemma \ref{lem: k=2 maximum on ball}) it holds that
    $$
            r^2 \left(1 - \frac{1}{1-\alpha^2}\right)
             = r^2 \left(1-\frac{2h^2}{\beta^2}\right)< 0,
    $$
    and since also $h > \frac{1}{2}$ as well as $\hat{r}^2 = 1-\frac{1}{2h}$ 
    $$
        \alpha^2 + 2(1-\alpha^2) + \frac{g''(r)\sqrt{1-\alpha^2}}{\sqrt{2}\beta}
        =
        \frac{(2h - 1)(\beta^2 - 2h^2)}{h^2} < 0.
    $$
    Therefore $\det \nabla^2 \tilde{\mathcal{B}}(\alpha,r) > 0$.
    \\
    For $k\ge 3$ using \eqref{eq: k>=3 alpha solution} we can write the Hessian of $\tilde{\mathcal{B}}$ at critical points $(\alpha,r)$ as
    \begin{equation}\label{eq: determinant hessian}
    \begin{array}{rcl}
        \nabla^2\tilde{\mathcal{B}}(\alpha,r)
        &=&
        \frac{1}{1-r^2}
        \begin{pmatrix}
        r^2 \left(k-1-\frac{1}{2\beta^2(1-r^2)^2}\right)
        & 
        (k-2) r \sqrt{1-2\beta^2(1-r^2)^2}
        \\
        (k-2) r \sqrt{1-2\beta^2(1-r^2)^2}
        & 
        \frac{(k(1-r^2)-2)(1-2\beta^2(1-r^2)^2)}{1-r^2}
        \end{pmatrix},
    \end{array}
    \end{equation}
    where the determinant is given by
    \begin{equation}\label{eq: determinant k>=3}
    \begin{array}{rcl}
        \det\left(\nabla^2\tilde{\mathcal{B}}(\alpha,r)\right)
        &=&
        \frac{1}{(1-r^2)^2}
        \frac{-r^2(1-2\beta^2(1-r^2)^2(-2 + k(1-r^2) + 2\beta^2(1-r^2)^2(2-4r^2+k(3r^2-1))))}{2\beta^2(1-r^2)^3}
        \\
        &=&
        \underbrace{-\tfrac{r^2(1-2\beta^2(1-r^2)^2)}{2\beta^2(1-r^2)^5}}_{< 0} \zeta_{k,\beta}(r^2),
    \end{array}
    \end{equation}
    where
    $$
        \zeta_{k,\beta}(q) = -2 + k (1-q) + 2 \beta^2 (1-q)^2 (2-4q + k(3q-1)).
    $$
    Recall $T(q)$ from \eqref{def: T(q)}, which is a non-negative function with $T(q_P)=T(1)=0$. We showed in Lemma \ref{lem: k>=3 main lemma} that $T(q)$ has exactly one critical point, and that $T(q) = \frac{1}{hk}$ has two solutions $q_1 < q_2$, where $\hat{q_2} = \hat{r}^2$. Thus we have $T'(q_1) > 0$ and $T'(q_2) = T'(\hat{r}^2) < 0$. Since
    $$
        T'(q) = 
        \underbrace{\frac{\left((1-2b^2(1-q)^2)q\right)^{\frac{k-4}{2}}}{2}}_{ > 0}
        \underbrace{
        \left(
            -2 + k (1-q) + 2 \beta^2 (1-q)^2 (2-4q + k(3q-1))
        \right)
        }_{=\zeta_{k,\beta}(q)},
    $$
    this must mean that $\zeta_{k,\beta}(\hat{r}^2) < 0$ and thus
    $$
        \det\left(\nabla^2\tilde{\mathcal{B}}(\hat{\alpha},\hat{r})\right) \stackrel{\eqref{eq: determinant k>=3}}{>} 0.
    $$
    For $k=1$ let us substitute $q$ for $r^2$, i.e. instead of $\tilde{\mathcal{B}}$ consider
    $$
        \mathscr{B}(\alpha,q) = h \sqrt{q}\alpha + \sqrt{2}\beta q \sqrt{1-\alpha^2} 
                            + \frac{\beta^2}{2}(1-q)^2 + \frac{1}{2}\log(1-q),
    $$
    where the Hessian is 
    \begin{equation}\label{eq: determinant hessian in q}
    \begin{array}{rcl}
        \nabla^2\mathscr{B}(\alpha,q)
        &=&
        \begin{pmatrix}
        -\frac{\sqrt{2}\beta q}{(1-\alpha^2)^{\frac{3}{2}}}
        & 
        -\frac{\sqrt{2}\beta\alpha}{\sqrt{1-\alpha^2}} + \frac{h}{2\sqrt{q}}
        \\
        -\frac{\sqrt{2}\beta\alpha}{\sqrt{1-\alpha^2}} + \frac{h}{2\sqrt{q}}
        & 
        \beta^2 - \frac{1}{2(1-q)^2} - \frac{h \alpha}{4 q^{\frac{3}{2}}}
        \end{pmatrix}.
    \end{array}
    \end{equation}
    Since for fixed $q$ the maximizing $\alpha(q)$ is $\frac{h}{\sqrt{h^2+2\beta^2q}}$ the determinant of $\nabla^2\mathscr{L}(q,\alpha(q))$ at the maximizer is given by
    $$
        \det\nabla^2\mathscr{B}(\alpha(q),q)
        =
        \frac{
            2\sqrt{q} (h^2+2\beta^2 q)^{\frac{3}{2}}
            \left(
                \frac{1}{2(1-q)^2}-\beta^2
            \right)
                + \frac{h^4}{2q}
        }{4 q \beta^2}.
    $$
    Since $q \ge 1 - \frac{1}{\sqrt{2}\beta}$ we have that $\det\nabla^2\mathscr{B}(q,\alpha(q)) > 0$, and therefore $\det\nabla^2\tilde{\mathcal{B}}(\hat{\alpha},\hat{r}) > 0$.
\end{proof}

\bigskip 
Lemmas \ref{lem: k=1 maximum on ball} - \ref{lem: k>=3 main lemma} together with Lemma \ref{lem: second derivative on ball} show that we can apply Theorem \ref{thm: ball} (b) for monomials $f(x) = h x^k$ and $\beta > 0$ whenever $h > h_c(k,\beta)$.
\bigskip

\medskip
\printbibliography
\end{document}